\documentclass[12pt, reqno,amsmath,amsthm,amssymb,amscd]{amsart}
\usepackage{mathrsfs,amssymb, amscd,amsmath,amsthm}
\usepackage[enableskew,vcentermath]{youngtab}
\usepackage{multicol}\multicolsep=0pt
\usepackage{amsthm}
\usepackage{amssymb}
\usepackage{amsmath}

\def\SUM#1#2{{\mbox{$\sum\limits_{#1}^{#2}$}}}
 \providecommand{\og}{``}
\providecommand{\fg}{''} \providecommand{\smfandname}{and}

\def\1{\hbox{1\kern-.35em\hbox{1}}}

\newtheorem{theorem}{Theorem}[section]
\newtheorem*{theorem*}{Theorem}
\newtheorem{lemma}[theorem]{Lemma}
\newtheorem{proposition}[theorem]{Proposition}
\newtheorem*{proposition*}{Proposition}
\newtheorem{corollary}[theorem]{Corollary}

\theoremstyle{definition}
\newtheorem{notation}[theorem]{Notation}%\def\Notation{Notation\ }
\newtheorem{definition}[theorem]{Definition}
\newtheorem{example}[theorem]{Example}

\theoremstyle{remark}
\newtheorem{remark}[theorem]{Remark}

\numberwithin{equation}{section}

\newcommand{\bea}{\begin{eqnarray}}
\newcommand{\eea}{\end{eqnarray}}
\newcommand{\be}{\begin{eqnarray*}}
\newcommand{\ee}{\end{eqnarray*}}

\newcommand{\Z}{{\mathbb Z}}

\newcommand{\C}{{\mathbb C}}

\newcommand{\mfg}{{\mathfrak g}}
\newcommand{\fh}{{\mathfrak h}}

\newcommand{\Hom}{{\rm Hom}}

\newcommand{\U}{{\rm U}}

\def\subcase#1{\vskip3pt%\medskip
\noindent\textbf{Case #1.}\leavevmode%\newline
}

\def\a{\alpha}

\def\b{\beta}
\def\d{\delta}
\def\D{\Delta}

\def\G{\Gamma}
\def\l{\lambda}

\def\si{\sigma}

\def\sc{\scriptstyle}
\def\ssc{\scriptscriptstyle}
\def\dis{\displaystyle}

\def\D{\Delta}

\def\bs{\backslash}

\def\Z{\mathbb{Z}}

\def\C{\mathbb{C}}

\def\F{\mathbb{F}}

\def\es{\varepsilon}

\def\Hom{{\rm Hom}}

\def\OT#1#2#3{\Set(\ot_{#2})}

\def\tau{\pi}

\def\OTIMES{{{\sc\!}\otimes{\sc\!}}}
\def\a{\alpha}\def\b{\beta}

\def\B{\mathscr B}

\def\DBr{{\mathscr  B_{r,t}^{p,q}(m,n)}}

\numberwithin{equation}{section}
\title[Affine walled Brauer   algebras ]{Affine walled Brauer  \vspace*{-8pt}algebras}

\author{Hebing Rui  {\normalfont \smfandname} Yucai \vspace*{-8pt}Su}
\address{H.~Rui: Department of Mathematics,  East China Normal
University, Shanghai, 200241, China} \email{hbrui@math.ecnu.edu.\vspace*{-5pt}cn}
\address{Y.~Su: Department of Mathematics, Tongji University,  Shanghai, 200092, China} \email{ycsu@tongji.edu.cn}
\thanks{Supported by NSFC (grant no.~11025104, 10825101), Shanghai Municipal Science and Technology Commission
~11XD1402200, 12XD1405000, and Fundamental Research Funds for the Central Universities of China.}
\begin{document}
\baselineskip15.5pt
\begin{abstract} A new class of  associative algebras  referred to as affine  walled Brauer \mbox{algebras} are
introduced.~These algebras are free with infinite rank over a commutative ring containing $1$.~Then level two walled
 Brauer \mbox{algebras} over $\mathbb C$ are defined, which are some cyclotomic quotients
 of  affine walled Brauer \mbox{algebras}.~We establish a super Schur-Weyl duality between affine  walled Brauer algebras and
 general linear Lie superalgebras, and \mbox{realize}
 level two walled
 Brauer algebras as endomorphism algebras of tensor modules of  Kac modules with mixed tensor products of the natural module and its dual over
general linear Lie superalgebras, under some conditions. % ${\mathfrak {gl}}_{m|n}$.~
We also prove the weakly  cellularity of level two walled
 Brauer algebras, and give a classification of  their irreducible modules  over
   $\mathbb C$.
%We conjecture that the decomposition matrices
% of level two walled Brauer algebras can be determined by the structures of certain indecomposable tilting  $\mfg$-modules.
% Such indecomposable tilting modules are  indecomposable direct summands of  $M^{rt}$ in (\ref{M-st==}),
This in turn enables us  to classify the  indecomposable direct summands of the said tensor modules.\end{abstract}
\sloppy \maketitle\vspace*{-30pt}

%\tableofcontents\newpage
\section{Introduction}
Schur-Weyl  reciprocities set up close relationship between polynomial representations of
general linear groups ${GL}_n$ over $\mathbb C$  and representations of  symmetric groups $\mathfrak S_r$~\cite{Gr}.~Such results have been generalized in various cases.  Brauer \cite{B} studied similar problems for symplectic groups and orthogonal
groups.~%over $\mathbb C$.~
A class of associative algebras $\mathscr B_{r}(\delta)$, called Brauer algebras, came into the picture,
%appeared in Brauer's work
which  play the same important role %similar
as that of
symmetric groups in Schur's work.

Walled Brauer algebras or rational Brauer algebras $\mathscr B_{r, t}(\delta)$ (cf.~Definition \ref{Wall-defi}) are subalgebras of Brauer algebras $\mathscr B_{r+t}(\delta)$. They first appeared independently in  Koike's work \cite{Koi} and Turaev's work \cite{Tur}, which were partially motivated by  Schur-Weyl
dualities between walled Brauer algebras and general linear groups   arising from mutually commuting
actions on mixed tensor  modules $V^{\otimes r}\OTIMES (V^*)^{\otimes t}$ of
the $r$-th power of the natural module $V$ and the $t$-th power of the dual natural module $V^*$ of $GL_n$ for various $r,t\in\Z^{\ge0}$. Benkart, etc.~\cite{BCHLLJ} used
walled Brauer algebras
to study decompositions of  mixed tensor modules %representations of general linear groups
of $GL_n$.~% over $\mathbb C$.
Since then,  walled Brauer algebras  have been intensively studied, e.g.,  \cite{BS, CD, CVDM,  DD,N, SM}, etc. In particular,  blocks and decomposition matrices
 of walled Brauer algebras over $\mathbb C$ were determined in \cite{CD, CVDM}.
Recently,  Brundan  and Stroppel \cite{BS} obtained $\Z$-gradings on $\mathscr B_{r, t}(\delta)$, proved the Koszulity of $\mathscr B_{r, t}(\delta)$  and established Morita equivalences between $\mathscr B_{r, t}(\delta)$  and
idempotent truncations of
certain infinite dimensional versions of Khovanov's arc algebras.

In 2002, by studying mixed tensor modules of general linear Lie superalgebras ${\frak{gl}}_{m|n}$, Shader and Moon \cite{SM} set up super  Schur-Weyl dualities between walled Brauer algebras and general linear Lie superalgebras. By studying tensor modules $K_\l\OTIMES V^{\otimes r}$ of Kac modules $K_\l$ with the
$r$-th power $V^{\otimes r}$ of the natural module $V$ of ${\frak{gl}}_{m|n}$, Brundan  and Stroppel  \cite{BS4}
further established super Schur-Weyl dualities between level two
Hecke algebras $H_r^{p,q}$ (denoted as $\mathscr H_{2, r}$ in the present paper) and general linear Lie superalgebras ${\frak{gl}}_{m|n}$, which provide powerful tools  enabling  them to obtain various results including a
spectacular one on Morita equivalences between blocks of categories of finite dimensional ${\frak{gl}}_{m|n}$-modules
and categories of
finite-dimensional left modules over some generalized Khovanov's diagram algebras.
A natural question is, what kind of algebras may come into the play if one replaces
the tensor modules $K_\l\OTIMES V^{\otimes r}$ by the tensor modules $V^{\otimes r}\OTIMES K_\l\OTIMES  (V^*)^{\otimes t}$ of
Kac modules $K_\l$ with the $r$-th power of the natural module $V$ and the $t$-th power of the dual natural module $V^*$ of $\frak{gl}_{m|n}$.
%
% It is natural to ask whether there is a level two walled Brauer algebra such that there is a super Schur-Weyl duality between  general
%linear  Lie superalgebras ${\frak{gl}}_{m|n}$ and level two walled Brauer algebra algebras.
This is one of our motivations to introduce a new class of associative algebras $\mathscr B_{r,t}^{\text{aff}}$ (cf. Definition \ref{wbmw}), referred to as affine walled Brauer algebras, over a   commutative ring containing~$1$.

At a first glance, the objects %affine walled Brauer algebras
$\mathscr B_{r,t}^{\text{aff}}$
we defined, which have as many as
26 defining  relations, are artificial. However, a surprising thing is that not only level two walled
 Brauer algebras ${\mathscr  B_{r,t}^{p,q}}(m,n)$ (defined as
cyclotomic quotients of  affine walled Brauer algebras $\mathscr B_{r,t}^{\text{aff}}$ with special defining parameters, cf.~\eqref{L2-wb} and \mbox{Definition} \ref{level2})
have weakly cellular structures (cf.~Theorem \ref{cellular-1}), but also
there is a \mbox{super Schur-Weyl} duality between
affine walled algebras $\mathscr B_{r,t}^{\text{aff}}$
and general linear Lie superalgebras
${\frak{gl}}_{m|n}$ over $\mathbb C$.
In this case, level two walled Brauer algebras ${\mathscr  B_{r,t}^{p,q}}(m,n)$
%, which are cyclotomic quotients of  affine walled Brauer algebras with special defining parameters,
can be realized as  endomorphism algebras
of tensor modules $V^{\otimes r}\OTIMES K_\l\OTIMES  (V^*)^{\otimes t}$, under some conditions
(cf.~Theorem \ref{level-2}). Furthermore, using the weakly cellular structures on
${\mathscr  B_{r,t}^{p,q}}(m,n)$, we are able to
give a classification of their irreducible modules % of ${\mathscr  B_{r,t}^{p,q}}(m,n)$
(cf.~Theorem \ref{main-3}).
The result in turn enables us to classify  the  indecomposable direct summands of the said tensor modules
of ${\frak{gl}}_{m|n}$ (cf.~Theorem \ref{level-3}).
In contrast to level two Hecke algebras $H_r^{p,q}$ in \cite[IV]{BS4},
%(denoted as $\mathscr H_{2, r}$ in the present paper),
which only depend on $p{\sc\!}-{\sc\!}q$ and $r$, level two walled  Brauer algebras $\DBr$ heavily depend on parameters $p\!-\!q,\,r,t,m,n$ (cf.~Remark \ref{rema-wll2}).
Nevertheless,
affine walled algebras %$\mathscr B_{r,t}^{\text{aff}}$
can be regarded as affinizations of walled Brauer algebras.
In this sense, %it seems to us that
the appearing of affine walled algebras $\mathscr B_{r,t}^{\text{aff}}$ is %very
natural.

Another motivation comes from
Nazarov's works \cite{Na} on the Juscy-Murphy elements of Brauer algebras and
affine Wenzl algebras.
We  construct a family of Juscy-Murphy-like elements of walled Brauer algebras (cf.~Definition \ref{yi}), which have close relationship with their certain
central elements. By studying properties of these elements in details, we are not only able to give the precise definition of
 affine walled Brauer algebras $\mathscr B_{r,t}^{\text{aff}}$ (which can also be regarded as  counterparts of Nazarov's affine Wenzl algebras \cite{Na} in this sense), but also able to  set up a family of homomorphisms $\phi_k$ from affine walled Brauer algebras $\mathscr B_{r,t}^{\text{aff}}$ to walled Brauer
algebras $\mathscr B_{k+r,k+t}(\omega_0)$ for any $k\in\Z^{\ge1}$ (cf.~Theorem \ref{alghom}). This then enables us to use the freeness of  walled Brauer
algebras to prove that affine walled Brauer algebras $\mathscr B_{r,t}^{\text{aff}}$ are free with infinite rank
over a commutative ring containing $1$ (cf.~Theorems \ref{main1} and
\ref{main2}).
%
%
%$M^{rt}$ of ${\frak{gl}}_{m|n}$, where $M^{rt}$ (cf.~\eqref{M-st==}) is the tensor of a Kac module with the $r$-th power of the natural module and the $t$-th power of the dual natural module.

 It is very natural to ask whether  level two walled Brauer algebras will play the role similar to that of level two Hecke algebras in Brundan-Stroppel's work  \cite{BS4}, etc.~This will be persuaded in a sequel,
where we will establish some relationship between
decomposition matrices of  level two walled Brauer algebras and structures of
indecomposable direct summands of the above-mentioned tensor modules, which are tilting modules when Kac modules $K_\l$ are typical.
% $V^{\otimes r}\otimes K_\l\otimes  (V^*)^{\otimes t}$.
%$M^{rt}$.

We organize the paper as follows. In Section~2, after recalling the notion of  walled Brauer algebras, we introduce affine walled Brauer algebras over a
commutative ring $R$ containing $1$. In Section~3, we introduce a family of
Juscy-Murphy-like elements
of walled Brauer algebras and %state their properties in details. Via these properties,  we set up
establish a family of  homomorphisms from affine walled Brauer algebras to walled Brauer algebras.
Using these homomorphisms and the freeness of walled Brauer algebras, we prove the freeness of  affine walled Brauer
algebras
%are free  over $R$ with infinite rank
in Section~4.
In Section~5, we study the super Schur-Weyl duality between  affine walled Brauer algebras
 (more precisely, level two walled Brauer algebras)
with special
 parameters and general linear Lie superalgebras. % ${\frak{gl}}_{m|n}$ over $\mathbb C$.
In Section~\ref{levelw-2}, we construct a weakly  cellular basis of  level two walled Brauer algebras.
Finally in Section~\ref{decmpw}, we give a classification of
their
irreducible modules, and a classification of  the  indecomposable direct summands of the aforementioned tensor modules.
 \vspace*{-5pt}

\section{The walled Brauer algebra $\mathscr B_{r, t}$ and its affinizatio\vspace*{-5pt}n}
Throughout this section, let  $R$ be  a commutative ring containing $1$. The {\it walled Brauer algebra} is an associative algebra over $R$ spanned by so called {\it walled
Brauer diagrams} as follows.

Fix two positive integers $r$ and $t$. A {\it walled  $(r, t)$-Brauer diagram} is a diagram with $(r\!+\!t)$ vertices on the top and bottom rows, and vertices on both rows are labeled  from left to right
by $r, \cdots,2, 1, \bar 1, \bar 2, \cdots, \bar t$.
Every vertex $i\!\in\!\{1, 2, \cdots, r\}$ (resp., $\bar i\!\in $ $\{\bar 1, \bar 2, \cdots, \bar t\}$) on each
row must be connected to a unique vertex $\bar j$ (resp., $j$) on the same row or a unique vertex
$j$ (resp., $\bar j$) on the other row. In this way, we obtain $4$ types of pairs $[i,j],$ $[i,\bar j]$, $[\bar i,j]$ and $[\bar i,\bar j]$.
The pairs $[i, j]$ and $[\bar i, \bar j]$ are called
{\it vertical edges}, and the pairs    $[\bar i, j]$ and $[i, \bar j]$ are called
{\it horizontal edges}. If we imagine that  there is a wall which separates the vertices $1, \bar 1$ on both top and bottom rows,
then a walled $(r,t)$-Brauer diagram is a diagram    with $(r+t)$ vertices on both rows such that each vertical edge can not cross the wall and each horizontal edge
has to cross the wall.  For convenience,  we  call a walled  $(r, t)$-Brauer diagram a {\it walled Brauer diagram} if there is no confusion.

\begin{example}

The following are $(r, t)$-Brauer diagrams:{\small
$$
\put(12,48){\line(0,-1){3}}
\put(12,42){\line(0,-1){3}}
\put(12,36){\line(0,-1){3}}
\put(12,30){\line(0,-1){3}}
\put(12,24){\line(0,-1){3}}
\put(12,18){\line(0,-1){3}}
\put(12,12){\line(0,-1){3}}
\put(12,6){\line(0,-1){3}}
\put(12,0){\line(0,-1){3}}
\put(12,-6){\line(0,-1){3}}
\put(0,2){\tiny$\ssc\bullet$}\put(20,2){\tiny$\ssc\bullet$}
\put(0,45){\small$1$}
\put(20,45){\small$\bar 1$}
\put(0,-10){\small$ 1$}
\put(20,-10){\small$\bar1$}
\put(0,0){\Huge$\frown$}
\put(0,30){\Huge$\smile$}
\put(0,38){\tiny$\ssc\bullet$}\put(20,38){\tiny$\ssc\bullet$}
\put(-63,36){$\cdots\cdots$}
\put(-63,-2){$\cdots\cdots$}
\put(-15,0){\tiny$\ssc\bullet$}\put(-15,38){\tiny$\ssc\bullet$}
\put(-15,45){\small$ 2$}\put(-15,-10){\small$ 2$}\put(-13,40){\line(0,-1){38}}
\put(-30,45){\small$3$}\put(-30,-10){\small$3$}\put(-28,40){\line(0,-1){38}}
\put(-30,0){\tiny$\ssc\bullet$}\put(-30,38){\tiny$\ssc\bullet$}
\put(60,0){\put(-15,0){\tiny$\ssc\bullet$}\put(-15,38){\tiny$\ssc\bullet$}
\put(-15,45){\small$\bar3$}\put(-15,-10){\small$\bar3$}\put(-13,40){\line(0,-1){38}}
\put(-30,45){\small$\bar 2$}\put(-30,-10){\small$\bar 2$}\put(-28,40){\line(0,-1){38}}
\put(-30,0){\tiny$\ssc\bullet$}\put(-30,38){\tiny$\ssc\bullet$}
\put(-8,36){$\cdots\cdots$}
\put(-8,-2){$\cdots\cdots$}}
\put(0,-25){\small Figure 1}
$$\\[-10pt]
$$
\put(12,48){\line(0,-1){3}}
\put(12,42){\line(0,-1){3}}
\put(12,36){\line(0,-1){3}}
\put(12,30){\line(0,-1){3}}
\put(12,24){\line(0,-1){3}}
\put(12,18){\line(0,-1){3}}
\put(12,12){\line(0,-1){3}}
\put(12,6){\line(0,-1){3}}
\put(12,0){\line(0,-1){3}}
\put(12,-6){\line(0,-1){3}}
\put(0,2){\tiny$\ssc\bullet$}\put(20,2){\tiny$\ssc\bullet$}
\put(0,45){\small$1$}
\put(20,45){\small$\bar 1$}
\put(0,-10){\small$ 1$}
\put(20,-10){\small$\bar1$}
%
%\put(0,0){\Huge$\frown$}\put(0,30){\Huge$\smile$}
%
\put(0,38){\tiny$\ssc\bullet$}\put(20,38){\tiny$\ssc\bullet$}
\put(1,40){\line(0,-1){38}}\put(21,40){\line(0,-1){38}}
\put(-17,36){$\cdots$}
\put(-17,-2){$\cdots$}
\put(-55,0){
\put(-88,36){$\cdots\cdots$}
\put(-88,-2){$\cdots\cdots$}
\put(-50,0){\tiny$\ssc\bullet$}\put(-50,38){\tiny$\ssc\bullet$}
\put(-50,45){\small$\!\!\!\! i\!+\!2$}\put(-50,-10){\small$\!\!\!\! i\!+\!2$}\put(-48,40){\line(0,-1){38}}
\put(-20,0){\tiny$\ssc\bullet$}\put(-20,38){\tiny$\ssc\bullet$}
\put(-20,45){\small$\!\!\!\! i\!+\!1$}\put(-20,-10){\small$\!\!\!\! i\!+\!1$}%\put(-33,40){\line(0,-1){38}}
\put(-17,39){\line(1,-2){18}}
\put(-17,2){\line(1,2){18}}
\!\!\!\!\!
\put(10,0){\tiny$\ssc\bullet$}\put(10,38){\tiny$\ssc\bullet$}
\put(10,45){\small$ i$}\put(10,-10){\small$ i$}%\put(-23,40){\line(0,-1){38}}
}
\put(-25,0){\tiny$\ssc\bullet$}\put(-25,38){\tiny$\ssc\bullet$}
\put(-25,45){\small$\!\!\!\! i\!-\!1$}\put(-25,-10){\small$\!\!\!\! i\!-\!1$}\put(-23,40){\line(0,-1){38}}
%
%
%\put(-30,45){\small$3$}\put(-30,-10){\small$3$}\put(-28,40){\line(0,-1){38}}
%\put(-30,0){\tiny$\ssc\bullet$}\put(-30,38){\tiny$\ssc\bullet$}
%
%
\put(60,0){
%\put(-15,0){\tiny$\ssc\bullet$}\put(-15,38){\tiny$\ssc\bullet$}
%\put(-15,45){\small$\bar3$}\put(-15,-10){\small$\bar3$}\put(-13,40){\line(0,-1){38}}
\put(-30,45){\small$\bar 2$}\put(-30,-10){\small$\bar 2$}\put(-28,40){\line(0,-1){38}}
\put(-30,0){\tiny$\ssc\bullet$}\put(-30,38){\tiny$\ssc\bullet$}
\put(-23,36){$\cdots\cdots$}
\put(-23,-2){$\cdots\cdots$}}
\put(0,-25){\small Figure 2}
$$\\[-10pt]
$$
\put(12,48){\line(0,-1){3}}
\put(12,42){\line(0,-1){3}}
\put(12,36){\line(0,-1){3}}
\put(12,30){\line(0,-1){3}}
\put(12,24){\line(0,-1){3}}
\put(12,18){\line(0,-1){3}}
\put(12,12){\line(0,-1){3}}
\put(12,6){\line(0,-1){3}}
\put(12,0){\line(0,-1){3}}
\put(12,-6){\line(0,-1){3}}
\put(-48,36){$\cdots\cdots$}
\put(-48,-2){$\cdots\cdots$}
\put(-15,0){\tiny$\ssc\bullet$}\put(-15,38){\tiny$\ssc\bullet$}
\put(-15,45){\small$ 2$}\put(-15,-10){\small$ 2$}\put(-13,40){\line(0,-1){38}}
\put(0,2){\tiny$\ssc\bullet$}\put(20,2){\tiny$\ssc\bullet$}
\put(0,45){\small$1$}
\put(20,45){\small$\bar 1$}
\put(0,-10){\small$ 1$}
\put(20,-10){\small$\bar1$}
%
%\put(0,0){\Huge$\frown$}\put(0,30){\Huge$\smile$}
%
\put(0,38){\tiny$\ssc\bullet$}\put(20,38){\tiny$\ssc\bullet$}
\put(1,40){\line(0,-1){38}}\put(21,40){\line(0,-1){38}}
%
%
%%%%%%%%%%\put(-17,36){$\cdots$}\put(-17,-2){$\cdots$}
%
\put(125,0){
%
%\put(-88,36){$\cdots\cdots$}
%\put(-88,-2){$\cdots\cdots$}
\put(-50,0){\tiny$\ssc\bullet$}\put(-50,38){\tiny$\ssc\bullet$}
\put(-50,45){\small$\!\!\!\! \overline{i\!-\!1}$}\put(-50,-10){\small$\!\!\!\! \overline{i\!-\!1}$}\put(-48,40){\line(0,-1){38}}
\put(-20,0){\tiny$\ssc\bullet$}\put(-20,38){\tiny$\ssc\bullet$}
\put(-20,45){\small$ \bar i$}\put(-20,-10){\small$\bar i$}%\put(-33,40){\line(0,-1){38}}
\put(-17,39){\line(1,-2){18}}
\put(-17,2){\line(1,2){18}}
\!\!\!\!\!
\put(10,0){\tiny$\ssc\bullet$}\put(10,38){\tiny$\ssc\bullet$}
\put(10,45){\small$\!\!\!\!\overline{i\!+\!1}$}\put(10,-10){\small$\!\!\!\!\overline{i\!+\!1}$}%\put(-23,40){\line(0,-1){38}}
\put(6,0){
\put(35,0){\tiny$\ssc\bullet$}\put(35,38){\tiny$\ssc\bullet$}
\put(35,45){\small$\!\!\!\! \overline{i\!+\!2}$}\put(35,-10){\small$\!\!\!\! \overline{i\!+\!2}$}\put(37,40){\line(0,-1){38}}
\put(44,36){$\cdots\cdots$}
\put(44,-2){$\cdots\cdots$}
}
}
%
%
%\put(-25,0){\tiny$\ssc\bullet$}\put(-25,38){\tiny$\ssc\bullet$}
%\put(-25,45){\small$\!\!\!\! i\!-\!1$}\put(-25,-10){\small$\!\!\!\! i\!-\!1$}\put(-23,40){\line(0,-1){38}}
%
%
%\put(-30,45){\small$3$}\put(-30,-10){\small$3$}\put(-28,40){\line(0,-1){38}}
%\put(-30,0){\tiny$\ssc\bullet$}\put(-30,38){\tiny$\ssc\bullet$}
%
%
\put(60,0){
%\put(-15,0){\tiny$\ssc\bullet$}\put(-15,38){\tiny$\ssc\bullet$}
%\put(-15,45){\small$\bar3$}\put(-15,-10){\small$\bar3$}\put(-13,40){\line(0,-1){38}}
\put(-30,45){\small$\bar 2$}\put(-30,-10){\small$\bar 2$}\put(-28,40){\line(0,-1){38}}
\put(-30,0){\tiny$\ssc\bullet$}\put(-30,38){\tiny$\ssc\bullet$}
\put(-23,36){$\cdots\cdots$}\put(-23,-2){$\cdots\cdots$}
}
\put(0,-25){\small Figure 3}
$$
}
\def\PICT{%
%\newpage
%
{\small\footnotesize\unitlength 0.7mm
\linethickness{0.4pt}
 \ifx\plotpoint\undefined\newsavebox{\plotpoint}\fi
%\small\footnotesize
\ \ \ \ \ \ \ \ \ \ \ \ \begin{picture}(50,150)(40,20) \put(68,161){\circle*{1}}
 \put(79,161){\circle*{1}} \put(90,161){\circle*{1}} \put(112,161){\circle*{1}} \put(123,161){\circle*{1}} \put(134,161){\circle*{1}}
\put(101.18,161.93){\line(0,-1){1}} \put(101.18,159.97){\line(0,-1){1}} \put(101.18,158.011){\line(0,-1){1}} \put(101.18,156.051){\line(0,-1){1}} \put(101.18,154.092){\line(0,-1){1}} \put(101.18,152.132){\line(0,-1){1}} \put(101.18,150.173){\line(0,-1){1}} \put(101.18,148.213){\line(0,-1){1}} \put(101.18,146.254){\line(0,-1){1}} \put(101.18,144.295){\line(0,-1){1}} \put(101.18,142.335){\line(0,-1){1}} \put(101.18,140.376){\line(0,-1){1}} \put(101.18,138.416){\line(0,-1){1}} \put(101.18,136.457){\line(0,-1){1}} \put(101.18,134.497){\line(0,-1){1}} \put(101.18,132.538){\line(0,-1){1}} \put(101.18,130.578){\line(0,-1){1}} \put(101.18,128.619){\line(0,-1){1}} \put(101.18,126.659){\line(0,-1){1}}
\put(68,125){\circle*{1}} \put(79,125){\circle*{1}} \put(90,125){\circle*{1}} \put(112,125){\circle*{1}} \put(123,125){\circle*{1}} \put(134,125){\circle*{1}}
 \put(68,161){\line(0,-1){34.75}} \put(79,161){\line(0,-1){34.75}} \put(123,161){\line(0,-1){35.75}} \put(134,161){\line(0,-1){34.75}} \qbezier(90,161)(101,147)(112,161) \qbezier(90,125)(101,147)(112,125)
 \put(90,163){1} \put(112,163){$\bar 1$} \put(90,118){1} \put(112,118){$\bar 1$}
\put(79,163){2}\put(123,163){$\bar 2$} \put(79,118){2} \put(123,118){$\bar 2$}
\put(68,163){3}\put(134,163){$\bar 3$} \put(68,118){3} \put(134,118){$\bar 3$}
\put(140,125){\line(1,0){1}} \put(142,125){\line(1,0){1}} \put(144,125){\line(1,0){1}} \put(146,125){\line(1,0){1}} \put(148,125){\line(1,0){1}} \put(150,125){\line(1,0){1}}\put(152,125){\line(1,0){1}}\put(154,125){\line(1,0){1}}
\put(140,161){\line(1,0){1}} \put(142,161){\line(1,0){1}} \put(144,161){\line(1,0){1}} \put(146,161){\line(1,0){1}} \put(148,161){\line(1,0){1}} \put(150,161){\line(1,0){1}}\put(152,161){\line(1,0){1}}\put(154,161){\line(1,0){1}}
\put(62,161){\line(1,0){1}} \put(60,161){\line(1,0){1}} \put(58,161){\line(1,0){1}} \put(56,161){\line(1,0){1}} \put(54,161){\line(1,0){1}} \put(52,161){\line(1,0){1}}\put(50,161){\line(1,0){1}}\put(48,161){\line(1,0){1}}
\put(62,125){\line(1,0){1}} \put(58,125){\line(1,0){1}} \put(56,125){\line(1,0){1}} \put(54,125){\line(1,0){1}} \put(52,125){\line(1,0){1}} \put(50,125){\line(1,0){1}}\put(60,125){\line(1,0){1}}\put(48,125){\line(1,0){1}}
\put(101,110){Figure 1}
\end{picture}
\vskip12pt
\ \ \ \ \ \ \ \ \ \ \ \ \begin{picture}(50,50)(10,30)
\put(35,161){\circle*{1}}\put(46,161){\circle*{1}}\put(57,161){\circle*{1}}\put(68,161){\circle*{1}} \put(90,161){\circle*{1}} \put(112,161){\circle*{1}} \put(123,161){\circle*{1}}
\put(35,125){\circle*{1}}\put(46,125){\circle*{1}}\put(57,125){\circle*{1}}\put(68,125){\circle*{1}} \put(90,125){\circle*{1}} \put(112,125){\circle*{1}} \put(123,125){\circle*{1}}
\put(101.18,161.93){\line(0,-1){1}} \put(101.18,159.97){\line(0,-1){1}} \put(101.18,158.011){\line(0,-1){1}} \put(101.18,156.051){\line(0,-1){1}} \put(101.18,154.092){\line(0,-1){1}} \put(101.18,152.132){\line(0,-1){1}} \put(101.18,150.173){\line(0,-1){1}} \put(101.18,148.213){\line(0,-1){1}} \put(101.18,146.254){\line(0,-1){1}} \put(101.18,144.295){\line(0,-1){1}} \put(101.18,142.335){\line(0,-1){1}} \put(101.18,140.376){\line(0,-1){1}} \put(101.18,138.416){\line(0,-1){1}} \put(101.18,136.457){\line(0,-1){1}} \put(101.18,134.497){\line(0,-1){1}} \put(101.18,132.538){\line(0,-1){1}} \put(101.18,130.578){\line(0,-1){1}} \put(101.18,128.619){\line(0,-1){1}} \put(101.18,126.659){\line(0,-1){1}}
 \put(68,161){\line(0,-1){34.75}} \put(90,161){\line(0,-1){34.75}} \put(123,161){\line(0,-1){35.75}} \put(112,161){\line(0,-1){34.75}}
 \put(35,161){\line(0,-1){34.75}} \put(46,161){\line(1,-3){12}}\put(46,125){\line(1,3){12}}
 \put(90,163){1} \put(112,163){$\bar 1$} \put(90,118){1} \put(112,118){$\bar 1$}
\put(123,163){$\bar 2$}  \put(123,118){$\bar 2$}
\put(66,163){$i\!-\!1$} \put(66,118){$i\!-\!1$}
\put(57,163){$i$} \put(57,118){$i$}
\put(44,163){$i\!+\!1\,$} \put(44,118){$i\!+\!1\,$}
\put(31,163){$i\!+\!2$} \put(31,118){$i\!+\!2$}
\put(74,161){\line(1,0){1}} \put(76,161){\line(1,0){1}} \put(78,161){\line(1,0){1}}
\put(80,161){\line(1,0){1}} \put(82,161){\line(1,0){1}} \put(84,161){\line(1,0){1}}
\put(74,125){\line(1,0){1}} \put(76,125){\line(1,0){1}} \put(78,125){\line(1,0){1}}
\put(80,125){\line(1,0){1}} \put(82,125){\line(1,0){1}} \put(84,125){\line(1,0){1}}
\put(129,125){\line(1,0){1}} \put(143,125){\line(1,0){1}} \put(141,125){\line(1,0){1}} \put(139,125){\line(1,0){1}} \put(137,125){\line(1,0){1}} \put(131,125){\line(1,0){1}}\put(133,125){\line(1,0){1}}\put(135,125){\line(1,0){1}}
\put(129,161){\line(1,0){1}} \put(131,161){\line(1,0){1}} \put(133,161){\line(1,0){1}} \put(135,161){\line(1,0){1}} \put(137,161){\line(1,0){1}} \put(139,161){\line(1,0){1}}\put(141,161){\line(1,0){1}}\put(143,161){\line(1,0){1}}
\put(29,161){\line(1,0){1}} \put(27,161){\line(1,0){1}} \put(25,161){\line(1,0){1}} \put(23,161){\line(1,0){1}} \put(21,161){\line(1,0){1}} \put(19,161){\line(1,0){1}}\put(17,161){\line(1,0){1}}\put(15,161){\line(1,0){1}}
\put(29,125){\line(1,0){1}} \put(27,125){\line(1,0){1}} \put(25,125){\line(1,0){1}} \put(23,125){\line(1,0){1}} \put(21,125){\line(1,0){1}} \put(19,125){\line(1,0){1}}\put(17,125){\line(1,0){1}}\put(15,125){\line(1,0){1}}
\put(70,110){Figure 2}
\end{picture}
\vskip10pt
\ \ \ \ \ \ \ \ \ \ \ \ \begin{picture}(0,-50)(-50,90)
\put(35,161){\circle*{1}}\put(46,161){\circle*{1}}\put(57,161){\circle*{1}}\put(68,161){\circle*{1}} \put(9,161){\circle*{1}} \put(-13,161){\circle*{1}} \put(-24,161){\circle*{1}}
\put(35,125){\circle*{1}}\put(46,125){\circle*{1}}\put(57,125){\circle*{1}}\put(68,125){\circle*{1}} \put(9,125){\circle*{1}} \put(-13,125){\circle*{1}} \put(-24,125){\circle*{1}}
\put(-2,161.93){\line(0,-1){1}} \put(-2,159.97){\line(0,-1){1}} \put(-2,158.011){\line(0,-1){1}} \put(-2,156.051){\line(0,-1){1}} \put(-2,154.092){\line(0,-1){1}} \put(-2,152.132){\line(0,-1){1}} \put(-2,150.173){\line(0,-1){1}} \put(-2,148.213){\line(0,-1){1}} \put(-2,146.254){\line(0,-1){1}} \put(-2,144.295){\line(0,-1){1}} \put(-2,142.335){\line(0,-1){1}} \put(-2,140.376){\line(0,-1){1}} \put(-2,138.416){\line(0,-1){1}} \put(-2,136.457){\line(0,-1){1}} \put(-2,134.497){\line(0,-1){1}} \put(-2,132.538){\line(0,-1){1}} \put(-2,130.578){\line(0,-1){1}} \put(-2,128.619){\line(0,-1){1}} \put(-2,126.659){\line(0,-1){1}}
 \put(68,161){\line(0,-1){34.75}} \put(9,161){\line(0,-1){34.75}} \put(-13,161){\line(0,-1){35.75}} \put(-24,161){\line(0,-1){34.75}}
 \put(35,161){\line(0,-1){34.75}} \put(46,161){\line(1,-3){12}}\put(46,125){\line(1,3){12}}
 \put(-13,163){1} \put(9,163){$\bar 1$} \put(-13,118){1} \put(9,118){$\bar 1$}
\put(-24,163){$ 2$}  \put(-24,118){$ 2$}
\put(66,163){$\overline{i\!+\!2}$} \put(66,118){$\overline{i\!+\!2}$}
\put(54,163){$\!\!\overline{i\!+\!1}\,$} \put(54,118){$\!\!\overline{i\!+\!1}\,$}
\put(44,163){$\bar i$} \put(44,118){$\bar i$}
\put(31,163){$\overline{i\!-\!1}$} \put(31,118){$\overline{i\!-\!1}$}
\put(74,161){\line(1,0){1}} \put(76,161){\line(1,0){1}} \put(78,161){\line(1,0){1}}
\put(80,161){\line(1,0){1}} \put(82,161){\line(1,0){1}} \put(84,161){\line(1,0){1}}
\put(74,125){\line(1,0){1}} \put(76,125){\line(1,0){1}} \put(78,125){\line(1,0){1}}
\put(80,125){\line(1,0){1}} \put(82,125){\line(1,0){1}} \put(84,125){\line(1,0){1}}
\put(-30,125){\line(1,0){1}} \put(-32,125){\line(1,0){1}} \put(-34,125){\line(1,0){1}} \put(-36,125){\line(1,0){1}} \put(-38,125){\line(1,0){1}} \put(-40,125){\line(1,0){1}}\put(-42,125){\line(1,0){1}}\put(-44,125){\line(1,0){1}}
\put(-30,161){\line(1,0){1}} \put(-32,161){\line(1,0){1}} \put(-34,161){\line(1,0){1}} \put(-36,161){\line(1,0){1}} \put(-38,161){\line(1,0){1}} \put(-40,161){\line(1,0){1}}\put(-42,161){\line(1,0){1}}\put(-44,161){\line(1,0){1}}
\put(29,161){\line(1,0){1}} \put(27,161){\line(1,0){1}} \put(25,161){\line(1,0){1}} \put(23,161){\line(1,0){1}} \put(21,161){\line(1,0){1}} \put(19,161){\line(1,0){1}}\put(17,161){\line(1,0){1}}\put(15,161){\line(1,0){1}}
\put(29,125){\line(1,0){1}} \put(27,125){\line(1,0){1}} \put(25,125){\line(1,0){1}} \put(23,125){\line(1,0){1}} \put(21,125){\line(1,0){1}} \put(19,125){\line(1,0){1}}\put(17,125){\line(1,0){1}}\put(15,125){\line(1,0){1}}
\put(10,110){Figure 3}
\end{picture}}\vspace*{-35pt}\
}

\noindent
Throughout, we denote  by $e_1, s_i, \bar s_i$ the diagrams in Figures~1--3,
  respectively.%\\[-5pt]
\end{example}

In order to define the product of two walled Brauer diagrams, we consider the {\it composition} $D_1\circ D_2$ of
two walled  Brauer diagrams $D_1$ and $D_2$, which is obtained by putting $D_1$ above $D_2$ and connecting each
 vertex on the bottom row of $D_1$ to the corresponding vertex on the top row of $D_2$.
If we remove all circles of $D_1\circ D_2$, we will get a walled Brauer diagram, say $D_3$. Let $n(D_1, D_2)$ be the number of circles appearing in $D_1\circ D_2$. Then the {\it product} $D_1 D_2$ of $D_1$ and $D_2$ is defined to be $\delta^{n(D_1, D_2)} D_3$, where $\delta$ is a fixed element in $R$.

\begin{definition}\label{Wall-defi} \cite{Koi,Tur} The walled Brauer algebra $\mathscr B_{r, t}(\delta)$ with respect to the defining parameter $\delta$ is the
associative algebra over $R$ spanned by all walled $(r, t)$-Brauer diagrams with product defined as above. \end{definition}
%In this section, we introduce certain elements of walled Brauer
%algebras over $R$ and state some of its properties. This will lead us to introduce the affine walled Brauer algebras.

\begin{remark}\label{Wall-brauer}
If we allow vertical edges can cross the wall and allow horizontal edges may not cross the wall (namely, a vertex can be connected to any other vertex), then we obtain $(r+t)$-{\it Brauer diagrams}.
The {\it  Brauer algebra} $\mathscr B_{r+t}(\delta)$ \cite{B} is  the free $R$-modules spanned by all $(r+t)$-Brauer diagrams with product defined as above.
Thus  a  walled Brauer diagram is a Brauer diagram, and the walled Brauer algebra $\mathscr B_{r,t}(\delta)$  is a subalgebra of the  Brauer algebra $\mathscr B_{r+t}(\delta)$.
\end{remark}
The following result can be found in \cite[Corollary~4.5]{Ha} for a special case and \cite[Theorem~4.1]{N} in general.
\begin{theorem}\label{wbmwf} Let $R$ be a commutative ring  containing $1$ and $\delta$.  Then  ${\mathscr B}_{r, t}(\delta)$  is an associative $R$-algebra
generated by $e_1, s_i, \bar s_j$ with $   1\le i\le r-1$  and  $1\le j\le t-1$
 subject to the following relations
\begin{multicols}{2}
\begin{enumerate}
\item $s_i^2=1$,  $1\le i<r$,
\item $s_is_j=s_js_i$,  $|i-j|>1$,
\item $s_is_{i+1}s_i\!=\!s_{i+1}s_is_{i+1}$, $1\!\le\! i\!<\!r\!-\!1$,
\item $s_ie_1=e_1s_i$,  $2\le i< r $,
\item $e_1s_1e_1 = e_1$,
\item $e_1^2=\delta e_1$,
\item $s_i\bar s_j=\bar s_j s_i$,
\item $\bar s_i^2=1$,  $1\le i<t$,
\item $\bar s_i\bar s_j=\bar s_j\bar s_i$, $|i-j|>1$,
\item $\bar s_i\bar s_{i+1}\bar s_i\!=\!\bar s_{i+1}\bar s_i\bar s_{i+1}$, $1\!\le \!i\!<\!t\!-\!1$,
\item $\bar s_ie_1=e_1\bar s_i$, $2\le i< t$,
\item $e_1\bar s_1e_1= e_1$,
\item $e_1s_1\bar s_{1} e_1s_{1} = e_1s_1\bar s_1 e_1\bar s_{1}$,
\item $s_1e_1s_1\bar s_1 e_1 = \bar s_1 e_1s_1 \bar s_1 e_1$.
\end{enumerate}
\end{multicols}\noindent
In particular, the rank of $\mathscr B_{r, t}(\delta)$ is $(r+t)!$.
\end{theorem}
We remark that Jung and Kang  gave a presentation of {\textit  {walled Brauer superalgebras}} in \cite[Theorem 5.1]{JK}, and
the presentation of walled Brauer algebras in Theorem \ref{wbmwf} can be obtained from
those of walled Brauer superalgebras  by removing the generators of Clifford algebras inside
walled Brauer superalgebras.
%We will use two definitions of $\mathscr B_{r, t}(\delta)$ freely later on.

The following two results can be deduced from Theorem~\ref{wbmwf}, easily.

 \begin{lemma}\label{anti1}
There is an $R$-linear anti-involution $\sigma: \mathscr B_{r,t}(\delta)\rightarrow \mathscr B_{r,t}(\delta)$
fixing defining generators $s_i, \bar s_j$ and $e_1$ for all possible $i, j$'s.\end{lemma}

\begin{proof} The result follows from the symmetry of relations in Theorem~\ref{wbmwf}, immediately.
In particular, the image of a walled Brauer diagram $D$ under the map $\sigma$ is the diagram which is obtained from $D$ by reflecting along a horizontal line.
\end{proof}
\begin{corollary}\label{iso-1} We have $\mathscr B_{r, t} (\delta)\!\cong \!\mathscr B_{t, r}(\delta)$. In particular, the corresponding isomorphism sends
$s_i, e_1, \bar s_j$ of  $\mathscr B_{r, t}(\delta) $ to $\bar s_i, e_1,  s_j$ of
 $\mathscr B_{t, r}(\delta)$.
\end{corollary}
\begin{proof}One can easily observe that
the automorphism can be obtained by first rotating a diagram through $180^{\rm o}$ and then reflecting along a horizontal line.
\end{proof}

In the present paper, we shall introduce a new class of associative algebras,  which are  natural generalizations of walled Brauer algebras,
and thus can be regarded as  affinizations of walled Brauer algebras. Such algebras
can  also be  considered as the counterparts of Nazarov's  affine Wenzl algebras in \cite{Na}. This is one of our motivations to introduce
these algebras. Another motivation originates from super Schur-Weyl dualities in \cite{BS4, SM} and ours in Section \ref{SW}.
%
%
%Motivated by Lemma~\ref{defrel}, Corollary~\ref{eye}, we introduce affine walled Brauer algebras as follows.

\begin{definition}\label{wbmw}  Let  $R$ be a commutative ring  containing $1, \omega_0,\omega_1$.
Fix $r,t \in \mathbb Z^{\ge0}$. The {\it affine  walled Brauer algebra}
$\mathscr B_{r,t}^{\text{aff}}(\omega_0, \omega_1)$ is the associative $R$-algebra generated by $e_1,  x_1, \bar x_1, s_i\,(1\!\le\! i\!\le\! r\!-\!1),\,
\bar s_j\,(1\!\le\! j\!\le\! t\!-\!1)$, and  two families of central elements
 $\omega_k\,(k\!\in\!\Z^{\ge2}),$ $\bar \omega_k\,(k\!\in\!\Z^{\ge0})$,
 subject to the following relations
\begin{multicols}{2}
\begin{enumerate}
\item\label{(1)} $s_i^2=1$,  $1\le i< r$,
\item\label{(2)} $s_is_j=s_js_i$,  $|i-j|>1$,
\item\label{(3)} $s_is_{i+1}s_i\!=\!s_{i+1}s_is_{i+1}$, $1\!\le\! i\!<\!r\!-\!1$,
\item\label{(4)} $s_ie_1=e_1s_i$,  $2\le i<r$,
\item\label{(5)}  $e_1s_1e_1 = e_1$,
\item\label{(6)}  $e_1^2=\omega_0 e_1$,
\item\label{(7)}  $s_i\bar s_j=\bar s_j s_i$,
\item\label{(8)}  $e_1(x_1+\bar x_1)=(x_1+\bar x_1) e_1=0$,
\item\label{(9)}  $e_1 s_1x_1s_1=s_1x_1s_1e_1$,
\item\label{(10)}  $s_ix_1=x_1s_i$, $2\le i< r$,
\item\label{(11)}  $s_i\bar x_1=\bar x_1 s_i$, $1\le i< r$,
\item\label{(12)}  $e_1x_1^ke_1=\omega_k e_1$, $\forall k\in \mathbb Z^{\ge0}$,
\item\label{(13)}  $x_1(s_1x_1s_1-s_1)=(s_1x_1s_1-s_1)x_1$,
\item\label{(14)}  $\bar s_i^2=1$,  $1\le i< t$,
\item\label{(15)}  $\bar s_i\bar s_j=\bar s_j\bar s_i$, $|i-j|>1$,
\item\label{(16)}  $\bar s_i\bar s_{i+1}\bar s_i\!=\!\bar s_{i+1}\bar s_i\bar s_{i+1}$, $1\!\le\! i\!< \!t\!-\!1$,
\item\label{(17)}  $\bar s_ie_1=e_1\bar s_i$, $2\le i<t$,
\item\label{(18)}  $e_1\bar s_1e_1= e_1$,
\item\label{(19)}  $e_1s_1\bar s_{1} e_1s_{1} = e_1s_1\bar s_1 e_1\bar s_{1}$,
\item\label{(20)}  $s_1e_1s_1\bar s_1 e_1 = \bar s_1 e_1s_1 \bar s_1 e_1$,
\item\label{(21)}  $x_1 (e_1+\bar x_1)= (e_1+\bar x_1)x_1$,
\item\label{(22)}  $e_1 \bar s_1\bar x_1\bar s_1=\bar s_1\bar x_1\bar s_1e_1$,
\item\label{(23)}  $\bar s_i\bar x_1=\bar x_1\bar s_i$, $2\le i<t$,
\item\label{(24)}  $\bar s_i x_1= x_1 \bar s_i$, $1\le i<t$,
%\item\label{(25)}  $e_1x_1^ae_1=\omega_a e_1$, $\forall a\in \mathbb N$,
\item\label{(25)}  $e_1\bar x_1^ke_1=\bar \omega_k e_1$, $\forall k\in \mathbb Z^{\ge0}$,
\item\label{(26)}  $\bar x_1(\bar s_1\bar x_1\bar s_1-\bar s_1)=(\bar s_1\bar x_1 \bar s_1-\bar s_1)\bar x_1$.
\end{enumerate}
\end{multicols}
\end{definition}

For simplicity, we  use $\mathscr B^{\rm {aff}}_{r, t}$ instead of $\mathscr B_{r, t}^{\rm {aff}} (\omega_0, \omega_1)$ later on.
In other words, we always assume that $\mathscr B^{\rm {aff}}_{r, t}$  is the affine walled Brauer algebra
with respect to the defining parameters $\omega_0$ and $\omega_1$.
\begin{remark}\label{remark-spe}Later on we shall be mainly interested in the case when all central elements $\omega_a$ with
$a\in\Z^{\ge2}$ are specialized to
some elements in $R$ (cf.~Theorem \ref{main2}). The reason we put $\omega_a$'s into generators is that in order to be able to
prove
the freeness of
$\mathscr B^{\rm {aff}}_{r, t}$  (cf.~Theorem \ref{main1}), we need to construct the homomorphism
$\phi_k$ (cf.~Theorem \ref{alghom}), which requires $\omega_a$'s to be generators.
\end{remark}

In~the~next~two~sections,~we~shall~prove~that~$\mathscr B^{\rm {aff}}_{r, t}$~is~a~free~$R$-algebra~%
with~infinite~rank.%\vspace*{-10pt}

\section{Homomorphisms from  $\mathscr B^{\rm {aff}}_{r, t}$  to  $\mathscr B_{k+r,k+ t}(\omega_0)$}

The purpose of this section is to establish a family of algebraic homomorphisms $\phi_k$ from $\mathscr B^{\rm {aff}}_{r, t}$ to $\mathscr B_{k+r,k+ t}(\omega_0)$ for all $k\in\Z^{\ge1}$. Then in the next section, we will
use these homomorphisms and the freeness of walled Brauer algebras  to prove the freeness of $\mathscr B^{\rm {aff}}_{r, t}$.
We remark that Nazarov \cite{Na}
used the freeness of Brauer algebras to prove the freeness of affine Wenzl algebras.

%First we remark that u
Unless otherwise indicated, all elements considered in this section are in the walled Brauer algebra  $\mathscr B_{r,t}(\delta)$ for some $r,t\in\Z^{\ge0}$ with parameter $\delta=\omega_0$.
%, which is obviously a subalgebra $\mathscr B^{\rm {aff}}_{r, t}$ by Definition \ref{wbmw}.

Denote by $\mathfrak S_r$ (resp.,  $\bar {\mathfrak S}_t$) the symmetric group in $r$ letters $1, 2, \cdots, r$ (resp.,  $t$ letters $\bar 1, \bar 2, \cdots, \bar t$).
It is well-known that the subalgebra of  $\mathscr B_{r,t}(\delta)$ generated by  $\{s_i\mid 1\!\le\! i\!< \!r\}$ (resp., $\{\bar s_j\mid 1\!\le\! j\!<\! t\}$)
is isomorphic to the group algebra $R\mathfrak S_r$ (resp.,  $R\bar {\mathfrak S}_t$)  of $\mathfrak S_r$ (resp.,  $\bar{\mathfrak S}_t$).

Let  $(i, j)\in \mathfrak S_r$  (resp.,    $(\bar i, \bar j)\in \bar{\mathfrak S}_t$) be the transposition  which  switches $i$ and $j$  (resp.,   $\bar i$ and $ \bar j$) and fixes others.
Then   $s_i$ and  $\bar s_j$  can be identified with
$$\mbox{$s_i=(i, i+1)$ \ \ and  \ \ $\bar s_j=(\overline{j}, \overline {j\!+\!1} {\sc\,})$.}$$
Set $ L_1=\bar L_1=0$ and
\begin{equation}\label{L-barL} L_i=\SUM{j=1}{i-1} (j, i) \text{, \ \  \ \ } \bar L_i=\SUM{j=1}{i-1} (\bar j, \bar i)\mbox{ \ \ for $i\ge 2$}.\end{equation}
Then  $ L_i$ for $1\le i\le r$ are known as the {\it Jucys-Murphy elements} of   $R\mathfrak S_r$,
 and $\bar L_j$ for $1\le j\le t$   are the {Jucys-Murphy elements} of $R \bar {\mathfrak S}_t$.
We will need the following
 well-known result.

\begin{lemma}\label{sl} In $R \mathfrak S_r$ and  $R \bar {\mathfrak S}_t$, for all possible $i,j$'s, we have
%\begin{enumerate}\item
$$\begin{array}{llll}
(1)\!\!\!& L_i s_j=s_j L_i,&\bar L_i \bar s_j=\bar s_j \bar L_i\mbox{ \ \ \ if \ \ $i\neq j, j+1$}.\ \ \ \ \ \ \ \ \ \ \ \ \ \\[4pt]
(2)\!\!\!& s_i L_i=L_{i+1} s_i\!-\!1,&\bar s_i \bar L_i=\bar L_{i+1} \bar s_i\!-\!1.\\[4pt]
(3)\!\!\!&(L_i\! +\!L_{i+1}) s_i\!=\!s_i(L_i\! +\!L_{i+1}),\!\!&(\bar L_i \!+\!\bar L_{i+1}) \bar s_i\!=\!\bar s_i(\bar L_i \!+\!\bar L_{i+1}).
\end{array}$$
%$L_i s_j=s_j L_i$ \ and \ $\bar L_i \bar s_j=\bar s_j \bar L_i$ \ \ if $i\neq j, j+1$.
%\item $ s_i L_i=L_{i+1} s_i-1$ \ and \  $\bar s_i \bar L_i=\bar L_{i+1} \bar s_i-1$.
%\item $(L_i +L_{i+1}) s_i=s_i(L_i +L_{i+1})$ \ and \  $(\bar L_i +\bar L_{i+1}) \bar s_i=\bar s_i(\bar L_i +\bar L_{i+1})$.
%\end{enumerate}
\end{lemma}

For convenience, we define the following cycles in $\mathfrak S_r$, where $1\le i,j\le r$,
\begin{equation}\label{s-ij}
s_{i, j}=s_is_{i+1}\cdots s_{j-1}=(j,j-1,...,i)\mbox{ \ for $i< j$},\end{equation} and $s_{i, i}=1$. If $i>j$, we set $s_{i, j}=s_{j, i}^{-1}=(j, j+1,...,i)$.
Similarly, for $1\le i,j\le t$, we define $\bar s_{i,j}=(\bar{j},\overline{j\!-\!1},...,\bar i)\in \bar{\mathfrak S}_t$ if $i<j$, or $1$ if $i=j$, or $\bar s_{j,i}^{^{\sc\,-1}}$ else.
Let $e_{i,j}$ be the element
whose corresponding diagram is the walled Brauer diagram such that any of its  edge  is of form $[k, k]$ or $[\bar k, \bar k]$
except two horizontal edges  $[i, \bar j]$  on both top and bottom rows.
Namely,
 \begin{equation}\label{e-ij}e_{i, j}= \bar s_{j, 1} s_{i, 1} e_1 s_{1, i} \bar s_{1, j}
\mbox{ \ for  $i, j$ with  $1\le i\le r $ and $1\le j\le t$.}
\end{equation}
We also simply denote $e_i=e_{i,i}$ for $1\le i\le\min\{r,t\}$.

It follows from \cite[Lemma 2.1]{BS} and \cite[Proposition~2.5]{RSong} that %the following  is a central element in $\mathscr B_{r, t}(\delta)$:
\begin{equation}\label{central}  c_{r, t} =\SUM{{1\le i\le r},\,{1\le j\le t}}{} e_{i, j}-\SUM{i=1}r L_i-\SUM{j=1}t \bar L_j,\end{equation}
is a central element in $\mathscr B_{r, t}(\delta)$.
 Such a central element
 has  already been used in \cite[Lemma~4.1]{CVDM} to study blocks of $\mathscr B_{r, t}(\delta)$ over $\mathbb C$.
%The following definition is
Motivated by (\ref{central}), we define {\it Juscy-Murphy-like elements} $y_i,\,\bar y_\ell$ below such that for any $k\in\Z^{\ge1}$, elements $y_{k+1},\,\bar y_{k+1}$ in the image of
the %algebraic
homomorphism $\phi_k$ (to be defined in Theorem \ref{alghom})
%$\mathscr B_{k+r,k+ t}(\delta)$
 will play the same roles as that of $x_1,\,\bar x_1$ in
 $\mathscr B^{\rm aff}_{r, t}$.

\begin{definition}\label{yi} Fix an element $\delta_1\in R$.   For  $1\le i\le r$ and $1\le  \ell\le t$, let
\begin{equation}\label{yi1} y_i=\delta_1+\SUM{j=1}{i-1} e_{i, j}-L_i, \text{ \ and \ }  \bar y_\ell=-\delta_1+\SUM{j=1}{\ell-1} e_{j, \ell}-\bar L_\ell.
\end{equation}
\end{definition}

\begin{lemma}\label{defrel} Let $i\in \mathbb Z$ with $1\le i\le \min\{r, t\}$. %, let $e_i=e_{i, i}$.
%\begin{enumerate}\item
$$\begin{array}{llll}
(1)\ \  e_i y_i=e_i(\delta_1+ \bar L_i-L_i),&e_i \bar y_i=e_i(-\delta_1+L_i-\bar L_i).\\[3pt]
%\item
(2)\ \  e_i (y_i+\bar y_i)=0,& (y_i+\bar y_i) e_i=0.\\[3pt]
%\item
(3)\ \  e_i s_i y_i s_i= s_i y_i s_i e_i,& e_{i} \bar s_i \bar y_i \bar s_i= \bar s_i \bar y_i \bar s_i e_i.\\[3pt]
%\item
(4)\ \  y_i(e_i+\bar y_i)=(e_i+\bar y_i) y_i.\\[3pt]
%\item
(5)\ \  y_i (s_iy_is_i\!-\!s_i)\!=\!(s_iy_is_i\!-\!s_i)y_i,\!\!&
\bar y_i (\bar s_i\bar y_i\bar s_i\!-\!\bar s_i)\!=\!(\bar s_i\bar y_i\bar s_i\!-\!\bar s_i)\bar y_i.\\[3pt]
%\item $
(6)\ \  s_j y_i=y_is_j,&\bar s_j \bar y_i=\bar y_i \bar s_j\mbox{ \ \ \ if \ \ \  $j\neq i\!-\!1,\, i$.}\ \ \ \ \ \ \ \
\ \ \ \ \ \ \ \ \ \ \\[3pt]
%\item $
(7)\ \  s_j \bar y_i=\bar y_i s_j,&\bar s_j y_i= y_i \bar s_j\mbox{ \ \ \ if \ \ \ $j\neq  i\!-\!1$.}\\[3pt]
%\item $
(8)\ \  e_{i+1} y_{i}=y_i e_{i+1},&e_{i+1} \bar y_{i}=\bar y_i e_{i+1}\mbox{ \ \ \ if \ \ \ $i<\min\{r,t\}$}.\\[3pt]
%\item $
(9)\ \  y_iy_{i+1}=y_{i+1}y_i,&\bar y_i\bar y_{i+1}=\bar y_{i+1}\bar y_i\mbox{ \ \ \ if \ \ \ $i<\min\{r,t\}$}.
\end{array}$$
%\end{enumerate}
\end{lemma}
\begin{proof} We remark that the second assertion of (2) follows from the first assertion of (2) by applying the anti-involution
 $\sigma$ in Lemma~\ref{anti1}. By Corollary~\ref{iso-1}, we need only check (4) and  the first assertions of others.

Since $e_i e_{i, j}=e_i (\bar j, \bar i)$ and  $e_i e_{ j, i}=e_i (j, i)$ for $ j\ne i$, we have (1) and (2).
 Further, (3) follows from the equalities $e_i e_{k, j}=e_{k, j} e_i$,  $e_i (k,j)=(k, j) e_i$  if $i\not\in\{k, j\}$ together with (\ref{sy}) as follows:
\begin{equation}\label{sy}  s_iy_is_i= \SUM{j=1}{i-1} e_{i+1, j}-\SUM{j=1}{i-1} (j, i+1)+\delta_1=y_{i+1}-s_ie_is_i+s_i.\end{equation}
By  Definition~\ref{yi}, we have
$$\begin{aligned} y_iy_{i+1}-&y_{i+1}y_i\\
= & \SUM{j=1}{i-1} e_{i, j}\SUM{k=1}i e_{i+1, k}-\SUM{k=1}i e_{i+1, k}\SUM{j=1}{i-1} e_{i, j}-\SUM{j=1}{i-1} e_{i, j} L_{i+1} +L_{i+1}\SUM{j=1}{i-1} e_{i, j}\\
 = & \SUM{j=1}{i-1} e_{i, j} s_i- s_i\SUM{j=1}{i-1} e_{i, j} -\SUM{j=1}{i-1} e_{i, j} s_i+s_i\SUM{j=1}{i-1} e_{i, j},
% \\
%                   =&\,0,\\
\end{aligned} $$which is equal to zero,
proving (9).

 Recall that $\sigma$ is the anti-involution on $\mathscr B_{r, t}(\delta)$ in Lemma~\ref{anti1}. We have $\sigma(y_j)=$ $y_j$ and $\sigma(s_j)=s_j$.
Using (\ref{sy}) and $\sigma$, we have
 \begin{equation} \label{sys}  y_i (s_iy_is_i-s_i) =y_iy_{i+1}-y_is_ie_is_i,\ \ \  (s_iy_is_i-s_i) y_i =y_{i+1} y_i- s_ie_is_i y_i.
 \end{equation}
By (3), we have  $y_is_ie_is_i=$ $s_ie_is_iy_{i}$. So, (5) follows from (9).

We remark that  (6) and (7) can be checked easily by using Theorem~\ref{wbmwf}$\sc\,$(2)--(4).
 Since $e_{i+1} e_{i, j} =e_{i, j}e_{i+1}$ and $e_{i+1} (j, i)=(j, i) e_{i+1}$ for %all $i,\,
$ 1\le j\le i\!-\!1$, we have (8).

Finally, we check (4).  We have  $(y_1+e_1)\bar y_{i+1}=\bar y_{i+1} (y_1+e_1)$ by  $e_1\bar y_{i+1}=\bar y_{i+1} e_1$. By induction on $j$, we have  $(y_j+e_j)\bar y_{i+1}=\bar y_{i+1} (y_j+e_j)$ and
$e_j \bar y_{i+1}=\bar y_{i+1} e_j$
for all $j$ with $ 1\le j\le i$. So,
\begin{equation} \label {yby} y_i \bar y_{i+1} =\bar y_{i+1} y_i.\end{equation} By (\ref{sy}) and Corollary~\ref{iso-1}, $e_i+\bar y_i=\bar s_{i}\bar y_{i+1} \bar s_i+\bar s_{i}$.
 So, (4) follows from (\ref{yby}) and (7).
\end{proof}

The following result is a special case of \cite[Proposition~2.1]{CVDM}.

\begin{proposition}\label{cvdm} Let $\mathscr B_{r, t}(\delta)$ be defined over a field $F$. For
% $k\in\Z^{[2,\min\{r,t\}]}$,
 $2\le k\le \min\{r, t\}$, let  $e=e_k$ if $\delta\neq 0$ or $e=e_k s_{k-1}$ otherwise.
  Let $\mathscr B_{k, k}(\delta) $ be the subalgebra of $\mathscr B_{r, t}(\delta)$ generated by $e_1, s_i, \bar s_i, 1\le i\le k$. Then
 $e \mathscr B_{k, k}(\delta) e=e \mathscr B_{k-1, k-1}(\delta)$,  which is isomorphic to $\mathscr B_{k-1, k-1}(\delta)$ as an $F$-algebra.
\end{proposition}

We remark that we are assuming  $\delta=\omega_0\ne0$. The following result immediately follows from Proposition \ref{cvdm}, where
elements $\omega_{a, k},\bar{\omega}_{a, k}$ will be crucial in obtaining the homomorphisms $\phi_k$
in Theorem \ref{alghom}.
%In the remaining part of this section, unless otherwise indicated, we always assume  $\delta\neq 0$.

\begin{corollary}\label{eye} For $a\in \Z^{\ge0}$,
there exist unique  $\omega_{a, k},\bar{\omega}_{a, k}\in \mathscr B_{k-1, k-1}$ such
 that
$$e_k y_k^a e_k=\omega_{a, k} e_k,\ \ \ e_k \bar y_k^a e_k=\bar{\omega}_{a, k} e_k.$$
Furthermore, $\omega_{1, k}=-\bar {\omega}_{1, k}= \delta\delta_1$
and $\omega_{0, k}=\bar {\omega}_{0, k}= \delta$.
\end{corollary}

\begin{lemma}\label{y1}For~any~$k\!\in\! {\mathbb {Z}}^{\ge 1}$,~we~have~$e_i \bar y_i^k\!=\!\sum_{j=0}^k a^{(i)}_{k, j} e_i y_i^j$~for~some~$a_{k,j}^{(i)}\!\in\!\mathscr B_{r,t}$~such~that
\begin{enumerate}\item $a_{k, k}^{(i)}=(-1)^k$,\vskip2pt
 \item $a_{k, j}^{(i)}=\omega_{0,i} a^{(i)}_{k-1, j}-a^{(i)}_{k-1, j-1}$, $1\le j\le k-1$,\vskip2pt
 \item  $a^{(i)}_{k, 0}=-\sum_{j=1}^{k-1} a^{(i)}_{k-1, j} \omega_{j, i}$.\end{enumerate}
  In particular,   $a^{(i)}_{k,j}\in R[\omega_{2, i}, \omega_{3,i} \ldots, \omega_{k-1, i}]$ for any $j$ with $1\le j\le k$ such that each monomial of $a^{(i)}_{k,j}$ is of form $\omega_{j_1, i}\cdots \omega_{j_\ell, i}$ with $\sum_{i=1}^\ell j_i\le k-1$.
\end{lemma}

\begin{proof} By Lemma~\ref{defrel}${\sc\,}$(2),  the result holds for $k=1$.  In general, by Lemma~\ref{defrel}${\sc\,}$(4),
$$e_i y_i^j \bar y_i=e_{i}(e_i-y_i) y_i^j-\omega_{j,i} e_i=\omega_{0, i} e_iy_i^{j}-e_iy_i^{j+1}-\omega_{j, i} e_i.$$ Now, the result follows from induction on $k$.
\end{proof}

%Recall the definitions of  $y_k,\bar y_k$,  $1\le k\le \min\{r, t\}$ in  Definition~\ref{yi}. By Corollary~\ref{eye},   there exist unique  $\omega_{a, k}, \bar{\omega}_{a, k}\in
%\mathscr B_{k-1, k-1}(\omega_0)$ such that $$e_k y_k^a e_k=\omega_{a, k} e_k,\ \ \ \ e_k \bar y_k^a e_k=\bar \omega_{a, k} e_k.$$
\begin{lemma}\label{free1} For $k, a\in\Z^{\ge2}$, we have $\bar \omega_{a, k}\in $ $ R[\omega_{2, k}, \omega_{3, k}, \cdots, \omega_{a, k}]$. Furthermore, both  $\omega_{a, k}$
and $\bar\omega_{a, k}$ are central in $\mathscr B_{k-1, k-1}$.\end{lemma}

\begin{proof} The first assertion follows from  Lemma~\ref{y1}. To
prove the second, note %So, we need only prove $\omega_{a, k}$ is central. Note
that any $h\in
\{e_1, s_i\mid  1\le i\le k-2\} $ commutes with $e_k, y_k$. So, $e_k (h \omega_{a, k})=e_k(\omega_{a, k} h)$. By Proposition~\ref{cvdm}, $h \omega_{a, k}=\omega_{a, k} h$.
Finally, we need to check $e_k (h \omega_{a, k})=e_k(\omega_{a, k} h)$ for any $h\in \{\bar s_1, \bar s_2, \cdots, \bar s_{k-2}\}$. In this case, we use Lemma~\ref{y1}.
More explicitly, we can use $\bar y_k$ instead of $y_k$
in $e_k y_k^a e_k$. Therefore,  $h \omega_{a, k}=\omega_{a, k} h$, as required.
\end{proof}

The following result follows from (\ref{sy}) and  induction on $a$.
\begin{lemma}\label{syk} For $k,a\in\Z^{\ge 1}$, we have $$s_k y_{k+1}^a=(y_k+e_k)^a s_k-\SUM{b=0}{a-1} (y_k+e_k)^{a-1-b}y_{k+1}^b.$$
\end{lemma}

The elements $z_{j, k},\,\bar z_{j, k}$ defined below will be crucial in the description of $\omega_{a, k}$ (cf. Lemma~\ref{omed}).
 For $1\le j\le k-1$, let
\begin{equation}\label{z-bar-z} z_{j, k}=s_{j, k-1} (y_{k-1}+e_{k-1})s_{k-1, j},\ \ \
\bar z_{j, k}=\bar s_{j, k-1} (\bar y_{k-1}+e_{k-1})\bar s_{k-1, j}.\end{equation}
Then the following result can be  verified, easily.
%Recall that $s_{j, k}=s_js_{j+1}\cdots s_{k-1}$ for $j< k$ and $s_{j, k}=s_{k, j}^{-1}$ if $j>k$ and $s_{k, k}=1$.
\begin{lemma}\label{dzjk}  For $1\le j\le k-1$, we have
 \begin{enumerate} \item   $ z_{j, k}=\sum_{\ell=1}^{k-1} e_{j, \ell}-\sum_{1\le s\le k-1, s\neq j}  (s, j)$,\vskip2pt
\item $\bar z_{j, k}=\sum_{\ell=1}^{k-1} e_{ \ell, j}-\sum_{\bar 1\le \bar s\le \overline{ k-1}, \bar s\neq \bar j}  (\bar s, \bar j)$.
\end{enumerate} \end{lemma}

Note that $\omega_{0,k}=\delta$ and $\omega_{1, k}=\delta\delta_1$, and $e_k h=0$ for $h\in \mathscr B_{k-1, k-1}$ if and only if $h=0$. We will use this fact  freely in the proof of
the following lemma, where we use the terminology that
a monomial in  $z_{j, k+1}$'s and $\bar z_{j, k+1}$'s
is a {\it leading term} in an expression if it has
the highest degree by defining ${\rm deg\,}z_{i,j}={\rm deg\,}\bar z_{i,j}=1$.

\begin{lemma}\label{omed} Suppose $a\in\Z^{\ge 2}$. Then  $\omega_{a, k+1}$ can be written as an $R$-linear combination of
monomials in  $ z_{j, k+1}$'s and $\bar z_{j, k+1}$'s for $1\le j\le k$ such that the leading terms
of  $\omega_{a, k+1}$ are
 $\sum_{j=1}^{k}( -z_{j, k+1}^{a-1}+(-1)^{a -1}\bar z_{ j, k+1}^{a-1})$.
% with $0\le \ell\le a-1$. ??What is $\ell$???
\end{lemma}

\begin{proof}  By  Corollary \ref{eye} and Lemma~\ref{defrel}$\sc\,$(1),  we have
\begin{equation}\label{ommmm}
\omega_{a,k+1}e_{k+1}=e_{k+1} y_{k+1}^a e_{k+1}=e_{k+1} (\bar L_{k+1}- L_{k+1}) y_{k+1}^{a-1}e_{k+1}.\end{equation}
Considering the right-hand side of \eqref{ommmm} and expressing $L_{k+1}$ by \eqref{L-barL}, using $(j,k+1)=$ $s_{j,k}s_ks_{k,j}$ (cf.~\eqref{s-ij}) and the fact that $s_{j,k},s_{k,j}$ commute with $y_{k+1},e_{k+1}$ (cf.~\eqref{e-ij} and Lemma~\ref{defrel}$\sc\,$(6)), we see that a term in the right-hand side of \eqref{ommmm} becomes
\begin{equation}\label{Expreee}{\sc\!}-s_{j, k} e_{k+1} s_{k}y_{k+1}^{a-1} e_{k+1} s_{ k, j}\!=\!s_{j, k} e_{k+1}\Big(\!{\sc\!}-(y_k\!+\!e_k)^{a-1}{\sc\!}\!+\!\SUM{b=0}{a-2}(y_k\!+\!e_k)^{a-2-b} \omega_{b,k+1}\!\Big) s_{k, j},\!\!\end{equation}
where the equality follows from Lemma \ref{syk} and Corollary \ref{eye}.
By Lemmas~\ref{free1}, $\omega_{b,k+1}$ commutes with $s_{k, j}$.
%is central in $\mathscr B_{k, k}$. So, $\omega_{k+1, b} =s_{k, j}\omega_{k+1, b}$.
Now by induction assumption, the right-hand side of \eqref{Expreee}
%$\sum_{j=1}^k e_{k+1} (j, k+1) y_{k+1}^{a-1}  e_{k+1}$
can be written as an $R$-linear combination of monomials with the required form such that
 the leading term is $-z_{j, k+1}^{a-1}$.

 Now we consider terms in \eqref{ommmm} concerning $\bar L_{k+1}$, namely we need to deal with  $e_{k+1} (\bar j, \overline{ k\!+\!1}) y_{k+1}^{a-1}  e_{k+1}$. We remark that it is  hard to compute
   it directly. However, by Lemma~\ref{y1} and induction on $a$, we can use
   $(-1)^{a-1} \bar y_{k+1}^{a-1} e_{k+1}$ to replace
$y_{k+1}^{a-1} e_{k+1}$ in $e_{k+1}  (\bar j, \overline{k\!+\!1})  y_{k+1}^{a-1} e_{k+1}$ (by forgetting lower terms). This enables us to consider  $(-1)^{a-1} e_{k+1} (\overline j,
 \overline {k\!+\!1})\bar y_{k+1}^{a-1}  e_{k+1}$ instead. As above, this term
can be written as the required form with leading term $(-1)^{a-1} \bar z_{j, k+1}^{a-1}$. The proof is completed.
\end{proof}

%\vskip4pt

\begin{lemma}\label{commomega}For~$a\!\in\!\Z^{\ge0},{\ssc\,}k\!\in\!\Z^{\ge 1}$,~both~$\omega_{a, k+1}$~and~%
$\bar \omega_{a, k+1}$~commute~with~$y_{k+1}$~and~$\bar y_{k+1}$.\end{lemma}
 \begin{proof} %In order to verify that $\omega_{a, k+1}$ commutates with $y_{k+1}$, b
 By Corollary~\ref{iso-1},  Lemmas~\ref{free1} and \ref{omed},  it suffices to prove that both  $z_{j, k+1}$ and $\bar z_{j, k+1}$ for $1\le j\le k$, commute with $y_{k+1}$. By
Lemma~\ref{defrel}{$\sc\,$}(9) and
$y_{k+1} e_k=e_k y_{k+1}$, we have $y_{k+1} (e_k+y_k)=(e_{k}+y_k) y_{k+1}$. Note that $z_{k, k+1}=y_k+e_k$, we have  $y_{k+1}z_{k, k+1}=$ $z_{k, k+1} y_{k+1}$.
In general, by Lemma~\ref{defrel}$\sc\,$(6),
  $y_{k+1}z_{j, k+1}=z_{j, k+1} y_{k+1}$.
  By (\ref{yby}) and Corollary~\ref{iso-1}, $y_{k+1} \bar y_k=\bar y_k y_{k+1}$.  Since $\bar y_k+e_k= \bar z_{k, k+1} $ (cf.~(\ref{z-bar-z})),  $y_{k+1}\bar z_{k, k+1}=$ $\bar z_{k, k+1} y_{k+1}$.
 So, by Lemma~\ref{defrel}{$\sc\,$}(7),   $y_{k+1}\bar z_{j, k+1}=\bar z_{j, k+1} y_{k+1}$. The result follows.
 % proof is completed.
% and the above arguments,
% $\omega_{a, k+1}$ commutates with $\bar y_{k+1}$.
 \end{proof}

The following  is the main result of this section. It follows from Theorem~\ref{wbmwf}, Lemmas~\ref{defrel},~\ref{commomega} and~ Corollary~\ref{eye}.

\begin{theorem}  \label{alghom} Let $F$ be a field containing $\omega_0, \omega_1$ with $\omega_0\neq 0$.
For any $k\in \mathbb Z^{>0}$, let $\mathscr B_{r+k, t+k}(\omega_0)$ be the walled Brauer algebra over $F$.
Then  there is an $F$-algebraic homomorphism
$\phi_k: \mathscr B_{r,t}^{\text{aff}}\rightarrow \mathscr B_{r+k,t+k}(\omega_0)$ sending
\begin{equation}\label{Auto-m}s_i, \, \bar s_j,\, e_1,\,x_1,\,\bar x_1,\,\omega_a,\,\bar {\omega}_a
 \mbox{ \,\,$\mapsto$ \,\,}s_{i+k},\,\bar s_{j+k},\,e_{k+1},\,y_{k+1},
  \,\bar y_{k+1},\,\omega_{a, k+1},\,\bar \omega_{a, k+1},\end{equation} respectively such that $\delta_1=\omega_0^{-1}\omega_1$.
\end{theorem}

%\begin{proof} The result .
%   $\omega_{a, k+1}$ commutes with $\phi(x_1)$ and $\phi(\bar x_1)$. In general, since
% $\omega_{a, k+1}\in \mathscr B_{k, k}$, $\omega_{a, k+1}$ commutes with $s_{k+b}$ and $\bar s_{k+c}$ for $b, c\ge 1$. So, $\omega_{a, k+1}$ commutes with $\phi(x_i)$ and $\phi(\bar x_j)$.
%In other words, $\omega_{a, k+1}$ is central in $\phi(\mathscr {B}_{r, t})$ for any $a\in \mathbb N$.
%\end{proof}

\section{A basis of an affine walled Brauer algebra}\label{freeness}
Throughout this section, we assume that $R$ is a commutative ring containing $1$, $\omega_0$ and $\omega_1$.  The main purpose of this section is to prove
 that $\mathscr B_{r,t}^{\text{aff}}$ is free over  $R$ with infinite rank.

\begin{lemma}\label{antiaff}
There is an $R$-linear anti-involution $\sigma: \mathscr B_{r,t}^{\text{aff}}\rightarrow \mathscr B_{r,t}^{\text{aff}}$
fixing defining generators $s_i, \bar s_j, e_1, x_1, \bar x_1$, $\omega_a$ and $\bar \omega_b$ for all possible $a, b, i, j$'s.
\end{lemma}
\begin{proof} This follows from  the symmetry of the defining  relations in Definition~\ref{wbmw} (cf.~Lemma \ref{anti1}).
\end{proof}

The following can be proven by arguments similar to those for Lemma \ref{y1}.

\begin{lemma}\label{x1} For any $k\in {\mathbb {Z}}^{\ge 1}$, we have $e_1 \bar x_1^k=\sum_{i=0}^k a_{k, i} e_1 x_1^i$ for some $a_{k,i}\in\B_{r,t}^{\text{aff}}$ such that
\begin{enumerate}\item $a_{k, k}=(-1)^k$,\vskip2pt
 \item $a_{k, i}=\omega_0 a_{k-1, i}-a_{k-1, i-1}$, $1\le i\le k-1$,\vskip2pt
 \item  $a_{k, 0}=-\sum_{i=1}^{k-1} a_{k-1, i} \omega_{i}$.\end{enumerate}
  In particular,   $a_{k,i}\in R[\omega_2, \omega_3, \ldots, \omega_{k-1}]$
 for all $i$ with $1\le i\le k$ such that each monomial of $a_{k,i}$ is of form $\omega_{j_1}\cdots \omega_{j_\ell}$ with $\sum_{i=1}^\ell j_i\le a-1$.
\end{lemma}

\begin{corollary} \label{bomega}
Assume $e_1$ is $R[\omega_2, \omega_3, \cdots,\bar\omega_0,\bar\omega_1,\cdots]$-torsion-free. Then
$\bar \omega_{0}=$ $\omega_{0}$, $\bar \omega_{1}=-\omega_{1}$ and
 $\bar \omega_{k} \in R[\omega_{2}, \omega_{3}, \cdots, \omega_{k}]$ for $k\ge2$.
\end{corollary}
\begin{proof} Applying $e_1$ on the right hand side of $e_1\bar x_1^k$ and using Lemma~\ref{x1} yield the result as required.
\end{proof}

\begin{remark} By Corollary~\ref{bomega}, $\mathscr B_{r, t}^{\rm {aff}}$ can be generated by $s_i, \bar s_j, e_1, x_1, \bar x_1,\omega_a$ for all possible $i, j,a$ if
$e_1$ is $R[\omega_2, \omega_3, \cdots,\bar\omega_0,\bar\omega_1,\cdots]$-torsion-free.
In fact, when we prove the freeness of $\mathscr B_{r, t}^{\rm {aff}}$, we do not need to assume that
$e_1$ is $R[\omega_2, \omega_3, \cdots,\bar\omega_0,\bar\omega_1,\cdots]$-torsion-free. What we need is that $\bar \omega_{k} $'s
  are determined by $\omega_2, \cdots, \omega_k$ and $a_{k, i}$ in Lemma~\ref{x1}.  If so,   $\mathscr B_{r, t}^{\rm {aff}}$ is free over $R$, forcing   $e_1$ to be
  $R[\omega_2, \omega_3, \cdots,]$-torsion-free, automatically.\end{remark}

%As stated in Remark \ref{remark-spe}, since we shall be mainly interested in
%the case when $\omega_a$'s are specialized (cf.~Theorem \ref{main2}),
{{In the remaining part of this paper,  we always keep this reasonable assumption on $e_1$}}.~So,~$\mathscr B_{r, t}^{\rm {aff}}$~is~generated by~$s_i, \bar s_j, e_1, x_1, \bar x_1,\omega_a$~for~all~possible~$i, j,a$'s.

%In fact, we will prove that $\mathscr  B_{r, t}^{\rm {aff}}$ is free over $R$. This  implies that $e_1$ is
%$R[\omega_2, \omega_3, \cdots,\bar\omega_0,\bar\omega_1,\cdots]$-torsion-free. So, the assumption on $e_1$ in Corollary~{\rm\ref{bomega}} can be removed.(??? I think we can not remove this
%assumption)
The elements defined below will play similar roles to that of $x_1$ and $\bar x_1$:
\begin{eqnarray}\label{X-bar-x}
x_i=s_{i-1} x_{i-1} s_{i-1}-s_{i-1},\ \ \ \bar x_j=\bar s_{j-1} \bar x_{j-1} \bar s_{j-1}-\bar s_{j-1},
\end{eqnarray}
for  $2\le i\le r,\ 2\le j\le t$.
The following result can be checked easily.
 \begin{lemma}\label{sx}We have
   \begin{enumerate} \item $ s_i x_i=x_{i+1}s_i+1$, \ \ \ $x_i x_j=x_jx_i$ \ \ for \ \ $1\le i< j\le r$.\vskip2pt
  \item $\bar s_i\bar x_i= \bar x_{i+1}\bar s_i+1$, \ \ \ $\bar x_i \bar x_j=\bar x_j\bar x_i$ \ \ for \ \ $1\le i< j\le t$.
  \item Let $\phi_k: \mathscr B_{r,t}^{\text{aff}}\rightarrow \mathscr B_{r+k,t+k}(\omega_0)$ be the homomorphism in Theorem~{\rm\ref{alghom}}. Then $($recall notation $e_{i,j}$ in \eqref{e-ij}$)$
  \begin{enumerate}\item[\rm(i)]  $\phi_k(x_\ell)=\sum_{j=1}^k e_{k+\ell, j} -L_{k+\ell}+\omega_0^{-1}\omega_1 $, \vskip2pt \item[\rm(ii)] $\phi_k(\bar x_\ell)=\sum_{j=1}^k e_{j, k+\ell} -\bar L_{k+\ell}-\omega_0^{-1}\omega_1$. \end{enumerate}
  \end{enumerate}
  \end{lemma}

\begin{lemma}\label{comm}  For  $1\le i\le r$ and $1\le j\le t$, we have
\begin{enumerate}
\item $x_i  (\bar x_j \!+\! e_{i, j} )\!=\! (\bar x_j\!+\! e_{i, j} )x_i $, \ \ \  \ \ $\bar x_j    (x_i \!+\! e_{ i, j} )\!=\!(x_i\! +\! e_{ i, j} )\bar x_j $.\vskip2pt
%\item $x_1 (\bar x_j + e_{1, j} )=(\bar x_j + e_{1, j} ) x_1$ \ and \
% $\bar x_1 (x_i+ e_{ i, 1} )=(x_i + e_{ i, 1} ) \bar x_1$,
\item $e_{ i, j} (x_i\!+\!\bar x_j) \!=\!-e_{i, j} (\bar L_j \!+\!L_i ) $, \ \
 $ (x_i\!+\!\bar x_j) e_{i, j} \! =\! -( \bar L_j \!+\!L_i ) e_{i, j}$.
\end{enumerate}
\end{lemma}
\begin{proof}  By symmetry  and Lemma \ref{antiaff}, we need only check the first assertions of (1)--(2). In fact,
  if $i=1$,  then (1) follows from  Definition~\ref{wbmw}{$\sc\,$}(\ref{(21)}),\,(\ref{(24)}). In general,  it follows from  induction on $i$.
By  Definition~\ref{wbmw}{$\sc\,$}(\ref{(8)}),  $e_{1, 2} (x_1+\bar x_2)=-e_{1,2} \bar L_2$.  Using Definition~\ref{wbmw}{$\sc\,$}(24), and induction on $j$ yields
$e_{1, j} (x_1+\bar x_j)=-e_{1,j} \bar L_j$. This is (2) for $i=1$. The general case follows from induction on $i$. \end{proof}

\begin{lemma} \label{comm1} Suppose $1\le i,  j\le r$ and $1\le k, \ell\le t$.
\begin{enumerate} \item If $i\ne j$, then $e_{i, k} (x_j+L_j)=(x_j+L_j) e_{i, k}$.\vskip2pt
\item If $k\ne \ell $, then $ e_{i, k } (\bar x_{\ell}+\bar L_{\ell})=(\bar x_{\ell}+\bar L_{\ell}) e_{i, k}$.
\end{enumerate}
\end{lemma}
\begin{proof} By  Corollary~\ref{iso-1},  we need only to check (1).
 By Definition~\ref{wbmw}{$\sc\,$}{$\ssc\,$}(\ref{(9)}), we have $e_1(x_2+L_2)=(L_2+x_2) e_1$. Using induction on $j$ yields
$e_1(x_j+L_j)=(x_j+L_j) e_1$ for $j\ge 3$. This is (1) for $i=k=1$. By induction on $k$,
$e_{1, k}(x_j+L_j)=(x_j+L_j) e_{1, k}$.
If  $i<j$, by Lemmas~\ref{sl} and \ref{sx},  we have  $e_{i,k} (x_j+L_j)=(x_j+L_j) e_{i, k}$.

In order to prove (1) for $i>j$, we need $e_{ 2, 1} x_1=x_1 e_{2, 1}$, which follows from  Definition~\ref{wbmw}{$\ssc\,$}(9).
By  Definition~\ref{wbmw}{$\ssc\,$}(4),\,(24), we have  $e_{ i, k} x_1=x_1 e_{i, k}$.
 By induction on $j$, we have  $e_{ i, k}( x_j+L_j) =(x_j+L_j) e_{i,k}$ for all $j$ with $j<i$, as required.
\end{proof}

\begin{lemma} \label{ex} Suppose  $ 1\le i\le r-1$ and $1\le j\le t-1$.
\begin{enumerate}\item $s_i(x_i\!+\!L_i)=(x_{i+1}\!+\!L_{i+1}) s_i$, \ \
%\item
$\bar s_j(\bar x_j+\bar L_j)=(\bar x_{j+1}+\bar L_{j+1}) \bar s_i$.\vskip2pt
\item $e_{i, j} (x_i+L_i)^a e_{i, j}=\omega_a e_{i, j}$, \ \ \  \
%for all $a\in \mathbb Z^{\ge0}$. \item
\ $ {\ssc\,}e_{i, j} (\bar x_j+\bar L_j)^a e_{i, j}=\bar \omega_a e_{i, j}$ \ \ for \ \ $a\in \mathbb Z^{\ge0}$.
\end{enumerate}
\end{lemma}
\begin{proof} By Corollary~\ref{iso-1}, it suffices to check the first assertions of  (1) and (2). We remark that (1) follows from Lemmas~\ref{sl} and ~\ref{sx}, and (2) follows from (1) together with induction on $i$. \end{proof}

We consider $\mathscr B_{r,t}^{\text{aff}}$ as a filtrated algebra defined as follows. Set
$$ \text{ $\text{deg}{\sc\,} s_i=$ $\text{deg}{\sc\,} \bar s_j=\text{deg}{\sc\,} e_1=\text{deg}{\sc\,}  \omega_a=0$ \ and \
$\text{deg}{\sc\,}{x_k}=\text{deg}{\sc\,} \bar x_\ell=1$},$$ for all possible $a, i, j, k, \ell$'s. Let $(\mathscr B_{r,t}^{\text{aff}})^{(k)}$ be the  $R$-submodule
 spanned by monomials with degrees  less than or equal to $k$ for  $k\in{\mathbb{Z}}^{\ge0}$. Then we have the following filtration
  \begin{equation}\label{filtr}
\mathscr B_{r,t}^{\text{aff}}\supset\cdots\supset (\mathscr B_{r,t}^{\text{aff}})^{(1)}\supset(\mathscr B_{r,t}^{\text{aff}})^{(0)}\supset (\mathscr B_{r,t}^{\text{aff}})^{(-1)}=0.\end{equation}
Let ${\rm gr} ( \mathscr B_{r,t}^{\text{aff}})\!=\!\oplus_{i\ge0}( \mathscr B_{r,t}^{\text{aff}})^{[i]}$, where
$( \mathscr B_{r,t}^{\text{aff}})^{[i]}\!=\!( \mathscr B_{r,t}^{\text{aff}})^{(i)}/( \mathscr B_{r,t}^{\text{aff}})^{(i-1)}$. Then
  ${\rm gr} ( \mathscr B_{r,t}^{\text{aff}})$ is a $\mathbb Z$-graded algebra associated to $\mathscr B_{r,t}^{\text{aff}}$.~We use the same symbols to denote elements in
${\rm gr} ( \mathscr B_{r,t}^{\text{aff}})$.~%
%
%\begin{lemma} \label{grade}In ${\rm gr} ( \mathscr B_{r,t}^{\text{\rm {aff}}})$, we have
%\begin{enumerate}\item $x_k \bar x_\ell=\bar x_\ell x_k$,
%\item $e_k x_k=-e_k \bar x_k$,
%\item $e_k x_\ell =x_\ell e_k$, $e_k \bar x_\ell =x_\ell e_k$ if $k\ne \ell$.
%\item $s_i x_{i}=x_{i+1}s_i$, $\bar s_i \bar x_{i}=\bar x_{i+1}\bar s_i$.
%\item $s_i \bar x_j=\bar x_j s_i$ and $\bar s_i x_j=x_j\bar s_i$.
%\end{enumerate}
%\end{lemma}
%\begin{proof}
%The results follow from  Lemmas~\ref{sx}--\ref{comm} and~\ref{ex}, immediately. \end{proof}
%
We remark that we will work with  ${\rm gr} ( \mathscr B_{r,t}^{\text{aff}})$ when we prove the freeness of~%
$\mathscr B_{r,t}^{\text{aff}}$.

Fix $r, t, f\in \mathbb Z^{>0}$ with $f\le \min \{r, t\}$. We define the following subgroups of
$\mathfrak S_r$, $\mathfrak{S}_r\times\bar{\mathfrak S}_t$ and $\bar{\mathfrak S}_t$ respectively,
\begin{eqnarray}\label{S-f--}
&&\mathfrak S_{r-f}=\langle s_j\,|\, f\!+\!1\le j< r\rangle,
%\mbox{ the subgroup of generated by $ $,}
\nonumber\\
%\mbox{ the subgroup of  generated by $ $,}
&&\mathfrak {G}_f\ \ \ =\langle\bar s_{i} s_{i}\,|\,1\le i< f\rangle,
\nonumber\\
&&\bar{\mathfrak S}_{t-f}=\langle\bar s_j\,|\, f\!+\!1\le j< t\rangle
.%\mbox{ the subgroup of generated by  $ $.}
\end{eqnarray}
Observe that $\mathfrak {G}_f$ is isomorphic to the symmetric group in $f$ letters.
%thus we use this symbol without confusion.
The following result has been given in \cite{RSong} without
%giving
a detailed proof.
 We remark that  $\mathscr{D}_{r,t}^f$ in (\ref{rcs11})
 was defined
in \cite[Proposition~6.1]{enyang} via certain {\it row-standard tableaux}.
\begin{lemma}\label{rcs}\cite[Lemma~2.6]{RSong}
The following $($recall notation $s_{i,j}$ in \eqref{s-ij}$)$  \begin{equation}\label{rcs11}  \mathscr{D}_{r,t}^f=\{ s_{f,i_f} \bar s_{f, j_f} \cdots
s_{1,i_1}\bar s_{1,{j_1}}\,|\,
 1 {\sc\!}\le{\sc\!} i_1{\sc\!}<{\sc\!} \cdots{\sc\!} <{\sc\!}i_f\le r,\,k {\sc\!}\le{\sc\!} {j_k}\},\end{equation}
is  a complete set of
right coset representatives for
$\mathfrak{S}_{r-f}\times\mathfrak{G}_f\times\bar{\mathfrak{S}}_{t-f}$ in
$\mathfrak{S}_r\times\bar{\mathfrak{S}}_t$.
 \end{lemma}

\begin{proof}We denote  by $\tilde {\mathscr{D}}_{r,t}^f$ the right-hand side of \eqref{rcs11}, and by $\mathscr{D}_{r,t}^f$ a complete set of
right coset representatives. Then obviously $\tilde {\mathscr{D}}_{r,t}^f\subset\mathscr{D}_{r,t}^f$.
In order to verify the inverse inclusion, it  suffices to prove that $|\tilde {\mathscr{D}}_{r,t}^f|$, the cardinality of $\tilde {\mathscr{D}}_{r,t}^f$, is
$\frac{r!t!}{(r - f)!(t - f)!f!}=C^{f}_rC^{f}_tf!$, which is clearly the cardinality of  $\mathscr{D}_{r,t}^f$, where $C^f_r$ is the binomial number.  This will be done by induction on $f$ as follows.

If $f=0$, there is nothing to be proven.
Assume $f\ge1$. For any element in \eqref{rcs11}, we have
$i_f\ge f$. For each fixed $i:=i_f$, there are  $t-f+1$ choices of $j_f$ with $j_f\ge f$, and further, conditions for other indices are simply conditions for $\mathscr{D}_{i-1,t}^{f - 1}$.
So,
$$\begin{array}{llll}
|\tilde {\mathscr{D}}_{r,t}^f| \!\!\!&= (t \!-\! f\!+\! 1)\sum\limits_{i=f}^{r}|\mathscr{D}_{i-1,t}^{f - 1} |
\\[11pt]&=
(t\!-\!f\!+\!1)\sum\limits_{i=f}^{r}C_{i-1}^{f-1}C_t^{f-1}(f\!-\!1)!
=\sum\limits_{i=f}^{r}C_{i-1}^{f-1}C_t^ff!=C_r^fC_t^ff!,
\end{array}$$
where the second equality follows from induction assumption on $f$, and the last follows from the well-known combinatorics formula $C_r^i=C_{r-1}^i+C_{r-1}^{i-1}$.
\end{proof}

%If $f=0$, there is noting to be proven. If $f=1$, then  $|\tilde {\mathscr{D}}_{r,t}^1|=rs$, which is the cardinality of  $\mathscr{D}_{r,t}^1$.  In general,
%we have  $i_f\in \{f, f+1, \cdots, r\}$. Since $j_f\ge f$,
%there are $r-f+1$ choices of $j_f$ with $j_f\ge f$. So,
%\begin{equation}\label{eqcoset}  |\tilde {\mathscr{D}}_{r,t}^f|=(r-f+1)\SUM{j=1}{r-f+1} |\tilde {\mathscr{D}}_{r-j,t}^{f-1}|.\end{equation}
%Simplifying the right-hand side of (\ref{eqcoset}) via induction assumption on $r$ yields the result, as required. \end{proof}

We denote \begin{equation}\label{e-f===}e^f=e_1e_2\cdots e_f\mbox{ \ for any $f$ with $1\le f\le\min\{r, t\}$}.\end{equation}
If $f=0$, we set $e^0=1$. In \cite[Theorem~6.13]{enyang}, Enyang constructed a cellular basis for  $q$-walled Brauer algebras.
The following result follows from this result immediately.

\begin{theorem}\label{basis} \cite{enyang}
The following is  an $R$-basis of  $\mathscr B_{r, t}(\omega_0) $,
$$\mathscr M=\{c^{-1} e^f w d\mid 1\le f\le  \min\{r, t\},\,
 w\in \mathfrak S_{r-f} \times \bar{\mathfrak S}_{t-f},\, c,d\in  \mathscr{D}_{r,t}^f\}.$$ \end{theorem}

\begin{definition}\label{def-mo}We say that
\begin{equation}\label{Monooo}\textit{\textbf{m}}:=\mbox{$\prod\limits_{i=1}^r x_i^{\alpha_i}  c^{-1} e^f w d \prod\limits_{j=1}^t \bar x_j^{\beta_j}\prod\limits_{k\in \mathbb Z^{\ge 2} } \omega_k^{a_k}  $}\end{equation} is a {\it regular monomial} if $c,d\in \mathscr {D}_{r,t}^f$,   $\alpha_i, \beta_j\in \mathbb Z^{\ge0}$
and $a_k\in \mathbb Z^{\ge0}$ for $k\ge 2$ such that $a_k=0$ for all but a finite many $k$'s.
%only finite  many $a_k$'s are not zero.
\end{definition}

\begin{proposition}\label{basis1}   Suppose  $R$ is a commutative  ring  which contains $ 1$, $\omega_0$, $\omega_1$. As an $R$-module,  $\mathscr B_{r,t}^{\text{aff}}$ is spanned by all regular monomials. \end{proposition}

\begin{proof} Let $M$ be the $R$-submodule of $ \mathscr B_{r,t}^{\text{aff}}$ spanned by all regular monomials $\textit{\textbf{m}}\in \mathscr B_{r,t}^{\text{aff}}$ given in (\ref{Monooo}). We want to prove
   \begin{equation} \label{veri}h\textit{\textbf{m}}= h \prod_{i=1}^r x_i^{\alpha_i}  c^{-1} e^f w d \prod_{i=1}^t \bar x_i^{\beta_i} \mbox{$\prod\limits_{i\in \mathbb Z^{\ge 2} }$} \omega_i^{a_i}\in M\mbox{ \ \ for any generator $h$ of $ \mathscr B_{r,t}^{\text{aff}}$}.\end{equation}
 If so,
then  $M$ is a left  $ \mathscr B_{r,t}^{\text{aff}}$-module, and thus $M=\mathscr B_{r,t}^{\text{aff}}$ by the fact that $1\in M$.

 We  prove (\ref{veri})  by induction on  $|\alpha|:=\sum_{i=1}^r \alpha_i$.   If  $|\alpha|=0$, i.e.,  $\alpha_i=0$ for all possible $i$'s,
then  by Theorem~\ref{basis}, we have (\ref{veri}) unless  $h= \bar x_1$.

 If $h=\bar x_1$, by Lemma~\ref{sx}, we need to compute $\bar x_k e^f $ for all $k$ with $1\le k\le t$. If $k\in \{1,2, \cdots, f\}$, by Lemma~\ref{comm}{$\sc\,$}(3), we can use $-x_k$ instead of $\bar x_k$. So, $h{\textit{\textbf{m}}}\in M$. Otherwise, by Lemma~\ref{comm1}{$\sc\,$}(2), we can use  $e^f \bar x_k $
 instead of $\bar x_k e^f $.  So, (\ref{veri}) follows from Lemma~\ref{sx} and Theorem~\ref{basis}.

Suppose $|\alpha|>0$. By Lemma~\ref{sx} and Theorem~\ref{basis}, we see that (\ref{veri}) holds  for $h\in$ $\{s_1, \cdots, s_{r-1}, \bar s_1, \cdots, \bar s_{t-1}, x_1\}$. If $h=\bar x_1$, then (\ref{veri})
  follows from Lemma~\ref{comm}{$\sc\,$}(1), and induction assumption.

   Finally, we assume  $h=e_1$. If   $\alpha_i\neq 0$ for some  $i$ with $ 2\le i\le r$, then (\ref{veri})  follows from Lemma~\ref{comm1}$\sc\,$(1) and induction assumption.
   Suppose  $x^{\alpha}=x_1^{\alpha_1}$ with  $\alpha_1>0$. We need to verify
  \begin{equation}\label{veri1} e_1 x_1^{\alpha_1} c^{-1} e^f w d \mbox{$\prod\limits_{i=1}^t$} \bar x_i^{\beta_i}\in M\ \text{ \ for $\alpha_1>0$}.\end{equation}
Note that $ce_1 c^{-1}=e_{i, j}$ for some $i,j$. By Lemma~\ref{sx} and induction assumption on $|\alpha|$, we can use  $c^{-1} x_i^{\alpha_1}$
  to replace  $x_1^{\alpha_1} c^{-1}$ in \eqref{veri1}. So, we need to verify
  \begin{equation} \label{v1} e_{i, j} x_i^{\alpha_1} e^f w d \mbox{$\prod\limits_{i=1}^t$} \bar x_i^{\beta_i}\in M.\end{equation}
  In fact,  by Lemma~\ref{comm}$\sc\,$(2) and induction assumption, it is equivalent to verifying
   \begin{equation}\label{v2} e_{i, j} \bar x_j^{\alpha_1} e^f w d \mbox{$\prod\limits_{i=1}^t$} \bar x_i^{\beta_i}\in M.\end{equation}
  If $j\ge f+1$,  \eqref{v2} follows from Lemma~\ref{comm1}$\sc\,$(2) and Theorem~\ref{basis}. Otherwise,  $j\le  f$.

  If $i=j$, by induction assumption, we use $( x_i+L_i)^{\alpha_1}$ instead of  $x_i^{\alpha_1} $ in  $e_{i,j} x_i^{\alpha_1} e_j$. So, \eqref{v1}  follows from Lemma~\ref{ex}$\sc\,$(2). If $i\neq j$,  we have  $$e_{i,j} x_i^{\alpha_1} e_j=e_{i,j} e_j x_i^{\alpha_1} = (i, j)x_i^{\alpha_1}e_j=x_j^{\alpha_1} (i, j)e_j,$$ which holds in
   ${\rm gr}(\mathscr B_{r,t}^{\text{aff}})$. By induction assumption and our previous result on  $h\in$ $\{s_1, \cdots, s_{r-1}, x_1\}$, we have \eqref{v1} and hence (\ref{veri1}). This completes the proof.  \end{proof}

%\begin{remark}\label{polyr} In the rest part of this section, we will prove the linear independence of all regular monomials of $ \mathscr B_{r,t}^{\text{aff}}$ by using
% Nazarov's idea for affine Wenzl algebras in \cite{Na}.
%As a byproduct,   $e_1$ is $R[\omega_2, \omega_3, \cdots, \omega_k]$-torsion-free.
%So, the assumption in Corollary~\ref{bomega} can be removed .
%\end{remark}

Now we are able to prove the main result of this section. We remark that the idea of the proof is motivated by
Nazarov's work on affine Wenzl algebras in \cite{Na}.

\begin{theorem} \label{main1}  Suppose  $R$ is a commutative ring  which contains $1,  \omega_0, \omega_1$. If $e_1$ is $R[\omega_2, \omega_3, \cdots, \bar \omega_0, \bar \omega_1,\cdots]$-torsion free, then
 $\mathscr B_{r,t}^{\text{aff}}$ is free over $R$ spanned by
all regular monomials in \eqref{Monooo}. In particular,  $\mathscr B_{r,t}^{\text{aff}}$ is  of infinite rank. \end{theorem}
\begin{proof} Let $\textit{\textbf{M}}$ be the set of all regular monomials of  $\mathscr B_{r,t}^{\text{aff}}$. First,  we  prove that $\textit{\textbf{M}}$ is $F$-linear independent
where   $F$ is the quotient field of $\mathbb Z[\omega_0, \omega_1]$ with  $\omega_0,\omega_1$ being indeterminates.

%Let $\mathcal S$ be any finite subset of $M$. We claim that $r_{\textbf m}=0$ for all ${\textbf m} \in \mathcal S$  if $\sum_{\textbf m\in \mathcal S} r_{\textbf m} {\textbf m}=0$.
% Otherwise, we can find a subset of $M$, say $\mathcal S$,
Suppose conversely there is a finite subset $\mathcal S$ of $\textit{\textbf{M}}$
such that  $\sum_{\textit{\textbf m}\in \mathcal S} r_{\textit{\textbf m}} {\textit{\textbf m}}=0$ with $r_{\textit{\textbf m}}\neq 0$ for all $\textit{\textbf m}\in \mathcal S$.
 Recall from Definition \ref{def-mo} that each regular monomial  is of the form in \eqref{Monooo}.
% : $${\textbf m}=\mbox{$\prod\limits_{i=1}^r$} x_i^{\alpha_i}  d^{-1} e^f w e \mbox{$\prod\limits_{i=1}^t$} \bar x_i^{\beta_i}
% \mbox{$\prod\limits_{i\in \mathbb Z^{\ge 2} }$} \omega_i^{a_i} \in \mathcal S,$$
%with  $a_i\neq 0$ for  finite many $i$'s.
 For each $\textit{\textbf m}\in \mathcal S$ as in \eqref{Monooo}, we set
\begin{equation}\label{k-km}k_{\textit{\textbf m}}=\max\Big\{|\alpha|+\SUM{j}{} j a_j,\, |\beta|+\SUM{j}{} j a_j\Big\},\ \ \ k=\max\{ k_{\textit{\textbf m}}\,|\,\textit{\textbf m}\in \mathcal S\},
\end{equation}
where $|\alpha|\!=\!{\sc\!}\sum_{i=1}^r{\ssc\!} \alpha_i,{\ssc\,}|\beta|\!=\!{\sc\!}\sum_{i=1}^t{\ssc\!} \beta_i$.~%
Consider the %algebraic
homomorphism $\phi_k{\ssc\!}:\! \mathscr B_{r,t}^{\text{aff}}{\ssc\!}\!\rightarrow {\ssc\!}\!\mathscr B_{r+k,t+k}(\omega_0)$ in Theorem~\ref{alghom}.
Then $\phi_k(\textit{\textbf m})$ can be written as a linear combinations of $(r+k, t+k)$-walled Brauer diagrams.

Using Lemma~\ref{sx}$\sc\,$(3) to express $\phi_k(x_\ell)$ and $\phi_k(\bar x_\ell)$, and using
  Lemma~\ref{omed} to express $\omega_{a,k+1}$ for $a\in \mathbb Z^{\ge2}$, we see that
% $$\phi_k(x_\ell)=\SUM{j=1}k e_{k+\ell, j} -L_{k+\ell}+\omega_0^{-1}\omega_1, \text{ \ \ } \phi_k(\bar x_\ell)=\SUM{j=1}k e_{j, k+\ell} -\bar L_{k+\ell}-\omega_0^{-1}\omega_1.$$
  some terms of $\phi_k(
\textit{\textbf m})$ are of forms
(we will see in the next paragraph %below
that other terms of $\phi_k(\textit{\textbf m})$ will not contribute to our computations)
\begin{equation} \label{prod1}  \mbox{$\prod\limits_{i=1}^r$} (k\! +\!i, i_{1}) \cdots (k\!+\!i, i_{\alpha_i}) \phi_k (c^{-1} e^f w d)
 \mbox{$\prod\limits_{j=1}^t$} (\overline{k\!+\!i}, \bar j_1) \cdots (\overline{k\!+\!i}, \bar j_{\beta_j})  \mbox{$\prod\limits_{i\ge 2}$} \mathbf c_i,
\end{equation}
where $\mathbf c_i$ ranges over products of  some disjoint cycles in $\mathfrak S_k$ (or $\bar {\mathfrak S}_k)$ with total length $ia_i$.
We remark that such $\mathbf c_i$'s come from $\omega_{i, k+1}$.  Further, the walled Brauer diagram corresponding to
$\phi_k (c^{-1} e^f w d)$ have vertical edges $[i,i]$ and $[\bar j, \bar j]$ for all $i, j$ with $1\le i, j\le k$.
    We call the terms of the form (\ref{prod1}) the {\it leading terms} if
 \begin{itemize}\item[(i)]
 either $k=|\alpha|+\sum_{j} j a_j$ or  $k=|\beta|+\sum_{j} j a_j$ (cf.~\eqref{k-km}), and
 \item[(ii)] the corresponding  $f$ in (\ref{prod1}) is minimal among all terms satisfying (i), and
\item[(iii)]
%such that,
in the first case of (i), the juxtaposition of the sequences $i_1, i_2, \cdots, i_{\alpha_i}$ for $1\!\le \!i\!\le\! r$ and $\mathbf c_i$, $i\ge 2$ run through all  permutations of the sequences in $1, 2, \cdots, k$; while in the second case of (i), the  juxtaposition of the sequences $j_1, j_2, \cdots, j_{\beta_j}$ for $1\le j\le r$ and $\mathbf c_i$, $i\ge 2$ run through all permutations of the sequences in $\bar 1, \bar 2, \cdots, \bar k$.
 \end{itemize}
% Further, i
If we identify the factor $\phi_k (c^{-1} e^f w d)$ in the %previous
leading terms with the corresponding walled Brauer diagrams,
 we  have   %\vspace*{-3pt}
 \begin{enumerate}
 \item there are exactly $f$ horizontal edges in both top and bottom rows,
 \item no vertical edge of form $[i, i]$, $1\le i\le k$  in the first case,
 \item no vertical edge of form $[\bar i, \bar i]$, $1\le i\le k$ in the second case,
 \item no horizontal edge of form $[i, \bar j]$, $1\!\le\! i\! \le\! k$, $\bar 1\!\le\! \bar j\!\le\! \bar k$
 in both %top and bottom
           rows.
 \end{enumerate}
 These leading terms exactly appear in $\phi_k(\textit{\textbf m})$ when conditions (i)--(iii) are satisfied.
%  by $\textit{{\textbf{m}}}$.
%
% the corresponding walled Brauer diagrams $\textbf m$  have  $f$  horizontal edges  in both top and bottom rows and $k_{\textbf m}=k$.

 Other terms in $\phi_k( \sum_{\textit{\textbf m}\in \mathcal S} r_{\textit{\textbf{m}}} {\textit{\textbf m}})$ are non-leading terms,
 which are terms
 obtained by (\ref{prod1}) by using some $e_{k+i, j}$'s (resp., $e_{j, k+i}$'s) or scalars instead of
some $(k+i, i_{j})$'s  (resp., $(\overline{k+i}, \bar i_{j})$'s${\sc\,}$) or using certain
product of  $e_{i, j}$'s, $1\le i, j\le k$ instead of some factors of some  cycles $\mathbf c_i$'s.
Thus such terms can not be proportional to any leading terms. Therefore  $\mathcal S$ is $F$-linear independent.
By Proposition~\ref{basis1},   $\textit{\textbf{M}}$ is a $\mathbb Z[\omega_0, \omega_1]$-basis.

Now, for an arbitrary commutative ring  $R$ containing $1,{\mathbf {\omega}}_0,{\mathbf {\omega}}_1$,
we can regard $R$ as a left $\mathbb Z[{\mathbf{\omega}}_0, \omega_1]$-module such that the indeterminates ${\mathbf {\omega}}_0,\,{\mathbf{\omega}}_1\in \mathbb Z[{\mathbf{\omega}}_0, \omega_1]$
act on $R$ as the scalars ${\mathbf {\omega}}_0,\,{\mathbf{\omega}}_1\in R$ respectively. By standard arguments on specialization,   $\mathscr B_{r, t}^{\rm {aff}}$, which is defined  over $R$,
is isomorphic to $A\otimes_{\mathbb Z[\omega_0, \omega_1]} R$, where $A$ is the algebra $\mathscr B_{r, t}^{\rm {aff}}$ defined over $\mathbb Z[\omega_0, \omega_1]$.
So,  $\mathscr B_{r, t}^{\rm {aff}}$ is free over $R$ with infinite rank. \end{proof}

Let  $R$ be a commutative ring  containing $1,$ $ \hat{\omega}_a$ for $a\in \mathbb Z^{\ge2}$.
Let $I$ be the two-sided ideal of $\mathscr B_{r,t}^{\text{aff}}$ generated by $\omega_a-\hat \omega_a,\,a\in\Z^{\ge2}$. Then there is an epimorphism
$\psi: \mathscr B_{r,t}^{\text{aff}}\rightarrow$ $  \mathscr B_{r,t}^{\text{aff}}/I$.
Let $\widehat {\mathscr  B_{r,t}}=\mathscr B_{r,t}^{\text{aff}}/I$, namely $\widehat {\mathscr  B_{r,t}}$ is the specialization of ${\mathscr  B_{r,t}}$ with $\omega_a$ being specialized to
$\hat \omega_a$ for $a\in\Z^{\ge2}$. Without confusion, we will simply denote elements $\hat \omega_a$ of $R$ as $\omega_a$.%\vspace*{-3pt}.
\begin{definition}\label{reduced-mo}
 We say that the image of a regular monomial $\textit{\textbf{m}}$  of $\mathscr B_{r,t}^{\text{aff}}$ (cf. \eqref{Monooo}) is a {\it regular monomial} of $\widehat {\mathscr  B_{r,t}}$
 if $\textit{\textbf{m}}$  does not contain factors $\omega_i$'s for $i\ge 2$.
\end{definition}

The following result  %is the main result in this section, which
  follows from Theorem~\ref{main1}, immediately.
 %\vspace*{-3pt}.
\begin{theorem}\label{main2}  Suppose $R$ is a commutative ring which contains $1, \omega_i$ with $i\in \mathbb Z^{\ge0}$.
Then $\widehat{\mathscr B_{r,t}}$ is free over $R$ spanned by
all regular monomials. In particular, $\widehat{\mathscr B_{r,t}}$   is  of infinite rank. \end{theorem}

We close this section by giving some relationship between affine walled Brauer algebras and
degenerate affine Hecke algebras~\cite{K}, walled Brauer algebras, etc\vspace*{-3pt}.

\begin{definition}\label{degenerate H} The degenerate affine Hecke algebra $\mathscr H_{n}^{\text{aff}}$ is the unital $R$-algebra generated by
$S_1,\dots,S_{n-1},Y_1,\dots,Y_n$ and relations
\begin{alignat*}{3}
&  S_iS_j =S_jS_i, \ \ \ \ \ \ \ \ \ \ \ \ \ \ \ \ \ \ \ \ \ \ Y_iY_k =Y_kY_i, \\
&  S_iY_i-Y_{i+1}S_i=-1,\ \ \ \ \ \ \ \ \ \ \ \ Y_iS_i-S_iY_{i+1}=-1, \\
&  S_jS_{j+1}S_j =S_{j+1}S_jS_{j+1}, \ \ \ \ \ \ \ S_i^2 =1,
\end{alignat*}for $1\le i<n$, $1\le j<n-1$ with $|i-j|>1$, and $1\le k\le n$.
\end{definition}

\begin{proposition}\label{epi} Let $R$ be a commutative ring containing $1,\, \omega_i$ with $i\in \mathbb Z^{\ge0}$. Let $\widehat {\mathscr  B_{r,t}}  $ be the affine walled Brauer algebra over $R$.
Let $I$ $($resp., $J{\ssc\,})$ be the two-sided ideal of $ \widehat {\mathscr  B_{r,t}}  $ generated by $x_1$ and $ \bar x_1$  $($resp., $e_1{\ssc\,})$.

\begin{enumerate}
 \item $\widehat {\mathscr  B_{r,t}} /I\cong \mathscr B_{r,t}$.
 \item  $\widehat {\mathscr  B_{r,t}} /J\cong \mathscr H_{r}^{\text{aff}}\otimes \mathscr H_{t}^{\text{aff}} $.
\item
The subalgebra of $\widehat {\mathscr  B_{r,t}}  $ generated by $e_1,\, s_1, \cdots, s_{r-1}, \bar s_1, \cdots, \bar s_{t-1}$ is isomorphic to
 the walled Brauer algebra $\mathscr B_{r,t}$ over $R$.

\item The subalgebra of $\widehat {\mathscr  B_{r,t}} $ generated by $s_1, \cdots, s_{r-1}$ and $x_1$ $($resp.,
$\bar s_1, \cdots, \bar s_{t-1}$ and $\bar x_1{\ssc\,})$
 is isomorphic to the degenerate affine Hecke algebra $\mathscr H_{r}^{\text{aff}}$ $($resp., $\mathscr H_{t}^{\text{aff}}{\ssc\,})$.
%\item  The subalgebra of $\widehat {\mathscr  B_{r,t}} $ generated by
% is isomorphic to the degenerate affine Hecke algebra .
%
 \end{enumerate}
Note that the isomorphism in %Proposition~$\ref{epi}{\sc\,}$
$(4)$ sends $x_1$ $($resp., $\bar x_1{\ssc\,})$ to $-Y_1$.
 \end{proposition}

In the rest part of the paper, we will be interested in the specialized algebra $\widehat{\mathscr B_{r,t}}$. Without confusion,
 we will use  $\mathscr B_{r,t}^{\text{aff}}$ to denote it.
%We hope to develop the representation theory of  $\mathscr B_{r,t}^{\text{aff}}$ elsewhere.

\section{Super Schur-Weyl duality }\label{SW}
The main purpose of this section is to set up the relationship between  affine walled Brauer algebras $\mathscr B_{r,t}^{\text{aff}}$  with special parameters
and general linear Lie superalgebras $\mathfrak {gl}_{m|n}$.~Throughout the section, we assume %that
the ground field
is~$\mathbb C$.

Denote $\mfg=\mathfrak{gl}_{m|n}$.
%Let $\mfg$ be the complex  general linear Lie superalgebra $\mathfrak{gl}_{m|n}$, and l
Let  $V=\C^{m|n}$ be the
natural $\mfg$-module. As a $\mathbb C$-vector superspace $V=V_{\bar0}\oplus V_{\bar1}$
with $\dim V_{\bar 0}=m$ and $\dim V_{\bar 1}=n$.
 Take a natural basis $\{v_i\,|\,i\in I\}$ of $V$, where $I=\{1,2,...,m+n\}.$
For convenience we define the map $[\cdot]:I\to\Z_2$ by $[i]=\bar0$ if $i\le m$ and $[i]=\bar1$ if $i>m$. Then $v_i$ has the parity $[v_i]=[i]$.
 Denote by $E_{ij}$ the matrix unit, which has parity $[E_{ij}]=[i]+[j]$. The Lie bracket on $\mfg$ is defined by
 \begin{equation}\label{lib}
 [E_{ij},E_{k\ell}]=\d_{jk}E_{i\ell}-(-1)^{([i]+[j])([k]+[\ell])}\d_{\ell{\ssc\,}i}E_{kj},\end{equation}
 where $\delta_{jk}=1$ if $j=k$ and $0$, otherwise.
 Let $V^*$ be the dual space of $V$ with dual basis $\{\bar v_i\,|\,i\in I\}$. Then $V^*$ is a left $\mfg$-module with action
\begin{eqnarray}\label{action-dual}
E_{ab}\bar v_i=-(-1)^{[a]([a]+[b])}\d_{ia}\bar v_b.\end{eqnarray}

Let $\fh={\rm span}\{E_{ii}\,\,|\,i\in I\}$ be a Cartan subalgebra of $\mfg$, and $\fh^*$ the dual space of $\fh$ with
 $\{\es_i\,|\,i\in I\}$ being the dual basis of $\{E_{ii}\,\,|\,i\in I\}$. Then an element $\l\in\fh^*$ (called a {\it weight}) can be written as \begin{equation}\label{weight-}\l=\SUM{i\in I}{}\l_i\es_i=(\l_1,...,\l_m\,|\,\l_{m+1},...,\l_{m+n})\mbox{ \ with \ }\l_i\in\C.\end{equation}
Take $$\rho=\SUM{i=1}m(1\!-\!i)\es_i+\SUM{j=1}n(m\!-\!j)\es_{m+j}=(0,-1,...,1\!-\!m\,|\,m\!-\!1,m\!-\!2,...,m\!-\!n),$$ and
denote
\begin{eqnarray}
\label{lavel-}
&\!\!\!\!\!\!\!\!\!\!\!\!&
|\l|:=\SUM{i\in I}{}\l_i\mbox{ \ (called the {\it size} of $\l$)},\\
\label{rho-l}
&\!\!\!\!\!\!\!\!\!\!\!\!&
\l^\rho=\l+\rho=(\l_1^\rho,...,\l^\rho_m\,|\,\l^\rho_{m+1},...,\l^\rho_{m+n}),\mbox{ \ where},
\\\nonumber&\!\!\!\!\!\!\!\!\!\!\!\!&
\l_i^\rho=\l_i+1-i \text{ \ if \ $i\le m$, \ and \ $\l^\rho_i=\l_i+2m-i$ \ if \ $i>m$.}\end{eqnarray}

A weight $\l$ is called {\it integral dominant} if
$\l_i\!-\!\l_{i+1}\!\in\!\Z^{\ge0}$ for $i\!\in \!I{\ssc\!}\bs{\ssc\!}\{m,m\!+\!n\}$. It is called  {\it typical} if $\l_i^\rho\!+\!\l_j^\rho\!\ne\!0$ for any $i\!\le\! m\!<\!j$ (otherwise it is called {\it atypical}) \cite{Kac77}.
\begin{example}
For any $p,q\in\C$,
\begin{equation}\label{l-pq}\l_{pq}=(p,...,p\,|\,-q,...,-q),\end{equation}
is a typical integral dominant weight if and only if
\begin{equation}\label{Do-condi}
\mbox{$p - q \notin \Z$ \ or \ $p - q \le -m$ \ or \ $p - q \ge n$.}\end{equation}
 (Note that the $\l_{pq}$ defined in \cite[IV]{BS4} is the $\l_{p,q+m}$ defined here.)
In this case,
the finite-dimensional irreducible $\mfg$-module $L_{\l_{pq}}$ with highest weight ${\l_{pq}}$
 coincides with the Kac-module $K_{\l_{pq}}$ \cite[IV]{ BS4}, \cite{Kac77}.
\end{example}\def\OT#1{{\raisebox{-8.5pt}{$\stackrel{{\mbox{\Large$\dis\otimes$}}}{\sc#1}$}}}

Let $M$ be any $\mfg$-module. For any $r,t\!\in\!\Z^{\ge0}$, set $M^{rt}=V^{\otimes r}\OTIMES M\OTIMES (V^*)^{\otimes t}$. For convenience
we denote the ordered set \begin{equation}\label{ordered-set}J=J_1\cup\{0\}\cup J_2,\
\mbox{ where $J_1=\{r,...,1\}$, $J_2=\{\bar 1,...,\bar t\}$,} \end{equation}
 and
$r\prec\cdots\prec1\prec0\prec\bar1\prec\cdots\prec\bar t$.
%For $i\in J$, we set
We write $M^{rt}$ as
\begin{equation}\label{M-st==}M^{rt}=\OT{i\in J}V_i,\mbox{ \ where $V_i=V$ if $i\prec0$, $V_0=M$ and $V_i=V^*$ if $i\succ0$,} \end{equation}
 (hereafter all tensor products will be taken according to the order in $J$), which is a left $U(\mfg)^{\otimes(r+t+1)}$-module (where $U(\mfg)$ is the universal enveloping algebra of $\mfg$), with the action given by
 $$\Big(\OT{i\in J} g_i\Big)\Big(\OT{i\in J} x_i\Big)=(-1)^{\sum\limits_{i\in J}{}[g_i]\sum\limits_{j\prec i}{}[x_j]}\OT{i\in J}(g_ix_i)\mbox{ for }g_i\in U(\mfg),\ x_i\in V_i.$$
In particular, if we delete the tensor $M$ (or take $M=\C$), then $M^{rt}$ is the mixed tensor product studied in
\cite{SM}, and if $t=0$ and $M=K_{\l_{pq}}$, then $M^{rt}$ is the tensor module studied in \cite[IV]{BS4}.

We denote \begin{eqnarray}\label{def-Omega}
\Omega=\SUM{i,j\in I}{}(-1)^{[j]}E_{ij}\OTIMES E_{ji}\in\mfg^{\otimes2}.\end{eqnarray}
Because of the following well-known property of $\Omega$, it is called a {\it Casimir element}.
\begin{lemma}\label{Casiii}
For any $g\in\mfg$, we denote
 $\D(g)=g\OTIMES1+1\OTIMES g\in\mfg^{\otimes2}$ $($i.e., $\D$ is the comultiplication of the quantum group $U(\mfg){\ssc\,})$. Then
 \begin{equation}\label{equ-Omega}[\D(E_{ab}),\Omega]=0\mbox{ \ for all \ }a,b\in I.\end{equation}\end{lemma}
 \begin{proof}
By definition, the left-hand side of \eqref{equ-Omega} is equal to
\begin{eqnarray*}&&\!\!\!\!\!\!\!\!\!\!\!\!\!\!\!\!\!\!\!\!\!\!\!\!\!\!
%[\D(E_{ab}),\Omega]\nonumber\\&\!\!\!=\!\!\!&
\mbox{$\sum\limits_{i,j\in I}$}(-1)^{[j]}\Big([E_{ab},E_{ij}]\OTIMES E_{ji}+(-1)^{([a]+[b])([i]+[j])}E_{ij}\OTIMES[E_{ab},E_{ji}]\Big)\nonumber\\
\ \ \ \ \ \ \ &\!\!\!=\!\!\!&\mbox{$\sum\limits_{j\in I}$}(-1)^{[j]}E_{aj}\OTIMES E_{jb}-
\mbox{$\sum\limits_{i\in I}$}(-1)^{[a]+([a]+[b])([i]+[a])}E_{ib}\OTIMES E_{ai}\nonumber\\
\ \ \ \ \ \ \ &\!\!\!\!\!\!&+
\mbox{$\sum\limits_{i\in I}$}(-1)^{[b]+([a]+[b])([i]+[b])}E_{ib}\OTIMES E_{ai}
-\mbox{$\sum\limits_{j\in I}$}(-1)^{[j]}E_{aj}\OTIMES E_{jb}
,%\nonumber\\&\!\!\!=\!\!\!&0,
\end{eqnarray*}
which is equal to zero by noting that \\[4pt]\hspace*{8ex}$[a]\!+\!([a]\!+\![b])([i]\!+\![a])\!=\!
([a]\!+\![b])[i]\!+\![a][b]\!=\![b]\!+\!([a]\!+\![b])([i]\!+\![b]).$
\end{proof}
For $a,b\in J$ with $a\prec b$, we define
$\pi_{ab}:U(\mfg)^{\otimes2}\to U(\mfg)^{\otimes(r+t+1)}$ by \begin{equation}\label{pi-ab}
\pi_{ab}(x\OTIMES y)=1\OTIMES\cdots\OTIMES1\OTIMES x\OTIMES 1\OTIMES\cdots\OTIMES1\OTIMES y\OTIMES1\OTIMES\cdots\OTIMES1,\end{equation}
where $x$ and $y$ are in the $a$-th and $b$-th tensors respectively.
\begin{notation}
From now on, we always suppose $M\!=\!K_{\l}$ is the Kac-module with highest weight $\l=\l_{pq}$ in \eqref{l-pq} for $p,q\in\C$ (at this moment, we do not impose any condition on $p,q$) and a
highest weight vector $v_{\l}$  defined to have parity $\bar0$.\end{notation}
Note~that
\begin{equation}\label{action-ofEij}
E_{ij}v_\l=\left\{\begin{array}{ll}pv_\l&\mbox{if }1\le i=j\le m,\\[4pt]-qv_\l&\mbox{if }m<i=j\le m+n,\\[4pt]
0&\mbox{if }1\le i\ne j\le m\mbox{ or }m\le i\ne j\le m+n,\end{array}\right.\end{equation}
 and $K_\l$ is $2^{mn}$-dimensional with a basis
\begin{equation}\label{basis-k-l}
B=\Big\{b^\si:=\mbox{$\prod\limits_{i=1}^{n}\,\prod\limits_{j=1}^{m}$}E_{m+i,j}^{\si_{ij}}v_\l\,\Big|\,\si=(\si_{ij})_{i,j=1}^{n,m}\in\{0,1\}^{n\times m}\Big\},
\end{equation}
where the products are taken in any fixed order (changing the order only changes the vectors
by $\pm1$). Then $M^{rt}$ is $2^{mn}(m+n)^{r+t}$-dimensional with a basis
\begin{equation}\label{basis-M-}
B_M=\Big\{b_M=\OT{i\in J_1}v_{k_i}\otimes b\otimes\OT{i\in J_2}\bar v_{k_i}\,\Big|\,b\in B,\,
k_i\in I\Big\}.
\end{equation}

Due to Lemma \ref{Casiii}, the elements defined below are   in the endomorphism algebra ${\rm End}_\mfg(M^{rt})$ of the $\mfg$-module $M^{rt}$, which will be used throughout the section.

\begin{definition}
By \eqref{equ-Omega}, we can use \eqref{pi-ab} to define the following elements of the endomorphism algebra ${\rm End}_\mfg(M^{rt})$,
\begin{eqnarray}\label{operator--1}&\!\!\!\!&
s_i=\pi_{i+1,i}(\Omega)|_{M^{rt}}\ (1\le i<r),\ \ \ \
\bar s_j=\pi_{\bar j,\overline{j+1}}(\Omega)|_{M^{rt}}\ (1\le j<t),\nonumber\\&\!\!\!\!&
x_1=\pi_{10}(\Omega)|_{M^{rt}},\ \ \ \ \ \ \ \bar x_1=\pi_{0\bar1}(\Omega)|_{M^{rt}},\ \ \ \ \ \ \
  e_1=\pi_{1\bar1}(\Omega)|_{M^{rt}}.
\end{eqnarray}
\end{definition}
%Now suppose $M=M_\l$ is any highest weight $\mfg$-module with any highest weight $\l$ and a
%highest weight vector $v_\l$  defined to have parity $\bar0$.

For convenience, we will write elements of ${\rm End}_\mfg(M^{rt})$ as right actions on $M^{rt}$.
However one shall always keep in mind that
all endomorphisms are defined by left multiplication of the Casimir element $\Omega$ such that
the first and second tensors in $\Omega$ act on some appropriate tensor positions in $M^{rt}$, and also one shall always do bookkeeping on the sign change whenever the order of two elements (factors) in a term are exchanged.
\begin{lemma}
For $w_1\in V^{\otimes(r-1)}$, $w_2\in (V^*)^{\otimes t}$ and $i\in I$,  we have
\begin{eqnarray}\label{action-x1}&\!\!\!\!\!\!\!\!\!&\!\!\!\!\!\!\!\!\!\!\!\!\!\!\!\!\!
(w_1\OTIMES v_i\OTIMES v_\l\OTIMES w_2)x_1\nonumber\\
%&\!\!\!=\!\!\!&w_1\OTIMES\mbox{$\sum\limits_{a,b\in I}$}(-1)^{[b]+[i]([a]+[b])}E_{ab}v_i\OTIMES E_{ba}v_\l\OTIMES w_2\nonumber\\&\!\!\!=\!\!\!&
%w_1\OTIMES\mbox{$\sum\limits_{a\in I,\,a\le i}$}(-1)^{[a][i]}v_a\OTIMES E_{ia}v_\l\OTIMES w_2
%\nonumber\\
&\!\!\!=\!\!\!&
\left\{\!\!\!\begin{array}{lll}
p{\sc\,} w_1\OTIMES v_i\OTIMES v_\l\OTIMES w_2&\!\!\!\mbox{if  }i\le m,\\
qw_1\OTIMES v_i\OTIMES v_\l\OTIMES w_2+\mbox{$\sum\limits_{a=1}^m$}w_1\OTIMES v_a\OTIMES E_{ia}v_\l\OTIMES w_2&\!\!\!\mbox{if }i>m.
\end{array}\right.
%\nonumber\\
%&\!\!\!=\!\!\!&\left\{\begin{array}{ll}pv_\l\OTIMES v_i\OTIMES w&\mbox{if \ }i\le m,\\
%\!\Big(\!(q\!+\!m)v_\l\OTIMES v_i
%\!+\!\sum\limits_{b=1}^m E_{ib}v_\l\OTIMES v_b\!\Big)\!\OTIMES w&\mbox{if \ }i>m,\end{array}\right.
\end{eqnarray}
\end{lemma}
\begin{proof}Since the
two tensors of $\Omega$ in $x_1=\pi_{10}(\Omega)|_{M^{rt}}$ act on $v_i$ and $v_\l$ respectively,  the left-hand side of \eqref{action-x1} is equal to
\begin{eqnarray*}
%\label{action-x11}
&\!\!\!\!\!\!\!\!\!&\!\!\!\!\!\!\!\!\!\!\!\!\!\!
%(w_1\OTIMES v_i\OTIMES v_\l\OTIMES w_2)x_1\nonumber\\&\!\!\!\!\!\!&
w_1\OTIMES\mbox{$\sum\limits_{a,b\in I}$}(-1)^{[b]+[i]([a]+[b])}E_{ab}v_i\OTIMES E_{ba}v_\l\OTIMES w_2\nonumber\\&\!\!\!=\!\!\!&
w_1\OTIMES\mbox{$\sum\limits_{a\in I,\,a\le i}$}(-1)^{[a][i]}v_a\OTIMES E_{ia}v_\l\OTIMES w_2,
%\nonumber\\&\!\!\!=\!\!\!&
%\left\{\!\!\!\begin{array}{lll}
%p{\sc\,} w_1\OTIMES v_i\OTIMES v_\l\OTIMES w_2&\!\!\!\mbox{if  }i\le m,\\
%qw_1\OTIMES v_i\OTIMES v_\l\OTIMES w_2+\mbox{$\sum\limits_{a=1}^m$}w_1\OTIMES v_a\OTIMES E_{ia}v_\l\OTIMES w_2&\!\!\!\mbox{if }i>m.
%\end{array}\right.
%\nonumber\\
%&\!\!\!=\!\!\!&\left\{\begin{array}{ll}pv_\l\OTIMES v_i\OTIMES w&\mbox{if \ }i\le m,\\
%\!\Big(\!(q\!+\!m)v_\l\OTIMES v_i
%\!+\!\sum\limits_{b=1}^m E_{ib}v_\l\OTIMES v_b\!\Big)\!\OTIMES w&\mbox{if \ }i>m,\end{array}\right.
\end{eqnarray*}
%by \eqref{Kac-action}.
which is equal to the right-hand side of \eqref{action-x1} by \eqref{action-ofEij}.
\end{proof}
\begin{lemma}
For  $w_1\in V^{\otimes r},\,w_2\in(V^*)^{\otimes(t-1)}$ and $i\in I$, we have
\begin{eqnarray}\label{action-x1*}&\!\!\!\!\!\!\!\!\!&\!\!\!\!\!\!\!\!\!\!\!\!\!\!\!\!\!
(w_1\OTIMES v_\l\OTIMES  \bar v_i\OTIMES w_2)\bar x_1\nonumber\\[-15pt]
%&\!\!\!=\!\!\!&
%w_1\OTIMES\mbox{$\sum\limits_{a,b\in I}$}(-1)^{[b]}E_{ab}v_\l \OTIMES E_{ba}\bar v_i\OTIMES w_2\nonumber\\
%&\!\!\!=\!\!\!&
%-w_1\OTIMES\mbox{$\sum\limits_{a\in I,\,a\ge i}$}(-1)^{[a][i]}E_{ai}v_\l\OTIMES \bar v_a\OTIMES w_2
\nonumber\\&\!\!\!=\!\!\!&
\left\{\!\!\!\begin{array}{lll}
-p{\sc\,}w_1\OTIMES v_\l\OTIMES  \bar v_i\OTIMES w_2-\!\!\mbox{$\sum\limits_{a=m+1}^{m+n}$}\!
w_1\OTIMES E_{ai}v_\l\OTIMES \bar v_a\OTIMES w_2\!\!\!
&\mbox{if  }i\le m,\\
-qw_1\OTIMES v_\l\OTIMES  \bar v_i\OTIMES w_2&\mbox{if }i>m.
\end{array}\right.
%\nonumber\\
%&\!\!\!\!\!\!\!\!&=%&\!\!\!=\!\!\!&
%\left\{\begin{array}{ll}\!\Big(\!-pv_\l\OTIMES w_1\OTIMES v^*_i\!-\!\sum\limits_{a=m+1}^{m+n}(-1)^{[w_1]}E_{ai}v_\l\OTIMES w_1\OTIMES v_a^*\!\Big)\!\OTIMES w_2&\mbox{if \ }i\le m,\\
%-(q+m)v_\l\OTIMES w_1\OTIMES v^*_i\OTIMES w_2&\mbox{if \ }i>m,\end{array}\right.
\end{eqnarray}
\end{lemma}\begin{proof} By definition and  \eqref{action-dual},  the left-hand side of \eqref{action-x1*} is equal to
\begin{eqnarray*}\ \ \ \ %&\!\!\!\!\!\!\!\!\!&\!\!\!\!\!\!\!\!\!\!\!\!\!\!\!\!\!
%(w_1\OTIMES v_\l\OTIMES  \bar v_i\OTIMES w_2)\bar x_1\nonumber\\&\!\!\!=\!\!\!&
w_1\OTIMES\mbox{$\sum\limits_{a,b\in I}$}(-1)^{[b]}E_{ab}v_\l \OTIMES E_{ba}\bar v_i\OTIMES w_2
=%\nonumber\\&\!\!\!=\!\!\!&
-w_1\OTIMES\mbox{$\sum\limits_{a\in I,\,a\ge i}$}(-1)^{[a][i]}E_{ai}v_\l\OTIMES \bar v_a\OTIMES w_2,
%\nonumber\\&\!\!\!=\!\!\!&
%\left\{\!\!\!\begin{array}{lll}
%-p{\sc\,}w_1\OTIMES v_\l\OTIMES  \bar v_i\OTIMES w_2-\!\!\mbox{$\sum\limits_{a=m+1}^{m+n}$}\!
%w_1\OTIMES E_{ai}v_\l\OTIMES \bar v_a\OTIMES w_2\!\!\!
%&\mbox{if  }i\le m,\\
%-qw_1\OTIMES v_\l\OTIMES  \bar v_i\OTIMES w_2&\mbox{if }i>m.
%\end{array}\right.
%\nonumber\\
%&\!\!\!\!\!\!\!\!&=%&\!\!\!=\!\!\!&
%\left\{\begin{array}{ll}\!\Big(\!-pv_\l\OTIMES w_1\OTIMES v^*_i\!-\!\sum\limits_{a=m+1}^{m+n}(-1)^{[w_1]}E_{ai}v_\l\OTIMES w_1\OTIMES v_a^*\!\Big)\!\OTIMES w_2&\mbox{if \ }i\le m,\\
%-(q+m)v_\l\OTIMES w_1\OTIMES v^*_i\OTIMES w_2&\mbox{if \ }i>m,\end{array}\right.
\end{eqnarray*}
which is equal to the right-hand side of \eqref{action-x1*} by \eqref{action-ofEij}.
\end{proof}
\begin{lemma}
For $w_1\in V^{\otimes(r-1)}$, $w_2\in V^{\otimes(t-1)}$, $i,j\in I$, we have
\begin{eqnarray}\label{action-e}&\!\!\!\!\!\!\!\!\!&\!\!\!\!\!\!\!\!\!\!\!\!\!\!\!\!\!
(w_1\OTIMES v_i\OTIMES v_\l\OTIMES  \bar v_j\OTIMES w_2)e_1
%\nonumber\\
%&\!\!\!=\!\!\!&\!
%w_1\OTIMES \mbox{$\sum\limits_{a,b\in I}$}(-1)^{[b]+[i]([a]+[b])}E_{ab}v_i\OTIMES v_\l\OTIMES
% E_{ba}\bar v_j\OTIMES w_2\nonumber\\
=%&\!\!\!=\!\!\!&
(-1)^{1+[i]}\d_{ij}\mbox{$\sum\limits_{a\in I}$}w_1\OTIMES v_a\OTIMES v_\l \OTIMES \bar v_a\OTIMES w_2.
\end{eqnarray}
\end{lemma}\begin{proof} By definition, the left-hand side of \eqref{action-e} is
$$w_1\OTIMES \mbox{$\sum_{a,b\in I}$}(-1)^{[b]+[i]([a]+[b])}E_{ab}v_i\OTIMES v_\l\OTIMES
 E_{ba}\bar v_j\OTIMES w_2,$$ which equals the right-hand side of \eqref{action-e}.
%\begin{eqnarray}\label{action-e}&\!\!\!\!\!\!\!\!\!&\!\!\!\!\!\!\!\!\!\!\!\!\!\!\!\!\!
%(w_1\OTIMES v_i\OTIMES v_\l\OTIMES  \bar v_j\OTIMES w_2)e_1
%\nonumber\\
%&\!\!\!=\!\!\!&\!
%w_1\OTIMES \mbox{$\sum\limits_{a,b\in I}$}(-1)^{[b]+[i]([a]+[b])}E_{ab}v_i\OTIMES v_\l\OTIMES
% E_{ba}\bar v_j\OTIMES w_2\nonumber\\
%&\!\!\!=\!\!\!&
%(-1)^{1+[i]}\d_{ij}\mbox{$\sum\limits_{a\in I}$}w_1\OTIMES v_a\OTIMES v_\l \OTIMES \bar v_a\OTIMES w_2.
%\end{eqnarray}
\end{proof}

Similarly, one can easily verify the following.
\begin{lemma}
For $a,b\in I$, and  $w_1\in V^{\otimes(r-1-i)},$ $w_2\in V^{\otimes(i-1)}\OTIMES  K_\l\OTIMES  (V^*)^{\otimes t},$ $
w'_1\in $ $V^{\otimes r}\OTIMES K_\l\OTIMES (V^*)^{\otimes(j-1)}$, $w'_2\in (V^*)^{\otimes (t-1-j)}$,
we have
\begin{eqnarray}\label{action-s1}&&
(w_1\OTIMES v_a\OTIMES v_b\OTIMES w_2)s_i=(-1)^{[a][b]}w_1\OTIMES v_b\OTIMES v_a\OTIMES w_2,\nonumber\\&&
(w'_1\OTIMES \bar v_a\OTIMES \bar v_b\OTIMES w'_2)\bar s_i=(-1)^{[a][b]}w'_1\OTIMES \bar v_b\OTIMES \bar v_a\OTIMES w'_2.
\end{eqnarray}
%\vskip4pt
\end{lemma}

The following is obtained in \cite[Proposition 3.2]{SM} when the tensor $M=K_\l$ is omitted.
\begin{proposition}\label{Without-M}There exists a $\C$-algebraic homomorphism $\Phi:{\mathscr B}_{r, t}(m-n)\to $ $ {\rm End}_{\mfg}(M^{rt})$, which sends generators $e_1,s_i,\bar s_j$, $1\le i\le r-1$, $1\le j\le t-1$ of
${\mathscr B}_{r, t}(m-n)$ in Theorem $\ref{wbmwf}$ to elements with the same symbols.
\end{proposition}

We need the following results before stating a main result of this section.
\begin{lemma}\label{m-le1}
Let $x_1,{\ssc\,}\bar x_1,e_1$ be defined in~\eqref{operator--1}.~Then $(x_1\!+\!\bar x_1)e_1\!=\!0\!=\!e_1(x_1\!+\!\bar x_1)$.\end{lemma}
\begin{proof}
To prove the result, it suffices to consider the case $r\!=\!t\!=\!1$.~By~\eqref{action-x1},~we~have
\begin{eqnarray}\label{equa-ex}
(v_i\OTIMES v_\l\OTIMES \bar v_j)x_1e_1
%&\!\!\!=\!\!\!&
%e_1\Big(\mbox{$\sum\limits_{a,b\in I}$}(-1)^{[b]+[i]([a]+[b])}E_{ab}v_i\OTIMES E_{ba}v_\l\OTIMES \bar v_j\Big)\nonumber\\
&\!\!\!=\!\!\!&\Big(\mbox{$\sum\limits_{a\in I}$}(-1)^{[i][a]}v_{a}\OTIMES E_{ia}v_\l\OTIMES \bar v_j\Big)e_1
\nonumber\\&\!\!\!=\!\!\!&
%=
(-1)^{[i][j]}(v_j\OTIMES E_{ij}v_\l\OTIMES \bar v_j)e_1\nonumber\\[4pt]
&\!\!\!=\!\!\!&(-1)^{[i][j]}\mbox{$\sum\limits_{a,b\in I}$}(-1)^{[b]+([a]+[b])[i]}E_{ab}v_j\OTIMES E_{ij}v_\l\OTIMES E_{ba} \bar v_j\nonumber\\
&\!\!\!=\!\!\!&(-1)^{[i][j]}\mbox{$\sum\limits_{a\in I}$}(-1)^{[j]+([a]+[j])[i]+1+[j]([j]+[a])}v_{a}\OTIMES E_{ij}v_\l\OTIMES \bar v_a\nonumber\\
&\!\!\!=\!\!\!&\mbox{$\sum\limits_{a\in I}$}(-1)^{1+[a]([i]+[j])}v_{a}\OTIMES E_{ij}v_\l\OTIMES \bar v_a,
\end{eqnarray}and by \eqref{action-x1*},
\begin{eqnarray}\label{equa-e-barx}
(v_i\OTIMES v_\l\OTIMES \bar v_j)\bar x_1e_1%&\!\!\!=\!\!\!&e_1\Big(\mbox{$\sum\limits_{a\in I}$}(-1)^{[j]}v_i\OTIMES E_{aj}v_\l\OTIMES E_{ja}\bar v_j\Big)\nonumber\\
&\!\!\!=\!\!\!&\Big(\mbox{$\sum\limits_{a\in I}$}(-1)^{1+[a][j]}v_{i}\OTIMES E_{aj}v_\l\OTIMES \bar v_a\Big)e_1
\nonumber\\&\!\!\!=\!\!\!&
(-1)^{1+[i][j]}(v_i\OTIMES E_{ij}v_\l\OTIMES \bar v_i)e_1\nonumber\\[4pt]
&\!\!\!=\!\!\!&(-1)^{1+[i][j]}\mbox{$\sum\limits_{a,b\in I}$}(-1)^{[b]+([a]+[b])[j]}E_{ab}v_i\OTIMES E_{ij}v_\l\OTIMES E_{ba} \bar v_i\nonumber\\
&\!\!\!=\!\!\!&(-1)^{1+[i][j]}\mbox{$\sum\limits_{a\in I}$}(-1)^{[i]+([a]+[i])[j]+1+[i]([i]+[a])}v_{a}\OTIMES E_{ij}v_\l\OTIMES \bar v_a\nonumber\\
&\!\!\!=\!\!\!&\mbox{$\sum\limits_{a\in I}$}(-1)^{[a]([i]+[j])}v_{a}\OTIMES E_{ij}v_\l\OTIMES \bar v_a.
\end{eqnarray}
Thus $(x_1+\bar x_1)e_1=0$. By \eqref{action-e}, we have
\begin{eqnarray}\label{equa-xe}
(v_i\OTIMES v_\l\OTIMES \bar v_j)e_1x_1%&\!\!\!=\!\!\!&\d_{ij}
%\Big(\mbox{$\sum\limits_{a\in I}$}(-1)^{[i]+[i]([a]+[i])}E_{ai}v_i\OTIMES v_\l\OTIMES E_{ia}\bar v_i\Big)x_1\nonumber\\
&\!\!\!=\!\!\!&(-1)^{1+[i]}\d_{ij}\mbox{$\sum\limits_{a\in I}$}(v_{a}\OTIMES v_\l\OTIMES \bar v_a)x_1\nonumber\\
&\!\!\!=\!\!\!&\d_{ij}\mbox{$\sum\limits_{a,b\in I}$}(-1)^{1+[i]+[a]+[a]([a]+[b])}E_{ba}v_{a}\OTIMES E_{ab}v_\l\OTIMES \bar v_a\nonumber\\
&\!\!\!=\!\!\!&\d_{ij}\mbox{$\sum\limits_{a,b\in I}$}(-1)^{1+[i]+[a][b]}v_{b}\OTIMES E_{ab}v_\l\OTIMES \bar v_a,
\end{eqnarray}
and similarly, we obtain $e_1\bar x_1=-e_1x_1$.
\end{proof}
\begin{lemma}\label{m-le2}
Let $x_1,\,\bar x_1,e_1$ be defined in \eqref{operator--1}. Then $x_1 (e_1+\bar x_1)= (e_1+\bar x_1)x_1$.\end{lemma}
\begin{proof}As above, we can suppose $r=t=1$. By \eqref{action-x1},
\begin{eqnarray}\label{equa-barx-x}
&\!\!\!\!\!\!\!&\!\!\!\!\!\!\!\!\!\!\!\!\!\!
(v_i\OTIMES v_\l\OTIMES \bar v_j)x_1\bar x_1%&\!\!\!=\!\!\!&\bar x_1\Big(\mbox{$\sum\limits_{b\in I}$}(-1)^{[i]+[i]([b]+[i])}E_{bi}v_i\OTIMES E_{ib}v_\l\OTIMES \bar v_j\Big)\nonumber\\
\nonumber\\
&\!\!\!=\!\!\!&\mbox{$\sum\limits_{b\in I}$}(-1)^{[i][b]}(v_{b}\OTIMES E_{ib}v_\l\OTIMES \bar v_j)\bar x_1\nonumber\\
&\!\!\!=\!\!\!&\mbox{$\sum\limits_{a,b\in I}$}(-1)^{[i][b]+[j]+([i]+[b])([j]+[a])}v_{b}\OTIMES E_{aj}E_{ib}v_\l\OTIMES E_{ja}\bar v_j\nonumber\\
&\!\!\!=\!\!\!&\mbox{$\sum\limits_{a,b\in I}$}(-1)^{[i][b]+[j]+([i]+[b])([j]+[a])+1+[j]([j]+[a])}v_{b}\OTIMES E_{aj}E_{ib}v_\l\OTIMES \bar v_a\nonumber\\
&\!\!\!=\!\!\!&\mbox{$\sum\limits_{a,b\in I}$}(-1)^{1+[i][b]+[a][j]+([i]+[b])([j]+[a])}v_{b}\OTIMES E_{aj}E_{ib}v_\l\OTIMES \bar v_a,
\end{eqnarray}
and by \eqref{action-x1*},
\begin{eqnarray}\label{equa-x-barx}
&\!\!\!\!\!\!\!&\!\!\!\!\!\!\!\!\!\!\!\!\!\!
(v_i\OTIMES v_\l\OTIMES \bar v_j)\bar x_1x_1
%&\!\!\!=\!\!\!&\mbox{$\sum\limits_{a\in I}$}(-1)^{[j]}(v_i\OTIMES E_{aj}v_\l\OTIMES E_{ja}\bar v_j)x_1\nonumber\\
\nonumber\\
&\!\!\!=\!\!\!&\mbox{$\sum\limits_{a\in I}$}(-1)^{1+[a][j]}(v_{i}\OTIMES E_{aj}v_\l\OTIMES \bar v_a)x_1\nonumber\\
&\!\!\!=\!\!\!&\mbox{$\sum\limits_{a,b\in I}$}(-1)^{1+[a][j]+[i]+[i]([i]+[b])} E_{bi}v_i\OTIMES E_{ib}E_{aj}v_\l\OTIMES \bar v_a)\nonumber\\
&\!\!\!=\!\!\!&\mbox{$\sum\limits_{a,b\in I}$}(-1)^{1+[a][j]+[b][i]}v_b\OTIMES E_{ib}E_{aj}v_\l\OTIMES  \bar v_a.
\end{eqnarray}
Using $E_{ib}E_{aj}=(-1)^{([i]+[b])([a]+[j])}E_{aj}E_{ib}+\d_{ab}E_{ij}+(-1)^{1+([i]+[b])([a]+[j])}\d_{ij}E_{ab}$ in \eqref{equa-x-barx}, we obtain
\begin{eqnarray*}%\label{equa-x-barx+}
&\!\!\!\!\!\!\!&\!\!\!\!\!\!\!\!\!\!\!\!\!\!
(v_i\OTIMES v_\l\OTIMES \bar v_j)\bar x_1x_1
\nonumber\\&\!\!\!=\!\!\!&(v_i\OTIMES v_\l\OTIMES \bar v_j)x_1\bar x_1+
\mbox{$\sum\limits_{a\in I}$}(-1)^{1+[a]([i]+[j])}v_a\OTIMES E_{ij}v_\l\OTIMES  \bar v_a\nonumber\\&\!\!\!\!\!\!&
+\d_{ij}\mbox{$\sum\limits_{a\in I}$}(-1)^{[a][i]+[b][i]+([i]+[b])([a]+[i])}v_b\OTIMES E_{ab}v_\l\OTIMES  \bar v_a.
\end{eqnarray*}
Comparing this with \eqref{equa-ex} and \eqref{equa-xe}, we obtain $x_1 (e_1+\bar x_1)= (e_1+\bar x_1)x_1$.
 \end{proof}

\begin{lemma}\label{m-le3}
Let $x_1,\,\bar x_1,e_1$ be defined in \eqref{operator--1}. For any $k\in\Z^{\ge0}$, we have $e_1x_1^ke_1=\omega_ke_1$ for some $\omega_k\in\C$ such that
$\omega_0=m-n,$ $\omega_1=nq-mp$.\end{lemma}
\begin{proof}As above, we can suppose $r=t=1$.
For any $k\in\Z_+$, we have \begin{eqnarray}\label{e-x-e===}
&\!\!\!\!\!\!\!&\!\!\!\!\!\!\!\!\!\!\!\!\!\!(v_i\OTIMES v_\l\OTIMES \bar v_j)e_1x_1^ke_1\nonumber\\
&\!\!\!=\!\!\!&\!(-1)^{1+[i]}\d_{ij}\mbox{$\sum\limits_{\ell_0\in I}$}
(v_{\ell_0}\OTIMES v_\l\OTIMES \bar v_{\ell_0})x_1^k e_1\nonumber\\[-4pt]
&\!\!\!\!\!\!\!=\!\!\!\!\!\!\!\!&\!\d_{ij}\!\mbox{$\sum\limits_{\ell_0,\ell_1,...,\ell_k\in I}$}\!{\sc\!}({\ssc\!}-1{\ssc\!})^{\!\!1+[i]+\!\sum\limits_{p=0}^{k-1}\![\ell_p]([\ell_p]+[\ell_{p+1}])}\!\!
(v_{\ell_k}{\ssc\!}\OTIMES{\ssc\!} E_{\ell_{k-1},\ell_k}E_{\ell_{k-2},\ell_{k-1}}{\sc\!}{\sc\cdots} E_{\ell_0,\ell_1}v_\l{\ssc\!}\OTIMES{\ssc\!} \bar v_{\ell_0})e_1\nonumber\\
&\!\!\!\!\!\!\!=\!\!\!\!\!\!\!\!&\omega_k(v_i\OTIMES v_\l\OTIMES \bar v_j)e_1\mbox{ \ for some \ }\omega_k\in\C,
\end{eqnarray}
where the last equality is obtained as follows: if $\ell_k\ne \ell_0$,
the corresponding terms become zero after applying $e_1$ by \eqref{action-e}; otherwise %if $\ell_k=\ell_0$,
$E_{\ell_{k-1},\ell_k}E_{\ell_{k-2},\ell_{k-1}}\cdots E_{\ell_0,\ell_1}v_\l$
is a weight vector in $K_\l$ with weight $\l$, thus a multiple, denoted by $\omega_k$,  of $v_\l$.

In particular, if $k=0$, from the first equality of \eqref{e-x-e===}, we immediately obtain $\omega_0=m-n$
by \eqref{action-e}.
If $k=1$, from the second equality of \eqref{e-x-e===} and the above arguments, there is only one  factor
$E_{\ell_1,\ell_0}$ with $\ell_1=\ell_0$ we need to consider in the summand. Using \eqref{action-ofEij}, we obtain $\omega_1=nq-mp$.
\end{proof}

Now we can prove the following.
\begin{theorem} \label{realization-1}
Let $M=K_\l$ be the Kac-module with highest weight $\l=\l_{pq}$ in \eqref{l-pq} for $p,q\in\C$, and let $s_i,\,\bar s_j,\,x_1,e_1,\bar x_1\in {\rm End}_\mfg(M^{rt})$ be defined as in \eqref{operator--1}--\eqref{action-s1}.
Then all relations in Definition $\ref{wbmw}$ hold with $\omega_a$'s being specialized to the complex numbers
\begin{eqnarray}\label{www-}&&\omega_0=m-n,\ \ \ \ \ \ \ \omega_1=nq-mp,\nonumber\\
&&\omega_a=(p+q-m)\omega_{a-1}-p(q-m)\omega_{a-2}\mbox{ \ for }a\ge2.\end{eqnarray}
Furthermore, $x_1,\,\bar x_1$ satisfy
\begin{equation}\label{x1-barx1}(x_1-p)(x_1+m-q)=0,\ \ \ (\bar x_1+p-n)(\bar x_1+q)=0.\end{equation}
\end{theorem}

\begin{proof}
Note that those relations in Definition $\ref{wbmw}$ which do not involve $x_1$ and $\bar x_1$ are relations of the  walled Brauer algebra $\mathscr B_{r,t}(m-n)$ in Definition \ref{wbmwf}, thus hold by Proposition \ref{Without-M}.

Definition $\ref{wbmw}${$\ssc\,$}(\ref{(9)})--(\ref{(11)}),{$\ssc\,$}(\ref{(22)})--(\ref{(24)}) can be verified, easily.~By~Lemmas~\ref{m-le1}~and~\ref{m-le2},~we have
Definition $\ref{wbmw}${$\ssc\,$}(\ref{(8)}) and (\ref{(21)}).
Definition $\ref{wbmw}${$\ssc\,$}(\ref{(12)}) and the first two equations of
\eqref{www-} follows from Lemma \ref{m-le3}.
Similarly by symmetry, one can prove Definition
 $\ref{wbmw}${$\ssc\,$}(\ref{(25)}).

The last equation of
\eqref{www-} follows from  \eqref{x1-barx1} by induction on $a$. Note that the first equation of \eqref{x1-barx1} is \cite[IV, Corollary 3.2]{BS4}, which can also  be  obtained directly by \eqref{basis-M-} by noting that $x_1$ has two eigenvalues $p$ and $q-m$ as the summand in the second case of \eqref{basis-M-} is equal to $\sum_{a=1}^mw_1\OTIMES (E_{ia}(v_a\OTIMES v_\l)-v_i\OTIMES v_\l)\OTIMES w_2.$
Similarly, we have the second equation of \eqref{x1-barx1}.
%
%, we have  $w_0=(m-n)$, $e_1x_1e_1=(\sum_{p=1}^m\l_p-\sum_{p=m+1}^{m+n}\l_p)e_1$. Similarly, we have
%$$e_1\bar x_1^ae_1=\bar w_ae_1\mbox{ \ for some \ }\bar w_a\in\C\mbox{ \ with \ }\bar w_0=w_0.$$

%$x_1\bar x_1=\bar x_1x_1+e_1x_1-x_1e_1$,
%i.e., $$
% x_1(\bar x_1+e_1)=(\bar x_1+e_1)x_1,\ \ \ \ \bar x_1(x_1+e_1)=(x_1+e_1)\bar x_1.
%$$

To prove Definition $\ref{wbmw}${$\ssc\,$}(\ref{(13)}),
let $x_i$ with $2\!\le\! i\!\le\! r$ be defined as in  \eqref{X-bar-x}. Then $x_i$'s are the exactly same as that in \cite[IV, Lemma 3.1]{BS4}, in particular, they commute with each other, i.e., we have Definition~$\ref{wbmw}${$\ssc\,$}(\ref{(13)}) as $x_2$ coincides with $s_1x_1s_1-s_1$. Similarly we have Definition~$\ref{wbmw}${$\ssc\,$}(\ref{(26)}).
\end{proof}

Denote by $\DBr$
the subalgebra of  ${\rm End}_\mfg(M^{rt})$ generated
by $s_i,\bar s_j,x_1,\bar x_1,e_1$.
Theorem \ref{realization-1} shows that $\DBr$ is a quotient algebra of the affine walled Brauer algebra
$\mathscr B_{r,t}^{\text{aff}}$ with specialized parameters \eqref{www-}.

We need to introduce the following notion in order to prove the next result.
\begin{definition}\label{degree-B-M}\rm
%First we need some notions.
For an element $b\!=\!b^\si\!\in\! B$ as in \eqref{basis-k-l}, we denote $|b|=\sum_{i,j}\si_{i,j}$, called the {\it degree} of $b$. If $\si_{ij}\ne0$, we say $E_{i+m,j}$ is a {\it factor} of $b$. For $b_M\in B_M$, we define its {\it degree} $|b_M|$ to be $|b|$, where $b\in B$ is a unique tensor factor of $b_M$.
\end{definition}

For any $\alpha=(\alpha_1,...,\alpha_r)\in\{0,1\}^r$,
$\beta=(\beta_1,...,\beta_t)\in\{0,1\}^t$, we define the following elements of $\DBr$:
\begin{equation}\label{x-al-bx-b}
x^\alpha=\mbox{$\prod\limits_{i=1}^rx_i^{\alpha_i},\ \ \ \bar x^\beta=\prod\limits_{j=1}^t\bar x_j^{\beta_j},$}
\end{equation}
where $x_i,\,\bar x_j$ are elements of $\DBr$
defined as in \eqref{X-bar-x}.
\begin{theorem}\label{theo-2222}We keep the assumption of Theorem $\ref{realization-1}$, and
assume $r+t\le$ $\min\{m,n\}$.  Then the monomials
\begin{equation}\label{monom}
{\textit{\textbf{m}}}:=c^{-1}x^\alpha e^f \bar x^{\beta}wd,\end{equation}
with $\alpha\in\{0,1\}^r,\,\beta\in\{0,1\}^t$ and $c,e^f,w,d$  as in
Theorem $\ref{basis}$ and Definition $\ref{def-mo}$,
 are $\mathbb C$-linearly independent
endomorphisms of $M^{rt}$.
\end{theorem}
\def\tif#1{{\textit{\textbf{#1{$\ssc\,$}}}}}\begin{proof}First we remark that
for convenience we arrange factors of
the monomial $\tif m$ in \eqref{monom}
in a different order from the corresponding monomial $\tif m$ in
\eqref{Monooo} (without factors $\omega_a$'s,~cf.~Definition \ref{reduced-mo}). Note that changing the order only differs an element by some element with lower degree, where the degree of $\tif m$ is defined to be ${\rm deg\,}\tif m:=|\a|+|\b|$, and $|\alpha|=\sum_{i=1}^r\alpha_i, $ $|\beta|=\sum_{i=1}^t\beta_i$.

Suppose there is a nonzero $\mathbb C$-combination $\tif c:=\sum_{\tif m}r_{\tif m}\tif m$ of
monomials \eqref{monom} being zero. We fix a monomial
$\tif m':=c'^{-1}x^{\alpha'} e^{f'} \bar x^{\beta'}w'd'$ in $\tif c$ with nonzero coefficient $r_{\tif m'}\ne0$ which satisfies the following conditions:\begin{itemize}\item[(i)] $|\alpha'|+|\beta'|$ is maximal;
\item[(ii)] $f'$ is minimal among all monomials satisfying (i).
%(iii)  is maximal among all monomials satisfying (i) and (ii).
\end{itemize}
We take the basis element $v=\OTIMES_{i\in J_1}v_{k_i}\OTIMES v_\l\OTIMES\OTIMES_{i\in J_2}\bar v_{k_i}\in B_M$ (cf.~\eqref{basis-M-}) such that (note that here is the place where we require condition $r+t\le\min\{m,n\}$)
\begin{itemize}\item[(1)]
 $k_i=i+\alpha'_i m$ if $i\prec 0$; \item[(2)] $k_{\bar i}=i$ for $1\le i\le f'$; \item[(3)]
$k_{\bar i}=r+i+(1-\beta'_i)m$ if $f'<i\le t$.\end{itemize}
We define $p_v$ to be the maximal integer such that there exist $p_v$ pairs
$(i,\bar j)\in J_1\times J_2$ satisfying $k_i\!-\!k_{\bar j}\!\in\!\{0,\pm m\}$.
Then from the choice of $v$, we have the following fact:
\begin{equation}\label{Factsss}p_v=f'.\end{equation}

Now take  \begin{equation}\label{b'===}u:=(v)c'{\tif c}d'^{-1}w'^{-1}\in M^{rt},\ \ \mbox{and}\ \ b'=\mbox{$\prod\limits_{i=1}^rE_{i+m,i}^{\alpha'_i}\prod\limits_{i=1}^t$}E_{r+i+m,i+r}^{\beta'_i}v_\l\in K_\l,\end{equation}
such that $b'$ is a basis element in $B$ with degree $|\alpha'|+|\beta'|$ by noting that $0\le \alpha'_i,\,\beta'_i\le1$.
We denote $B^{b'}_M$ to be the subset of $B_M$ consisting of elements with $b'$ being a tensor factor.
We define the projection $\hat\pi_{b'}:M^{rt}\to\otimes_{i\in J\bs\{0\}}V_i$ (cf.~\eqref{M-st==}) by mapping a basis element $b_M\in B_M$ to zero if $b_M\notin B^{b'}_M$, or else  to the element obtained from $b_M$ by deleting the tensor factor $b'$. Motivated by \cite[IV, Corollary 3.3]{BS4},
we refer to $\hat\pi_{b'}(u)$ as the {\it $b'$-component} of $u$. We want to prove $\hat\pi_{b'}(u)\ne0$.

Assume a monomial $\tif m$ in \eqref{monom} appears in the expression of $\tif c$ with $r_{\tif m}\ne0$. Consider the following element of $M^{rt}$ which contributes to $u$ in \eqref{b'===},
\begin{eqnarray}\label{u-11111}
u_1&\!\!\!:=\!\!\!&(v)c'{\tif m} d'^{-1}w'^{-1}=
(v)c'c^{-1}x^\alpha e^f\bar x^\beta w d d'^{-1}w'^{-1}\nonumber\\
&\!\!\!=\!\!\!&
\Big(\,\OT{i\in J_1}v_{k_{(i)c'c^{-1}}}\OTIMES v_\l\otimes\OT{i\in J_2}\bar v_{k_{(i)c'c^{-1}}}\Big)x^\alpha e^f\bar x^\beta w d d'^{-1}w'^{-1},
\end{eqnarray}
where the last equality follows by noting that elements in $\mathfrak{S}_r\times\bar{\mathfrak S}_t$ have natural right actions on $J_1\cup J_2$ by permutations.
Write $u_1$ as a $\mathbb C$-combination of  basis $B_M$, and for $b_M\in B_M$, if $b_M$ appears as a term with a nonzero coefficient in the combination, then we say that $u_1$ {produces} $b_M$.
By \eqref{action-x1}, \eqref{action-x1*}, Definition \ref{degree-B-M} and condition (i), $u_1$
cannot produce a basis element with degree higher than
$|\alpha|+|\beta|$.
Thus the {$b'$-component} of $u_1$ is zero if $|\alpha|+|\beta|<|\alpha'|+|\beta'|$.
So we can assume $|\alpha|{\!}+{\!}|\beta|{\!}={\!}|\alpha'|{\!}+{\!}|\beta'|$ by condition (i).~Then
$f{\!}\ge{\!} f'$ by condition~(ii).

Note from definitions \eqref{X-bar-x} and \eqref{pi-ab} that $$x_i=\pi_{i0}(\Omega)|_{M^{rt}}+\mbox{ some element of degree zero},$$ and
$e^f=e_1\cdots e_f$ and $e_i=\pi_{i\bar i}(\Omega)|_{M^{rt}}$ (cf.~\eqref{operator--1}). By \eqref{action-e}, we see that in order for $u_1$ in \eqref{u-11111} to produce a basis element $b_M$ in $B_M^{b'}$ (note that $b_M\in B_M^{b'}$ has tensor factor $b'$ and all factors of $b'$ have the form $E_{i+m,i}$ by \eqref{b'===}), we need at least $f$ pairs $(i,\bar j)\in$ $ J_1\times J_2$ with $k_i-k_{\bar j}\in\{0,\pm m\}$ by \eqref{action-e}. Thus we can suppose $f=f'$ by \eqref{Factsss} and the fact that $f\ge f'$.

Set $J_{f'}=(J_1\cup J_2)\cap\{i\,|\,f'\preceq i\preceq\bar f'\}$ (cf.~\eqref{ordered-set}).
If $c\ne c'$, then by definition \eqref{rcs11}, we have \begin{equation}\label{j---in}j':=(j)c'c^{-1}\notin J_{f'}\mbox{ \ for some $j\in J_{f'}$}.\end{equation} Say $j'\!\in\! J_1$ (the proof is similar if $j'\!\in\! J_2$), then $f'\!<\!j'\!\le\! r$.
Condition (1) shows that either $f'\!<\!k_{j'}\!=\!j'\!\le\! r$ or else $f'\!+\!m\!<\!k_{j'}\!=\!j'\!+\!m\!\le\! r\!+\!m$.
Then conditions (2) and (3) show that there is no $\bar\ell\in J_2$ with $k_{j'}-k_{\bar\ell}\in\{0,\pm m\}$. Since all factors of $b'$ have the form $E_{i+m,i}$, we see that $u_1$ cannot produce a basis element in $B^{b'}_M$.
Thus we can suppose $c=c'$.

By conditions (1) and (2), we see that if $\alpha_i\ne\alpha'_i$ for some $i$ with $1\le i\le f$, or $\alpha_i=1\ne\alpha'_i$ for some $i\in J_1$, then again
$u_1$ cannot produce a basis element in $B^{b'}_M$.~Thus we suppose: $\alpha_i\!=\!\alpha'_i$ if $1\!\le\! i\!\le\! f$, and
$\alpha'_i\!=\!0$ implies $\alpha_i\!=\!0$ for~$i\!\in\! J$.

Consider the coefficient $\chi^{u_1}_{\tilde b_M}$ of the basis element $\tilde b_M:=\OTIMES_{i\in J_1}v_i\OTIMES b'\OTIMES\OTIMES_{i\in J_2}\bar v_{i+m}$ in $u_1$. If $\alpha'_i=1$ but $\alpha_i=0$ for some $i\in J$, then $u_1$ can only produce some basis elements which have at least a tensor factor, say $v_\ell$, with $\ell>m$, and thus $\tilde b_M$ cannot be produced.~Thus we can suppose $\alpha\!=\!\alpha'$.~Dually, we can suppose~$\beta'\!=\!\beta$.

Now rewrite $w d d'^{-1}w'^{-1}$ as $w d d'^{-1}w'^{-1}=\tilde d {\tilde d}'^{-1}w''$, where $\tilde d=wdw^{-1}$,
${\tilde d}'=$ $wd'w^{-1}$ and $w''=ww'^{-1}$. Note that %$\tilde d,\tilde d'$ are still in $\mathscr{D}_{r,t}^{f'}$, and
$w''\in \mathfrak S_{r-f'} \times \bar{\mathfrak S}_{t-f'}$, which only permutes elements of $(J_1\cup J_2)\bs J_{f'}$.
We see that if $\tilde d\ne \tilde d'$, then as in \eqref{j---in}, there exists some $j\in J_{f'}$ with
$j':=(j)\tilde d {\tilde d}'^{-1}w''\notin J_{j'}$, thus $\tilde b_M$ cannot be produced. So assume $\tilde d=\tilde d'$. Similarly we can suppose $w''=1$.

The above has in fact proved that if the coefficient $\chi^{u_1}_{\tilde b_M}$ is nonzero then  $u_1$ in \eqref{u-11111} must satisfy $(c,\alpha,f,\beta,d,w)=(c',\alpha',f',\beta',d',w')$, i.e.,
$u_1=(v)x^{\alpha'} e^{f'}\bar x^{\beta'}$. In this case, one can easily verify that $\chi^{u_1}_{\tilde b_M}=\pm1$. This proves that $u$ defined in \eqref{b'===} is nonzero, a contradiction. The theorem is proven.\end{proof}

As in \cite[IV]{BS4}, we shall be mainly interested in the case when the Kac module $K_\l$ is typical, namely, either $p-q\notin\Z$ or $p-q\le-m$ or $p-q\ge n$. In this case, the tensor module $M^{rt}$ is a tilting module. Using the vector space isomorphism ${\rm Hom}_{\mfg}(M_1\otimes V,M_2)\cong
{\rm Hom}_{\mfg}(M_1, M_2\otimes V^*)$ for any two $\mfg$-modules $M_1,M_2$, one can easily obtain the vector space isomorphism: ${\rm End}_{\mfg}(M^{rt})\cong$ $
{\rm End}_{\mfg}(K_\l\otimes V^{\otimes(r+t)})$. Thus
${\rm dim_\C{\sc\,}}{\rm End}_{\mfg}(M^{rt})=2^{{\ssc\,}r+t}(r+t)!$ by \cite[IV]{BS4}, which is the same as the total number of all monomials of the form \eqref{monom}.

Now we can state the main result of this section.
%Therefore, the above theorem  proves the following main result of this section.
\begin{theorem}\label{level-2} {\rm(Super Schur-Weyl duality)}
Assume $r\!+\!t\!\le\!\min\{m,n\}$, and $p\!-\!q\!\notin\!\Z$ or $p\!-\!q\!\le\!-m$ or $p\!-\!q\!\ge\!n$. We %  With the assumption of Theorem $\ref{theo-2222}$, we
have
% If $r+t\le\min\{m,n\}$ then
${\rm End}_{\mfg}(M^{rt})=\DBr$ and \begin{equation}\label{L2-wb}\DBr\cong\mathscr B_{r,t}^{\text{aff}}/\langle
(x_1-p)(x_1+m-q),(\bar x_1+p-n)(\bar x_1+q)\rangle.\end{equation}
\end{theorem}
\begin{proof} Recall that $\DBr$ is a subalgebra of ${\rm End}_{\mfg}(M^{rt})$ defined before Definition~\ref{degree-B-M}. By Theorem~\ref{theo-2222}, $\dim_{\mathbb C}\DBr\ge 2^{r+t} (r+t)!=\dim_{\mathbb C} {\rm End}_{\mfg}(M^{rt})$, forcing ${\rm End}_{\mfg}(M^{rt})=\DBr$.

Denote  the right-hand side of \eqref{L2-wb} by $A$.
By Theorem~\ref{main2} and arguments on the degree, we see easily that the (image of) monomials in \eqref{monom} span
$A$. Thus ${\rm dim}_\C{\sc\,}A\le $ $2^{{\ssc\,}r+t}(r+t)!$ (which is the number of monomials in \eqref{monom}).
%By Theorem~\ref{theo-2222}, ${\rm dim}_\C{\sc\,}A= $ $2^{{\ssc\,}r+t}(r+t)!$

%$\mathscr B_{r,t}^{\text{aff}}/\langle (x_1-p)(x_1+m-q),(\bar x_1+p)(\bar x_1+q-n)\rangle$.
%By Theorem~\ref{theo-2222}, such monomials form a basis of this quotient algebra.
Using Theorem~\ref{realization-1} yields an epimorphism from $\mathscr B_{r,t}^{\text{aff}}$ to ${\rm End}_{\mfg}(M^{rt})$  killing the two-sided ideal $\langle
(x_1\!-\!p)(x_1\!+\!m\!-\!q),(\bar x_1\!+\!p\!-\!n)(\bar x_1\!+\!q)\rangle$ of   $\mathscr B_{r,t}^{\text{aff}}$,
thus induces an epimorphism from $A$ to $\DBr$. Thus
 ${\rm dim}_\C{\sc\,}A\ge \dim_{\mathbb C}\DBr$. Comparing their dimensions, we have  the isomorphism in \eqref{L2-wb}, as required.
\end{proof}Because of  Theorem~\ref{level-2}, we refer to the right-hand side of \eqref{L2-wb} as a
level two walled Brauer algebra, defined below.
\begin{definition}\label{level2}Let $\mathscr B_{r,t}^{\text{aff}}$ be the affine walled Brauer algebra
defined over $\C$ with specialized parameters \eqref{www-}.
The cyclotomic quotient associative $\C$-algebra $\mathscr B_{r,t}^{\text{aff}}/\langle
(x_1\!-\!p)(x_1\!+\!m\!-\!q),(\bar x_1\!+\!p\!-\!n)(\bar x_1\!+\!q)\rangle$ is called a {\it level two walled Brauer algebra}.
 By abusing of notations, we denote this algebra by  $\DBr$.
\end{definition}
%In general, for any $m,n,r,t\in\Z^{\ge1},\,p,q\in\C$, we define the level two  Brauer algebra $\DBr$ by \eqref{L2-wb},
By arguments in the proof of Theorem~\ref{level-2}, we obtain

%Then by noting that the proof of Theorem \ref{theo-2222} does not depend on whether
%or not the Kac module $K_\l$ is typical, we obtain
\begin{corollary}\label{coro-222}
Let $m,n,r,t\in\Z^{\ge1},\,p,q\in\C$ such that $r+t\le\min\{m,n\}$. Then the level two walled  Brauer algebra $\DBr$ is of dimension $2^{r+t}(r+t)!$ over $\C$ with all monomials of the form \eqref{monom} being a $\C$-basis of $\DBr$.\end{corollary}

\begin{remark}\label{rema-wll2}Note that level two walled  Brauer algebras $\DBr$ heavily depend on parameters $p\!-\!q,\,r,t,m,n$, in sharp contrast to level two Hecke algebras $H_r^{p,q}$ in \cite[IV]{BS4} (denoted as $\mathscr H_{2, r}$ in the present paper), which only depend on $p\!-\!q,\,r$.
\end{remark}
We are now going to determine when $\DBr$ is semisimple. For this purpose, we need the following result, which
is a slight generalization of \cite[Lemma 5.2]{SHK} and \cite[Lemma 3.6]{SZ4},
% Let
%$\Lambda_2^+(k)$ be the set of all bipartitions of $k$. If $r+t\le{\rm min}\{m,n\}$, then for any
% bipartitions  $(\eta^{(1)},\eta^{(2)})$ and $(\zeta^{(1)},\zeta^{(2)})$ of $r$ and $t$ respectively,
% we have $\eta^{(1)}_i=\eta^{(2)}_j=
% \zeta^{(1)}_i=\zeta^{(2)}_j=0$ for all $i>m$ and $j>n$. Thus they correspond to two weights
%$(\eta^{(1)}\,|\,\eta^{(2)}),\,(\zeta^{(1)}\,|\,\zeta^{(2)})\in\fh^*$ (cf.~\eqref{weight-}).
where a {\it $\mfg$-highest weight} of $M^{rt}$ means a weight $\mu\in\fh^*$ such that there exists a nonzero {\it $\mfg$-highest weight
vector} $v\in M^{rt}$ with weight $\mu$ (i.e., $v$ is a vector satisfying $E_{ii}v=\mu_iv,\,E_{ji}v=0$ for $1\le i<j\le m+n$).

\begin{lemma}\label{Highest-w}
We keep  the assumption of Theorem $\ref{theo-2222}$, and assume $\mu\in\fh^*$
is a $\mfg$-highest weight of $M^{rt}$. Then $|\mu|=|\l|+r-t$ and
$-t\le\sum_{i\in S}(\mu_i-\l_i)\le r$ for any subset $S\subset I$, where $|\l|,\,|\mu|$ are sizes of $\l,\,\mu$ $($cf.~\eqref{lavel-}$)$.
%(\eta^{(1)}\,|\,\eta^{(2)})-(\zeta^{(1)}\,|\,\zeta^{(2)})$
%such that $(\eta^{(1)},\eta^{(2)})$ and $(\zeta^{(1)},\zeta^{(2)})$ are  bipartitions  of $r$ and $t$ respectively.
\end{lemma}
\begin{proof}
Let $w_\mu$ be  a $\mfg$-highest weight vector  with weight $\mu$, and write $w_\mu$ in terms of basis $B_m$ in \eqref{basis-M-}. As in the proofs of
\cite[Lemma 5.2]{SHK} and \cite[Lemma 3.6]{SZ4}, $w_\mu$ must contain a basis element, say $b_M$, with degree $0$ (cf.~Definition \ref{degree-B-M}), i.e., $b_M$ has the form
 $w_1\OTIMES v_\l\OTIMES w_2$
for some $w_1\in V^{\otimes r}$ and $w_2\in (V^*)^{\otimes t}$ such that $w_1$ (resp., $w_2$) is a weight vector with some weight $\eta$ (resp., $\zeta$)
of size $r$ (resp., $-t$) satisfying $\eta_i\in\Z^{\ge0}$ (resp., $\zeta_i\in\Z^{\le0}$) for all $i\in I$. The result follows.
\end{proof}
\begin{theorem}\label{Semi-simple}
  We keep the assumption of Theorem $\ref{level-2}$,  then $\DBr$ is semisimple
 if and only if $p\!-\!q\!\notin\!\Z$ or $p\!-\!q\!\le-m\!-\!r$ or $p\!-\!q\!\ge\!n\!+\!t$.
\end{theorem}
\begin{proof}
First assume $p-q\in\Z$ and $p-q\le-m-r$.
Let $\mu\in\fh^*$ be a $\mfg$-highest weight of $M^{rt}$.
For $1\!\le\! i\!\le\! m\!<\!j\!\le\! m\!+\!n$, by definition of \eqref{rho-l}, we have
(hereafter we define the partial order on $\C$ such that $a\le b$ if and only if $b-a$ is a nonnegative real number)
\begin{eqnarray*}\mu^\rho_i+\mu^\rho_j&\!\!\!=\!\!\!&\mu_i+\mu_j+1+2m-i-j\\
&\!\!\!\le\!\!\!&\l_i+\l_j+r+1+2m-1-(m+1)\\
&\!\!\!=\!\!\!&p+m+r-q-1,\end{eqnarray*}
which is strictly less  than zero,
i.e., $\mu$ is a typical integral dominant weight, where the  inequalty follows from Lemma \ref{Highest-w}
and $i\ge1,\,j\ge m+1$.
 By \cite{Kac77}, $M^{rt}$ is a completely reducible module which can be decomposed as a direct sum of typical finite dimensional irreducible modules:
$M^{st}=\oplus_{\mu\in T} L_\mu^{\oplus k_\mu}$, where $T$ is a finite set consisting of typical
integral dominant weights, and $k_\mu\in\Z^{\ge1}$. Thus $\DBr\cong{\rm End}_{\mfg}(M^{rt})\cong
\oplus_{\mu\in T}M_{k_\mu}$ is a semisimple associative algebra, where $M_{k_\mu}$ is
the algebra of matrices of rank $k_\mu$.
The case $p-q\notin\Z$ or $p-q\ge n+t$ can be proven similarly.

Now suppose
$p{\sc\!}-{\sc\!}q{\sc\!}\in{\sc\!}\Z$ and $q{\sc\!}-{\sc\!}m{\sc\!}-{\sc\!}r{\sc\!}<{\sc\!}p{\sc\!}<{\sc\!}q{\sc\!}+{\sc\!}n{\sc\!}+{\sc\!}t$. This together with condition \eqref{Do-condi} shows that either $|p{\sc\!}-{\sc\!}q{\sc\!}+{\sc\!}m|$, the absolute value of difference of two eigenvalue of $x_1$, is an integer $<r$, or,
$|p{\sc\!}-{\sc\!}q{\sc\!}-{\sc\!}n|$, the absolute value of difference of two eigenvalue of $\bar x_1$, is an integer $<t$. Thus by \cite[Theorem~6.1]{AMR},
either the level two degenerate Hecke algebras $\mathscr H_{2, r}\!:=\!\mathscr H_{r}^{\text{aff}}/\langle (x_1\!-\!p)(x_1\!+\!m\!-\!q)\rangle$
(cf. Proposition \ref{epi}{$\sc\,$}(3)) is not semisimple or else
$\bar{\mathscr H}_{2, t}\!:=\!\bar{\mathscr H}_{t}^{\text{aff}}/\langle (\bar x_1\!+\!p\!-\!n)(\bar x_2\!+\!q)\rangle$ is not semisimple. In any case, $\mathscr H_{2, r}\otimes\bar{\mathscr H}_{2, t}=\DBr/\langle e_1\rangle$ is not semisimple.
As a result, $\DBr$, the preimage of $\mathscr H_{2, r}\otimes\bar{\mathscr H}_{2, t}$, cannot be semisimple.
%
%In this case we can use Theorem 2.14 and Corollary 2.9 of \cite{BS4}
%to prove that $M^{rt}$ contains a direct summand of the projective module $P_\mu$ (the projective cover of $K_\mu$) for some atypical integral dominant weight $\mu$. Thus $M^{rt}$ is not completely reducible and therefore $\DBr\cong{\rm End}_{\mfg}(M^{rt})$
%is not semisimple. In fact, one can simply prove that $M^{rt}$ is not completely reducible as follows:
%Let $w_\nu:=v_1^{\otimes r}\otimes v_\l\otimes\bar v_{m+n}^{\otimes t}\in M^{rt}$ with weight $\nu:=\l+(r,0,...,0\,|\,0,...,0,-t)$ (which is the maximal weight of $M^{st}$ by the lexicographical order,
%thus in particular, $w_\nu$ is a $\mfg$-highest weight vector). Then $w_\nu$ generates a submodule isomorphic to $K_\nu$. Note that $\nu$ is an atypical weight since we have $\nu_1=p+r$ and for $k:=m+p+r-q$, we have $m<k\le m+r\le m+n$ and $\nu_k=-q-m$ if $k<m+n$ or $\nu_k=-q-m-t=-q-m$ if $k=m+n$
%(by noting that if $k=m+n$ then $r=n$ and thus $t=0$ as $r+t\le n$), and
%$\nu_1^\rho+\nu^\rho_{k}=\nu_1+\nu_k+2m-k=0$. Thus $K_\nu$ is an indecomposable and reducible submodule of $M^{rt}$ by \cite{Kac77}.
\end{proof}

In the next two sections, we shall study $\DBr$ for given $m,n$ (with the assumption $r\!+\!t\!\le\! \min\{m, n\}$), thus we omit $(m,n)$ from the notation, and simply
denote it by %$\DBr$ by
\def\DBr{{\mathscr B_{r,t}^{p,q}}}$\DBr$. When there is no confusion, the notation is further simplified to
${\mathscr  B}$.

\def\DBr{{\mathscr  B}}%
\let\gdom\rhd%
\let\gedom\unrhd%
\def\a{\mathfraxk a}%
\def\m{\mathfrak m}%
\def\floor#1{\lfloor\tfrac#1\rfloor}%
\def\UPD{\mathscr{T}^{ud}}%
\def\Std{\mathscr{T}^{std}}%
\def\m{\mathfrak m}%
\def\n{\mathfrak n}%
\def\s{\mathfrak s}%
\def\ts{\tilde\s}%
\def\t{\mathfrak t}%
\def\u{\mathfrak u}%
\def\v{\mathfrak v}%
\def\bfs{\s}%
\def\bft{\t}%
\def\F{\mathcal F}%
\def\G{\mathcal G}%
\def\Hom{\text{Hom}}%
\def\Ind{\text{Ind}}%
\def\Res{\text{Res}}%
\def\U{\mathbf U}%
\def\textsf#1{{\textit{#1}}}%
\section{Weakly cellular basis of level two walled Brauer algebras}\label{levelw-2}
In this section, we shall use Corollary \ref{coro-222} to construct a weakly cellular basis of $\DBr=\mathscr  B_{r,t}^{p,q}(m,n)$ over $\mathbb C$
for $m,n,r,t\in\Z^{\ge1},\,p,q\in\C$ such that $r+t\le\min\{m,n\}$.
%
%Throughout this section we consider level two walled Brauer algebras $\DBr$ over the complex field with special parameters given in Theorem~\ref{realization-1}. We construct a cellular basis
%of  $\DBr$ and use standard results on cellular algebras in \cite{GL} to classify the irreducible $\DBr$-modules.
%In section~\ref{decmpw}, we will  set up the relationship between Kac modules
% of $\mfg$ and the cell modules of $\DBr$. Finally, we will determine the decomposition matrices for   $\DBr$.
%%%%%
%%%%%%Recall that a   \textsf{partition} of $m$ is a sequence of non--negative integers
%%%%%%$\lambda=(\lambda_1,\lambda_2,\dots)$ such that $\lambda_i\ge
%%%%%%\lambda_{i+1}$ for all positive integers $i$ and
%%%%%%$|\lambda|:=\lambda_1+\lambda_2+\cdots=m$. Similarly, a
%%%%%%\textsf{bipartition} of $m$ is an ordered $2$-tuple
%%%%%% of partitions $\lambda=(\lambda^{(1)},\lambda^{(2)})$ such that
%%%%%%$|\lambda|:=|\lambda^{(1)}|+|\lambda^{(2)}|=m$. Let
%%%%%%$\Lambda_2^+(m)$ be the set of all bipartitions of $m$.

First recall that a   \textit{partition} of $k\in\Z^{\ge0}$ is a sequence of non-negative integers
$\lambda=(\lambda_1,\lambda_2,\dots)$ such that $\lambda_i\ge
\lambda_{i+1}$ for all positive integers $i$ and
$|\lambda|:=$ $\lambda_1+\lambda_2+\cdots=k$. Let
$\Lambda^+(k)$ be the set of all partitions of $k$.
A
\textit{bipartition} of $k$ is an ordered $2$-tuple
 of partitions $\lambda=(\lambda^{(1)},\lambda^{(2)})$ such that
$|\lambda|:=|\lambda^{(1)}|+|\lambda^{(2)}|=k$.

Let $\Lambda_2^+(k)$ be the set of all bipartitions of $k$.
Then  $\Lambda_2^+(k)$ is a poset with the {\it dominance order} $\trianglerighteq$ as the partial order on it. More explicitly, we say
$\lambda=(\lambda^{(1)}, \lambda^{(2)}) $ is dominated by $\mu=(\mu^{(1)}, \mu^{(2)})$ and
write $\mu\trianglerighteq\l $~if
\begin{equation}\label{par}\mbox
{$\sum\limits_{j=1}^i \lambda_j^{(1)}\le  \sum\limits_{j=1}^i \mu_j^{(1)}$ \ \ and   \ \ $|\lambda^{(1)}|+\sum\limits_{j=1}^\ell \lambda_j^{(2)}\le |\mu^{(1)}|+
\sum\limits_{j=1}^\ell \mu_j^{(2)}$},\end{equation}
 for all possible $i, \ell$'s.
We write $\mu\vartriangleright\l$ if $\mu\trianglerighteq\l$ and $\l\ne\mu$.

For each partition $\lambda$ of $k$,  the {\it Young diagram} $[\lambda]$ is a
collection of boxes arranged in left-justified rows with $\lambda_i$
boxes in the $i$-th row of $[\lambda]$. If  $\lambda=(\lambda^{(1)},\lambda^{(2)})\in \Lambda_2^+(k)$,
then the corresponding Young diagram $[\lambda]$ is $([\lambda^{(1)}], [\lambda^{(2)}])$.
In this case,  a {\it $\lambda$-tableau} $\s=(\s_1, \s_2)$ is
obtained by inserting $i,\, 1\le i\le k$ into $[\lambda]$ without
repetition.

 A $\lambda$-tableau $\s$ is said to be {\it standard} if the
entries in  $\s_1$ and $\s_2$    increase both from left to right in each row
and from top to bottom in each column. Let $\Std(\lambda)$ be the
set of all standard $\lambda$-tableaux.
\begin{definition}\label{t-lambda}We define\begin{itemize}\item
$\t^\lambda$ to be the $\lambda$-tableau obtained from the Young
diagram $[\lambda]$ by adding $1, 2, \cdots, k$ from left to right
along the rows of $[\lambda^{(1)}]$ and then $[\lambda^{(2)}]$;
\item
$\t_{\lambda}$ to be the $\lambda$-tableau obtained from
$[\lambda]$ by adding $1, 2, \cdots, k$ from top to bottom along the columns of $[\lambda^{(2)}]$ and then
$[\lambda^{(1)}]$.
\end{itemize}
\end{definition}
For
example, if $\lambda=((3,2,1), (2,1))$, then
\begin{equation}\label{tla}
\t^{\lambda}=\left( \ \ \young(123,45,6),\  \young(78,9)\ \ \right) \quad \text{ and \ }
 \t_{\lambda}=\left(\ \ \young(479,58,6), \ \young(13,2)\ \ \right).\end{equation}

The symmetric group  $\mathfrak S_k $  acts on a $\lambda$-tableau
$\s$ by permuting its entries. For $w\in\mathfrak S_k $, if $\t^\lambda w=\s$,
we write $d(\s)=w$. Then  $d(\s)$ is
uniquely determined by $\s$.

Given a $\lambda\in \Lambda_{2}^+(k)$, let $\mathfrak S_\lambda$ be the
row stabilizer of $\t^\lambda$. Then $\mathfrak S_\lambda$ (sometimes denoted as $\mathfrak S_{\bar\lambda}$) is the
Young  subgroup of $\mathfrak S_k$ with respect to the composition $\bar \lambda$, which is
obtained from $\lambda$ by concatenation.
For example, $\bar \lambda=(3,2,1, 2, 1)$ if $\lambda=((3,2,1), (2,1))$.

Recall that  $\mathscr H_r^{\rm aff}$ is the degenerate affine Hecke algebra  generated by $S_i$'s and $Y_j$'s (cf. Definition~\ref{degenerate H}).
Let $\mathscr H_{2, r}=\mathscr H_r^{\rm aff}/I$, where $I=\langle  (Y_1-u)(Y_1-v)\rangle $ is  the two-sided ideal of $\mathscr H_r^{\rm aff}$ generated by $(Y_1-u)(Y_1-v)$ for $u,v\in\C$.
Then $\mathscr H_{2, r}$ is  known as the level two degenerate Hecke algebra with defining parameters $u, v$. As mentioned in Proposition \ref{epi},
%Section~3,
our current elements
$x_1, \bar x_1$,  which are two  generators of  $\mathscr B_{r, t}^{\rm aff}$,  corresponds   $-Y_1$ in Definition~\ref{degenerate H}. Thus,  when we use the construction of
cellular basis   for $\mathscr H_{2, r}$ in \cite{AMR}, we need to use $-Y_1, -u, -v$ instead of $Y_1, u, v$, respectively.
%So, the current elements $u$ and $v$ should be replaced by $-u$ and $-v$, respectively if
%we use the construction of cellular basis for  $\mathscr H_{2, r}$ in \cite{AMR}.
By abusing of notations, we will not distinguish between them. The following definition  on $\m^\l_{\s \t}$  is a special case of that in \cite{AMR} for degenerate Hecke algebra $\mathscr H_{r, m}$ of type $G(r, 1, m)$.

Suppose  $\lambda=(\lambda^{(1)}, {\lambda^{(2)}})\in \Lambda_{2}^+(r)$ with $a=|\lambda^{(1)}|$.
We set $\pi_0=1$ if $a=0$, and
 $\pi_a=$ $(x_1-v)(x_2-v)\cdots ( x_a-v)$ for $a\ge 1$.  Let
\begin{equation}\label{Aff-basis}
\m^\l_{\s\t}=d(\s)^{-1} \pi_a \m_{\bar \lambda}  d(\t),  \end{equation}
where $\s, \t\in \Std(\lambda)$ and
 $\m_{\bar \lambda}=\sum_{w\in \mathfrak S_{\bar \lambda}} w $. In the following, we shall always omit the $\l$ from notations $\m^\l_{\s\t},\,C^\l_{\s\t}$, etc., and simply denote them as $\m_{\s\t},\,C_{\s\t}$, etc.

\begin{definition}\cite{GL}\label{GL}
    Let  $A$ be  an algebra over a commutative ring $R$ containing $1$.
    Fix a partially ordered set $\Lambda=$ $(\Lambda,\gedom)$, and for each
    $\lambda\in\Lambda$, let $T(\lambda)$ be a finite set. Further,
    fix $C_{\s\t}\in A$ for all
    $\lambda\in\Lambda$ and $\s,\t\in T(\lambda)$.
    Then the triple $(\Lambda,T,C)$ is a \textit{cell datum} for $A$ if:
    \begin{enumerate}
    \item $\mathcal C:=\{C_{\s\t}\,|\,\lambda\in\Lambda,\,\s,\t\in
        T(\lambda)\}$ is an $R$-basis for $A$;
    \item the $R$-linear map $*: A\rightarrow A$ determined by
        $(C_{\s\t})^*=C_{\t\s}$ for all
        $\lambda\in\Lambda$ and all $\s,\t\in T(\lambda)$ is an
        anti-involution of $A$;
    \item for all $\lambda\in\Lambda$, $\s\in T(\lambda)$ and $a\in A$,
        there exist scalars $r_{\t\u}(a)\in R$ such that
        $$C_{\s\t} a
            =\SUM{\u\in T(\lambda)}{}r_{\t\u}(a)C_{\s\u}
                     \pmod{A^{\gdom\lambda}},$$
            where
    $A^{\gdom\lambda}=R\text{-span}%
      \{C^\mu_{\u\v}\,|\,\mu\gdom\lambda,\,\u,\v\in T(\mu)\}$.
     Furthermore, each scalar $r_{\t\u}(a)$ is independent of $\s$. \end{enumerate}
     An~algebra~$A$~is~a~\textit{cellular~algebra}~if~it~has~%
a~cell~datum.~We~call~$\mathcal C$~%
a~\textit{cellular~basis}~of~$A$.
\end{definition}

The notion of {\it weakly cellular algebras} in \cite[Definition~2.9]{G} is
obtained from Definition~\ref{GL} with condition (2) replaced by: there exists an
anti-involution $*$ of $A$ satisfying
\begin{equation}\label{Weakkk} (C_{\s\t})^*\equiv
C_{\t\s} \pmod {A^{\gdom\lambda}}.\end{equation}
%instead of $(C_{\s\t})^*=C_{\t\s}$ in Definition~\ref{GL}{$\sc\,$}(2).
The results and proofs of ~\cite{GL}  are equally valid for weakly cellular algebras, so in the remainder of the paper we will not distinguish between cellular algebras and weakly cellular algebras.

We remark that  \cite[Theorem~6.3]{AMR} holds over  any commutative ring containing $1$. In this paper, we need its special case below.

\begin{theorem}\cite{AMR}\label{celldh}   The set  $\{\m_{\s\t} \mid \s, \t\in
\Std(\lambda), \lambda\in \Lambda_2^+(r)\}$ with $\m_{\s\t}$ defined in \eqref{Aff-basis} is a cellular basis of $\mathscr H_{2, r}$ over $\mathbb C$.
\end{theorem}

%We remark that we do not need this result in the current paper.

Now, we construct a weakly cellular basis of $\DBr$ over $\mathbb C$.
Fix $r, t, f\in \mathbb Z^{>0}$ with $f\le \min \{r, t\}$. We need to redefine some notations.
In contrast to
\eqref{S-f--}, we define the following subgroups of
$\mathfrak S_r$, $\mathfrak{S}_r\times\bar{\mathfrak S}_t$ and $\bar{\mathfrak S}_t$ respectively,
\begin{eqnarray}\label{1=S-f--}
&& S_{r-f}=\langle s_j\,|\, 1\le j< r-f\rangle,
%\mbox{ the subgroup of generated by $ $,}
\nonumber\\
%\mbox{ the subgroup of  generated by $ $,}
&&\mathcal {G}_f\ \ \ =\langle\bar s_{t-i} s_{r-i}\,|\,1\le i< f\rangle,
\nonumber\\
&&\bar{  S}_{t-f}=\langle\bar s_j\,|\, 1\le j< t-f\rangle
.%\mbox{ the subgroup of generated by  $ $.}
\end{eqnarray}
%
%Let $\mathcal  {G}_f$ be the subgroup of
%$\mathfrak{S}_r\times\bar{\mathfrak S}_t$ generated by $\bar s_{t-i} s_{r-i}$, $1\le i\le f-1$, and
%$ S_{r-f}$
%(resp.,   $\bar S_{t-f}{\sc\,})$ be the subgroup of $\mathfrak S_r$ $($resp.,  $\bar{\mathfrak S}_t{\sc\,})$ generated by $s_j,\, 1\le j\le r-f$ (resp.,   $\bar s_j,\, 1\le j\le t-f{\sc\,})$.
Let $\mathcal D_{r, t}^f$ be the set consisting of the following elements:
\begin{eqnarray} \label{rcs1}
c\!:=\! s_{r-f+1,i_{r-f+1}}{\sc\!} \bar s_{t-f+1, j_{t-f+1}}{\sc\!}\! \cdots\!
s_{r,i_r}{\sc\!}\bar s_{t,{j_t}}\mbox{ with }
 r  {\!}\ge{\!} i_r{\!}>{\!}{\sc\!} \cdots{\sc\!}{\!} >{\!}i_{r-f+1}
% \nonumber\\ &\!\!\!\!\!\!\!\!\!&
\mbox{ and  }{j_k}\!\le\! t\!-\!k.\end{eqnarray}
Then by arguments similar to those for Lemma~\ref{rcs},  $\mathcal D_{r, t}^f$  is   a complete set of
right coset representatives for
${S}_{r-f}\!\times\!\mathcal {G}_f\!\times\!\bar{{S}}_{t-f}$ in
$\mathfrak{S}_r\!\times\!\bar{\mathfrak{S}}_t$.
% We remark that one can check (\ref{rcs1}) .
Let
\begin{equation}\label{poset} \Lambda_{2, r,t} = \left\{ (f,(\lambda, \mu) )\,|\, (\lambda, \mu) \in \Lambda_2^+(r\!-\!f)\times \Lambda_2^+(t\!-\!f),\,  0\!\le\! f\! \le\! \min \{r,t \} \right\}.\end{equation}
\begin{definition}For $ (f, \lambda, \mu), (\ell, \alpha, \beta) \in  \Lambda_{2, r, t} $, we
define $$(f,(\lambda, \mu))\unrhd (\ell, (\alpha, \beta))\ \Longleftrightarrow\
\mbox{ either }f>\ell \mbox{ or }
f=\ell\mbox{ and }\lambda\unrhd_1 \alpha,\,\mu\unrhd_2 \beta,$$
where in case $f=\ell$, the orders $\unrhd_1$ and $\unrhd_2$  are dominance orders on $\Lambda_2^+(r\!-\!f)$ and $\Lambda_2^+(t\!-\!f)$ respectively
(cf.~(\ref{par})).
%  We write
% $(f,(\lambda, \mu))\rhd (\ell,(\alpha, \beta))$ if $(f,(\lambda, \mu))\unrhd (\ell,(\alpha, \beta))$
% and  $(f,(\lambda, \mu))\neq(\ell,(\alpha, \beta))$.
Then  $(\Lambda_{2, r,t}, \unrhd)$ is a poset.
\end{definition}

 For each $c\in \mathcal D^f_{r, t}$, as in  \eqref{rcs1}, let $\kappa_c$ be the
$r$-tuple \begin{equation}\label{k-ccc}
\kappa_c\!=\!(k_1,\dots,k_r)\!\in\!\{0,1\}^r\mbox{ such that $k_i\!=\!0$  unless $i\!=\!i_r,i_{r-1},\dots,i_{r-f+1}$.}\end{equation}
Note that $\kappa_c$ may have more than one choice for a fixed $c$, and it may
be equal to $\kappa_d$ although $c\neq d$ for $c, d\in \mathcal D^f_{r,
t}$. We set $x^{\kappa_c}=\prod_{i=1}^r x_i^{k_i}$. By Lemma~\ref{sx}, %???what is $s_c$???
\begin{equation}\label{td}
c x^{\kappa_c}\!=\!s_{r-f+1,i_{r-f+1}} x_{i_{r-f+1}}^{k_{i_{r-f+1}}}{\!}     \cdots{\sc\!}
s_{r-1,i_{r-1}} x_{i_{r-1}}^{k_{i_{r-1}}} {\sc\!}  s_{r, i_r} x_{i_r}^{k_{i_{r}}}
 \bar s_{t-f+1, j_{t-f+1}}{\!} \cdots{\sc\!}
\bar s_{t,{j_t}}
.
\end{equation}
For each $(f, \lambda)\in \Lambda_{2, r, t}$ (thus $\l$ is now a pair of bipartitions), let
\begin{equation} \label{index} \delta(f,\lambda)
  =\{(\t, c, \kappa_c)\,|\,\t\in\Std(\lambda),c\in
                        \mathcal D^f_{r, t}\text{ and }\kappa_c\in\mathbf N_f
                        \}, \end{equation}where $\mathbf N_f\! =\!\{\kappa_c \,|\, c\!\in \!\mathcal D^f_{r,
  t}\}$.~We remark that in \eqref{index},  $\lambda\!=\!(\lambda^{(1)}, \lambda^{(2)})$ with
  $\lambda^{(1)}\!\in\! \Lambda_2^+(r\!-\!f)$
  and  $\lambda^{(2)}\!\in\! \Lambda_2^+(t\!-\!f)$, and
$\t=(\t^{(1)}, \t^{(2)})$ with $\t^{(i)}$ being a $\lambda^{(i)}$-tableau for $i=1, 2$.
  In contrast to \eqref{e-f===}, we define
 \begin{equation}\label{EE-f==}
\mathfrak e^f =e_{r, t} e_{r-1, t-1} \cdots e_{r-f+1, t-f+1}\mbox{ \ if \ $f\ge1,$ \ and \ $\mathfrak e^0=1$}.\end{equation}

\begin{definition}\label{cellbasis}
For each  $(f,\lambda)\in\Lambda_{2, r, t}^+$ and
$(\s,\kappa_d,d),(\t,\kappa_c,c)\in\delta(f,\lambda)$,  we define
\begin{equation} \label{CCCCC} C_{(\s,\kappa_d,d)(\t,\kappa_c,c)}
              =x^{\kappa_d} d^{-1} \mathfrak e^{f} \m_{\s\t} c x^{\kappa_c}, \end{equation}
              where
 $ \m_{\s\t}$ is a product of   cellular basis elements for
$\mathscr H_{2, r-f}$  and $ \mathscr H_{2, t-f}$ described in Theorem~\ref{celldh}. \end{definition}

  We remark that an element in $\mathscr H_{2, r-f}$ (generated by $s_1, \cdots, s_{r-f-1}$ and $x_1$)  may not commute with an element  of  $ \mathscr H_{2, t-f}$
  (generated by $\bar s_1, \cdots, \bar s_{t-f-1}$ and $\bar x_1$).
   So, we always fix $\m_{\s\t}$ as  the product $ab$,  such that $a$ (resp., $b$) is obtained from the corresponding  cellular basis element of    $\mathscr H_{2, r-f}$
  (resp. $ \mathscr H_{2, t-f}$) described in Theorem~\ref{celldh} by using $-x_1, -p, m-q$ (resp.  $-\bar x_1, q, p-n$) instead of $Y_1, u, v$, respectively.

\let\proj=\varepsilon
\def\Ef{{\mathcal E}}
Let $I$ be the
two-sided ideal of $\DBr$ generated by~$e_1$. By Proposition~\ref{epi}{$\sc\,$}(2),   there is a
$\mathbb C$-algebraic isomorphism $\proj_{r, t}:\mathscr H_{2,r}\times \mathscr H_{2, t}\cong\DBr /I $
such that
\begin{equation}\label{epf}\proj_{r, t}(s_i)=s_i+ I,\ \proj_{r, t}(\bar s_j)=\bar s_j+ I,  \  \proj_{r, t}(x_k)=x_k+I , \  \proj_{r, t}(\bar x_\ell)=\bar x_\ell+I ,\end{equation}
for all possible $i, j, k, \ell$'s.

 For each $f$ with $ 0\le f\le \min\{r, t\}$, let $\DBr(f)$ be the two-sided ideal of $\DBr$ generated by $\mathfrak e^f$.  Then there is  a filtration of two-sided ideals
of  $\DBr$ as follows:
$$\DBr=\DBr(0)\supset \DBr(1)\supset\dotsi
     \supset \DBr(k)\supset \DBr(k+1)=0,\mbox{ where }k=\min\{r, t\}.$$

\begin{definition}\label{cellmodule-def} Suppose  $0\le
f\le\min\{r, t\}$ and   $\lambda\in \Lambda^+_2(r\!-\!f)\times\Lambda^+_2(t\!-\!f)$. Define
$\DBr^{\unrhd(f, \lambda)} $ to be the two--sided ideal of
$\DBr$ generated by   $\DBr(f+1)$ and $S$, where
$$S=\{e^{f}\m_{\s\t}|\s,\t\in \Std(\mu) \text{ and }
    \mu\in \Lambda^+_2(r\!-\!f)\times\Lambda^+_2(t\!-\!f)%\Lambda_{2, r-f, t-f}
    \text{ with
    }\mu\trianglerighteq\lambda\} .$$  We
also define $\DBr^{\rhd (f, \lambda)} =\sum_{\mu\gdom
\lambda}\DBr^{\gedom(f, \mu)}$, where
$\mu\in \Lambda^+_2(r\!-\!f)\times\Lambda^+_2(t\!-\!f).$%\Lambda_{2, r-f} \times\Lambda_{2, t-f}$
\end{definition}

 By Corollary~\ref{coro-222},  $\mathscr B_{r-f, t-f}^{p,q}$  can be embedded into $\DBr$, thus we regard it as a subalgebra of $\DBr$.

\begin{lemma}\label{if} Suppose $d\in \mathcal D^f_{r, t}$
 with  $0 \le f< \min\{r, t\}$. Then $\mathfrak e^f \langle e_1\rangle \subset \DBr(f+1)$,
 where  $\langle e_1\rangle$
 is the  two-sided ideal of $\mathscr B_{r-f, t-f}^{p,q}$ generated by $e_1$.
\end{lemma}

\begin{proof} By  assumption, we have $r-f\ge 1$ and $t-f\ge 1$.  It is easy to check  that $\mathfrak e^f $ commutes with any element in  $\mathscr B_{r-f, t-f}^{p,q}$.
Since  $e_{r-f, t-f}=$ $\bar s_{t-f, 1} s_{r-f, 1} e_1 s_{1, r-f} \bar s_{1, t-f}$, we have $\mathfrak e^f  e_1\in \DBr(f+1)$, proving the result.
  \end{proof}

For $0\le f\le\min\{r, t\}$, let $\pi_{f,
r, t}: \ \DBr(f)\rightarrow \DBr(f)/\DBr(f+1)$ be the canonical
epimorphism. Since both $\DBr(f)$ and $\DBr(f+1)$ are  $\DBr$-bimodules,
 $\pi_{f,
r, t} $ is a homomorphism as  $\DBr$-bimodules. %We remark that we set $\mathscr B(f+1)=0$ if $f=\min\{r, t\}$.
 The following result follows from (\ref{epf}) and Lemma~\ref{if}, immediately.

\begin{lemma} For each $f\in\Z^{\ge0}$ with %non-negative integer
$
f<\min\{r, t\}$, there is a
well-defined $\mathbb C$-homomorphism $\sigma_f: \mathscr H_{2, r-f} \times \mathscr H_{2, t-f}
\rightarrow \DBr(f)/\DBr(f+1)$ such that
$$\sigma_f(h)=\mathfrak e^f \proj_{r-f, t-f}(h)'+\DBr(f+1) \text{ \
for $h\in \mathscr H_{2, r-f}\times \mathscr H_{2, t-f} $, }$$ where $\proj_{r-f, t-f}(h)'$ is the preimage
of the element $\proj_{r-f, t-f}(h)\in\DBr_{r-f,t-f}^{p,q}/I$ in $\DBr_{r-f,t-f}^{p,q}$, where
$I$ is the
two-sided ideal of $\DBr_{r-f,t-f}^{p,q}$ generated by~$e_1$.\end{lemma}

\begin{lemma}\label{M_st properties} Suppose
$\lambda\in\Lambda_{2, r, t}$ and $0\le f\le \min\{r, t\}$. For any
$\s,\t\in\Std(\lambda)$,
\begin{enumerate}
\item $\mathfrak e^f \m_{\s\t}=\m_{\s\t}\mathfrak e^f\in \DBr(f) $.
\item  $\sigma_f(\m_{\s\t})=\pi_{f, r, t} (\mathfrak e^f\m_{\s\t})$.
\item $\sigma(\mathfrak e^f \m_{\s\t}) \equiv \mathfrak e^f \m_{\s\t}\pmod {\DBr^{\rhd (f, \lambda)} }$, where $\sigma$ is the anti-involution on $\DBr$ induced from that in Lemma~$\ref{antiaff}$.
\end{enumerate}
\end{lemma}
\begin{proof} By Lemma~\ref{comm1}{$\sc\,$}(1), $e_{i, j} (x_k+L_k)=(x_k+L_k) e_{j, n}$ if $i\neq k$. Furthermore, $e_{i, j} (\ell, k)=(\ell, k)e_{i, j}$
if $1\le \ell<k<i$. So, $e_{i, j} L_k=L_k e_{i, j}$, forcing $e_{i, j} x_k=x_k e_{i, j}$. Similarly,  $e_{i, j} \bar x_k=\bar x_k e_{i, j}$ for $k<j$.  So,   $\mathfrak e^f \m_{\s\t}=\m_{\s\t}\mathfrak e^f$.
 The second assertion is trivial. By Lemma~\ref{comm}{$\sc\,$}(3),  $x_i\bar x_j\equiv \bar x_j x_i \pmod {J}$, where $J$ is the two-sided ideal of $\mathscr B_{r-f, t-f}^{p, q}$ generated by $e_1$. Now, (3) follows from Lemma~\ref{if} and (1).
\end{proof}

Recall that the degree of  a monomial $\textit{\textbf m}\!\in\! \mathscr B$ in \eqref{monom} is $|\alpha|\!+\!|\beta|\!=\!\sum_{i=1}^r \alpha_i\!+\!\sum_{j=1}^t\beta_j$. So, $\mathscr B$ is
a filtered algebra, which associates to a $\Z$-graded algebra ${\rm gr}(\mathscr B)$ defined the same as in \eqref{filtr}.

The following %Proposition~\ref{tool1}
is motivated by Song and one of authors' work on $q$-walled Brauer algebras  \cite[Proposition~2.9]{RSong}.

\begin{proposition}\label{tool1} Fix $r, t, f\in \mathbb Z^{>0}$ with  $ f\le \min\{r, t\}$. Let $M_f$ be
the left  ${\mathscr{B}}_{r-f,t-f}^{p, q}$-module  generated
\begin{equation} \label{key-2} V_{r,t}^f=\{ \mathfrak e^f d x^{\kappa_d}\mid (d, \kappa_d) \in \mathcal {D}_{r,t}^f\times \mathbf N_f \}.\end{equation}
Then $M_f$ is a right
$\DBr$-module.
\end{proposition}

\begin{proof} We prove the result by induction on the degree of   $\mathfrak e^f d x^{\kappa_d}$. If the degree is $0$, then  $\mathfrak e^f d x^{\kappa_d}= \mathfrak e^f d$.
By the result on the walled Brauer algebra (which is the special case of \cite[Proposition~2.9]{RSong}), we have $\mathfrak e^f d h\in M_f$ for any $h\in \mathscr B_{r, t}(\omega_0)$.
Note that $\mathscr B_{r, t}(\omega_0)$ is a subalgebra of $\DBr$.

Now, we consider $\mathfrak e^f d x_1$, where $d$ has the form in (\ref{rcs1}). If $i_j=1$ for some $j\ge r-f+1$, then $j=r-f+1$ and   $\mathfrak e^f d x_1\in V_{r, t}^f $. Otherwise, we have $(1)d=1$, and
 $d x_{1}=x_{1} d$. Note that $r-f+1\ge i_{r-f+1}>1$, we have
 $\mathfrak e^f x_{1} d=x_1 \mathfrak e^f d\in M_f$.

 We have $\mathfrak e^f d \bar x_1=\mathfrak e^f  \bar x_k d+\mathfrak e^f  w $  for some $k, 1\le k\le t$ and some $w\in {\mathbb C}\mathfrak S_r\times \mathbb C\bar {\mathfrak S}_t$. By corresponding
 result for walled Brauer algebras, we have  $\mathfrak e^f  w\in M_f$. If $k\le r-f$, by  Lemma~\ref{comm1}, $\mathfrak e^f  \bar x_k d$ can be replaced by  $\bar x_k\mathfrak e^f d\in M_f$.
If $k\ge r-f+1$, by Lemma~\ref{comm}$\sc\,$(2), we can use $x_k$ instead of $\bar x_k$ in $ \mathfrak e^f  \bar x_k d$. So, the required result follows from our previous arguments on $s_i, \bar s_j$ and $x_1$.
This completes the proof
when the degree of $\mathfrak e^f d x^{\kappa}$ is $0$.

Suppose the degree of  $\mathfrak e^f d x^{\kappa_d}$ is not $0$. We want to prove    $\mathfrak e^f d x^{\kappa_d} h\in M_f$ for any generators $h$ of $\DBr$.

\subcase{1: $h\in \bar{\mathfrak S}_t$} We have  $x^{\kappa_d} h=h x^{\kappa_d}$. By our pervious result on degree $0$, we have
$\mathfrak e^f d h\in M_f$. Therefore, we need to check $\mathfrak e^f (d h) x^{\kappa_d}\in M_f$. If $x_j$ is a term of $x^{\kappa_d}$, by induction on the degree, we have $\mathfrak e^f dh x\in M_f$, where
$x$ is obtained from $x^{\kappa_d}$ by removing the factor $x_j$.
So, $\mathfrak e^f dh x^{\kappa_d}\in M_f$ by induction assumption on $\text{deg} (\mathfrak {e}^f dh x^{\kappa_d})-1$.

\subcase {2: $h\in \mathfrak S_r$} We have $x^{\kappa_d} h=h x$ in  $\rm{gr} (\DBr)$, where $x$ is obtained from $x^{\kappa_d}$ by permuting some indices. By induction assumption,  it suffices to verify
$\mathfrak e^f d h  x\in $ $M_f$ with $\text{deg}(x)=\text{deg}(x^{\kappa_d})$. This has already been verified in Case~1.

\subcase{3: $h=x_1$} If   $x_1$ is a factor of $x^{\kappa_d}$,  we have $x_1^2=(p+q-m)x_1-p(q-m)$ (cf.~\eqref{x1-barx1}). So,  $\mathfrak e^f d x^{\kappa_d} x_1\in M_f$ by induction assumption on   $\text{deg} (\mathfrak {e}^f d x^{\kappa_d})-1$. If $x_1$ is not a factor of $x^{\kappa_d}$, and if $i_{r-f+1}=1$,  where $d$ has the form
 in \eqref{rcs1},
%= s_{r-f+1,i_{r-f+1}} \bar s_{t-f+1, j_{t-f+1}} \cdots
%s_{r,i_r}\bar s_{t,{j_t}},$$
then there is nothing to be proven.
Otherwise,  $i_{r-f+1}>1$ and   $\mathfrak e^f d x^{\kappa_d} x_1=x_1\mathfrak e^f d x^{\kappa_d}\in M_f$.

\subcase {4: $h=\bar x_1$} By Lemma~\ref{comm}{$\sc\,$}(1), $ x^{\kappa_d} \bar x_1=\bar x_1 x^{\kappa_d}$ in  $\rm{gr} (\DBr)$. So, the result follows from induction assumption on degree and our previous results in Cases~2--3.

\subcase{5: $h= e_1$} We can assume $x^{\kappa_d}=x_1$. Otherwise,  the result follows from Lemma~\ref{comm1},  induction assumption and our previous results in Cases~1--3.
    In this case, $i_{r-f+1}=1$ and $e_{r-f+1, t-f+1} d=de_{1,j}$ for some $j,\, 1\le j\le t$. So, $\mathfrak e^f d x_1 e_1=\mathfrak e^{f-1} d e_{1, j} x_1 e_1$.

     If $j=1$,
     the required result follows from  the equality $e_1 x_1e_1=(nq-mp) e_1$ (cf.~\eqref{www-}).
      Otherwise, by Lemmas~\ref{comm} and \ref{comm1},
$$e_{1, j} x_1 e_1=-\bar x_{1} e_{1, j}e_1=- e_{1, j}\bar x_{1}(\bar 1, \bar {j}).$$
So, we need to verify $\mathfrak e^{f-1} d e_{1, j}\bar x_{1}(\bar 1, \bar {j})=\mathfrak e^f d  \bar x_{1}(\bar 1, \bar{j})\in M_f$, which follows from  our previous results in Cases~1, 2 and~4.
 This completes the proof of Proposition~\ref{tool1}.
\end{proof}

\begin{proposition}\label{cell-m} Suppose  $(f, \lambda)\in \Lambda_{2, r, t}$. Then $\Delta^R(f,\lambda)$ $($resp.,  $ \Delta^L(f,\lambda){\sc\,})$ is a right $($resp., left$)$
$\DBr$-module, where
\begin{itemize} \item   $\Delta^R(f,\lambda)$ is $\mathbb C$-spanned by $\{\mathfrak e^f\m_{\t^\lambda \s} d x^{\kappa_d} +\DBr^{\rhd(f,\lambda)}\,|\,(\s, d, \kappa_d) \in \delta(f, \lambda)$, and
\item   $\Delta^L(f,\lambda)$  is  $\mathbb C$-spanned by   $\{d^{-1} \mathfrak e^f\m_{\s\t^\lambda } +
\DBr^{\rhd(f,\lambda)}\,|\, (\s, d, \kappa_d) \in \delta(f, \lambda)\}$.
\end{itemize}
\end{proposition}
\begin{proof} We remark that $x_i \bar x_j=\bar x_j x_i$ in $\Delta^R(f,\lambda)$ $($resp.,  $ \Delta^L(f,\lambda){\sc\,})$ for all possible $i, j$'s. So,
the result follows from Proposition~\ref{tool1} and Theorem~\ref{celldh} on the cellular basis of level two degenerate Hecke algebras $\mathscr H_{2, r-f}\times \mathscr H_{2, t-f}$.
 \end{proof}

\begin{theorem}\label{cellular-1}
Let $m,n,r,t\!\in\!\Z^{\ge1},\,p,q\!\in\!\C$ such that $r\!+\!t\!\le\!\min\{m,n\}$.~The~set
$$\mathscr C=\{C_{(\s,\kappa_c,c)(\t,\kappa_d,d)}\,|\,
            (\s,\kappa_c,c),(\t,\kappa_d,d)\in\delta(f,\lambda),
              \forall (f,\lambda)\in\Lambda_{2, r, t}\},$$
%              Then $\mathscr C$
is a weakly  cellular basis $\DBr=\DBr_{r,t}^{p,q}(m,n)$ over  $\mathbb C$.
\end{theorem}

\begin{proof}  Suppose $0\le f\le
\min\{r, t\}$. By Proposition~\ref{cell-m},  $\DBr(f)
/\DBr(f+1)$  is spanned by $C_{(\s, c, \kappa_c)(\t, d, \kappa_d)}+
\DBr(f+1)$ for all  $(\s, c, \kappa_c), (\t, d, \kappa_d) \in \delta(f,
\lambda)$  and $ \lambda\in \Lambda_{2, r-f, t-f}$. So, $\DBr$ is spanned by $\mathcal C$.  Counting the cardinality of $\mathcal C$ yields
$|\mathcal C|=2^{r+t} (r+t)!$, which is the dimension of $\DBr$, stated in Corollary~\ref{coro-222}. So, $\mathcal C$ is a $\mathbb C$-basis of $\DBr$. By Lemma~\ref{M_st properties}{$\sc\,$}(3) and  Proposition~\ref{cell-m},
it is a weakly cellular basis in the sense of \eqref{Weakkk}.
\end{proof}

\begin{remark} If we consider level two walled Brauer algebras over  a commutative ring containing $1$,
 and if we know its rank is equal to $2^{r+t} (r+t)!$,  then all results in this section hold. We hope to prove this result elsewhere.\end{remark}

\section{ Irreducible modules for $\DBr$}\label{decmpw}
 In this section, we   classify    irreducible $\DBr$-modules
over  $\mathbb C$ via Theorem~\ref{cellular-1}.  So, we assume
$r+t\le \min\{m, n\}$.

 First, we   briefly recall the representation theory
of cellular algebras~\cite{GL}.  At moment, we keep the notations in
Definition~\ref{GL}. So, $R$ is a commutative ring $R$ containing $1$ and  $A$ is a (weakly) cellular algebra over $R$
with a weakly cellular basis $\{C_{\s\t}\,|\,\s,\t\in
T(\lambda), \lambda\in\Lambda\}$. We consider
the right $A$-module in this section.

Recall that each cell module $C(\lambda)$ of $A$ is the free $R$-module with basis
$\{C_\bfs\,|\, \bfs\in T(\lambda)\}$, and every irreducible
$A$-module arises in a unique way as the simple
head of some cell module~\cite{GL}.  More explicitly, each
$C(\lambda)$ comes equipped with the invariant form
$\phi_{\lambda}$ which is determined by the equation
$$C_{\bfs\bft}
           C_{\bft'\bfs}
 \equiv\phi_{\lambda}\big(C_{\bft},
               C_{\bft'}\big)\cdot
        C_{\bfs\bfs}\pmod{ A^{\rhd \lambda}}.$$
 Consequently,
$$\text{Rad\,} C(\lambda)=\{x\in C(\lambda)\,|\,\phi_{\lambda}(x,y)=0\text{ for all }
                          y\in C(\lambda)\},$$
is an $A$-submodule of $C(\lambda)$ and $D^{\lambda}=C(\lambda)/\text{Rad\,}
C(\lambda)$ is either zero or absolutely irreducible. Graham and
Lehrer \cite{GL} proved the following result.

\begin{theorem} \cite{GL}\label{cell-tool} Let $(A, \Lambda)$ be a $($weakly$)$ cellular algebra over a field $F$.
  The set  $\{D^\lambda\,|\, D^\lambda\neq 0, \lambda\in \Lambda\}$  consists of a
complete set of pairwise non-isomorphic irreducible $A$-modules.
 \end{theorem}

By  Theorem~\ref{cellular-1}, we have  cell
modules  $C(f, \lambda)$ with $(f, \lambda)\in \Lambda_{2, r, t}$ for $\DBr$. In fact, it is $\Delta^R(f, \lambda)$ in Proposition~\ref{cell-m} up to an isomorphism.  Let    $\phi_{f,
\lambda}$ be  the corresponding invariant form on $C(f, \lambda)$. We use Theorem~\ref{cell-tool}
 to classify the irreducible $\DBr$-module over $\mathbb C$.

Let  $ \mathscr H_{2, r-f}$ (resp., $\mathscr H_{2, t-f}$)
 be the level two Hecke algebra which is isomorphic to
the  subalgebra of $\mathscr B_{r-f, t-f}^{p, q}$
generated by $s_1, s_2, \cdots, s_{r-f-1}$ and $x_1$ (resp., $\bar s_1, \bar s_2, \cdots, \bar s_{t-f-1}$ and $\bar x_1$).
So, the eigenvalues of $x_1$ (resp., $\bar x_1$) are given in (\ref{x1-barx1}).
 By Theorem~\ref{celldh},
   $$\{\m_{\s\t}\,|\,\s,
\t\in\Std(\lambda),\,\lambda\in\Lambda_2^+( r-f)\times \Lambda_2^+(t-f) \}$$ is a cellular basis of
$\mathscr H_{2, r-f}\times \mathscr H_{2, t-f}$.
We remark that $\m_{\s\t}$ is a product of cellular basis elements of $\mathscr H_{2, r-f}$ and $\mathscr H_{2, t-f}$ described in  Theorem~\ref{celldh}.

 Let $C(\lambda)$ be the cell module  with respect to  $\lambda\in \Lambda_{2, r-f}\times \Lambda_{2, t-f}$ for
$\mathscr H_{2, r-f}\times \mathscr H_{2, t-f}$.
 Let $\phi_{\lambda}$ be the invariant form on   $C(\lambda)$. For simplicity, we use $\mathbf H(2, f)$ to denote  $\mathscr
H_{2, r-f}\times \mathscr H_{2, t-f}$.

\begin{proposition} \label{aebe} Suppose $r, t\in \mathbb Z^{\ge 2}$.  We have  $e_{r, t} \mathscr B_{r, t}^{p, q} %(m, n)
  e_{r, t} \subseteq  e_{r, t} \mathscr B_{r-1, t-1}^{p, q}%(m, n)
  $.
\end{proposition}
\begin{proof} Recall that $\mathscr B_{r, t}^{p, q}%(m, n)
$ is a (weakly) cellular algebra with cellular basis given in Theorem~\ref{cellular-1}. So, it suffices to  verify
\begin{equation}\label{veriaa} e_{r, t} C_{(\s,\kappa_d,d)(\t,\kappa_c,c)} e_{r, t} \in e_{r, t} \mathscr B_{r-1, t-1}^{p, q}
%(m, n)
,\end{equation} where
  $C_{(\s,\kappa_d,d)(\t,\kappa_c,c)}
              =x^{\kappa_d} d^{-1} \mathfrak e^{f} \m_{\s\t} c x^{\kappa_c}$ (cf.~\eqref{CCCCC}).

Let $f\!=\!0$.~Since $\m_{\s\t} $ is a combination of monomials of form
$\prod_{i=1}^r x_i^{\alpha_i} w_1 \bar w_1 \prod_{i=1}^t \bar x_i^{\beta_i}$, it suffices to  verify
\begin{equation}\label{veribbb}\mbox{$ e_{r, t} \prod\limits_{i=1}^r x_i^{\alpha_i} w_1 \bar w_1
\prod\limits_{i=1}^t \bar x_i^{\beta_i} e_{r, t}\in e_{r, t}\mathscr B_{r-1, t-1}^{p, q}$}
%(m, n)
\end{equation}
We prove  (\ref{veribbb}) by induction on the degree of $\prod_{i=1}^r x_i^{\alpha_i} w_1 \bar w_1
\prod_{i=1}^t \bar x_i^{\beta_i}$. The case for degree $0$ follows from  \cite[Proposition~2.1]{CVDM}. In general, we  assume that $\alpha_i=0$ for $1\le i\le r-1$ and $\beta_j=0$ for $1\le j\le t-1$.
Otherwise, (\ref{veribbb}) follows from induction assumption and  the equalities $e_{r, t} x_i=x_i e_{r, t}$ and $e_{r, t} \bar  x_j=\bar x_j e_{r, t}$
for $i\neq r$ and $j\neq t$.

By symmetry, we assume
$\alpha_r=1$. Write $w_1=s_{r, k} w_2$ for some $k, 1\le k\le r$ and some $w_2\in \mathfrak S_{r-1}$. Since
  any element in $\mathfrak S_{r-1}$ commutes with $\bar s_j\in\bar{\mathfrak S}_t,\, \bar x_t$ and $e_{r, t}$, we can assume
$w_1\in \{1,s_{r-1}\}$.

If $w_1=1$, by Lemma~\ref{comm}(2), it suffices to verify  $e_{r, t} \bar x_t \bar w_1 \bar x_t^{\beta_t} e_{r, t}\in e_{r, t} \mathscr B_{r-1, t-1}^{p, q}$.
In this case, we have $ \bar x_t \bar w_1=\bar w_1 \bar x_k+ h$ for some $h\in \mathbb C \bar {\mathfrak S}_t$ and some $k$ with $(\bar t)\bar w_1= \bar k$.
 By induction on  degree, we need to verify   \begin{equation} \label{dd} e_{r, t} \bar w_1 \bar x_k \bar x_t^{\beta_t} e_{r, t}\in  e_{r, t} \mathscr B_{r-1, t-1}^{p, q}.\end{equation}
  By Lemma~\ref{comm1}(2), and induction assumption on degree,
 we have (\ref{dd}) if $k\neq t$. Otherwise, we have $k=t$ and $\bar w_1\in \bar{\mathfrak S}_{t-1}$. So,   $e_{r, t} \bar w_1=\bar w_1 e_{r, t}$. By induction on degree, we use $(\bar x_t +\bar L_t)^{1+\beta_t}$ instead of
 $ \bar x_t\bar x_t^{\beta_t}$ in  $e_{r, t} \bar x_t\bar x_t^{\beta_t} e_{r, t}$. So, the result follows from Lemma~\ref{ex}(2). This verifies (\ref{veribbb}) provided  $f=0$.

Suppose $f>0$. By (\ref{CCCCC}), we rewrite  (\ref{veriaa}) as follows:
\begin{equation}\label{ddd}\mbox{$ e_{r, t}x^{\kappa_d} d^{-1} \mathfrak e^{f} \m_{\s\t} c x^{\kappa_c} e_{r, t}\in e_{r, t}\mathscr B_{r-1, t-1}^{p, q}$}.\end{equation}
Applying our previous result for $f=0$ on  $e_{r, t}x^{\kappa_d} d^{-1} \mathfrak e^{f}$ and $\mathfrak e^{f} \m_{\s\t} c x^{\kappa_c}  e_{r, t}$ and
noting that $\mathfrak e^f=e_{r, t} \cdots e_{r-f+1, t-f+1}$, we have (\ref{ddd}) as required.
\end{proof}

By recalling the definitions of $\omega_0$ and $\omega_1$ in (\ref{www-}), we see that   $\omega_0=\omega_1=0$ if and only if $m=n$ and $p=q$.
 For any $\lambda\in \Lambda_{2}^+(r-f)\times \Lambda_{2}^+(t-f)$ and $\t\in \Std(\lambda)$, we define   $\m_\t=\m_{\t^\lambda\t}+\mathbf H(2, f)^{\rhd \lambda}$, where
 $\mathbf H(2, f)$ is given above Proposition~\ref{aebe}.

\begin{lemma}\label{bilequal} Let $\DBr$ be the level two walled Brauer algebra defined over $\mathbb C$.  Suppose  $(f,\lambda)\in\Lambda_{2, r,t } $
and $f\neq r$ if $r=t$.
 Then $\phi_{f,\lambda}\neq 0$ if and only if $\phi_{\lambda}\neq 0$.
\end{lemma}

\begin{proof}  If  $r\neq t$ or if $r=t$ and $f\neq r$, then  either $s_{r,r-f}$ or  $\bar s_{t,t-f}$ is well-defined. We denote such an element  by $w$.
So, $\mathfrak e^f w \mathfrak e^f=\mathfrak e^f$.

If  $\phi_{\lambda}\neq 0$, then  $\phi_{\lambda} (\m_\s, \m_\t)\neq 0$
 for some  $\s, \t\in \Std(\lambda)$.  We have $\phi_{f,\lambda}\neq 0$ since
 $$\m_{\t^\lambda \s}\mathfrak e^f w \mathfrak e^f \m_{\t \t^\lambda}  \equiv\phi_{\lambda}(\m_{\s}, \m_{\t
})\mathfrak e^f \m_{\t^\lambda\t^\lambda} \pmod{
\DBr^{\rhd(f,\lambda)}}. $$
We remark that $ x_i\bar x_j=\bar x_jx_i$ in $C(f, \lambda)$ for $1\le i\le r-f$ and $1\le j\le t-f$ (cf. Lemma~\ref{if}).

Conversely, if $\phi_{f,\lambda}\neq 0$, then  $$\phi_{f, \lambda}(\mathfrak e^f\m_{\s}
{d} x^{\kappa_d}, \mathfrak e^f \m_{\t} {c} x^{\kappa_c} )\neq 0,$$ for some $(\s, d, \kappa_d ), (\t, c, \kappa_c)\in \delta(f,
\lambda)$. We have $\phi_{\lambda}\neq 0$. Otherwise,
$$\m_{\t^\lambda \s} h \m_{\t\t^\lambda}\equiv 0 \pmod { \mathbf
H(2, f)^{\rhd \lambda}},$$ for all $h\in \mathbf H(2, f)$.
Using Proposition~\ref{aebe}   repeatedly, we have
$$\mathfrak e^f\m_{\t^\lambda \s} c d^{-1}e^f \m_{\t\t^\lambda}\equiv
\m_{\t^\lambda \s} h \m_{\t\t^\lambda} \mathfrak e^f \pmod {\DBr (f+1)}, $$ for some $h\in \mathbf H(2, f)$,  forcing   $\phi_{f, \lambda}(\mathfrak e^f\m_{\s}
{d} x^{\kappa_d}, \mathfrak e^f \m_{\t} {c} x^{\kappa_c} )=0$, a contradiction.
\end{proof}
\begin{lemma}\label{bilequa2} Let $\DBr$ be the level two walled Brauer algebra defined over $\mathbb C$ with $r=t$. Then
  $\phi_{r, 0}\neq  0$ if at least one of $\omega_0$ and  $\omega_1$ is non-zero. Otherwise, $\phi_{r, 0}=0$.
\end{lemma}

\begin{proof} Suppose $\omega_0\neq 0$. We have $\phi_{r, 0}\neq 0$ since
$\mathfrak e^f \mathfrak e^f =\omega_0^f \mathfrak e^f$.
Otherwise,  $\omega_0=$ $m-n=0$, forcing  $m=n$ and  $\omega_1=n(q-p)$.
We consider $e_r e_{r-1}\cdots e_1 \prod_{i=1}^r (x_i+L_i)$ (where the product is in any order) and  $e_r e_{r-1}\cdots e_1 $ in the cell module $C(r, 0)$.
By Lemma~\ref{ex}(2),  $$e_r e_{r-1}\cdots e_1 \mbox{$\prod\limits_{i=1}^r$} (x_i+L_i) e_r e_{r-1}\cdot e_1 = \omega_1^r e_re_{r-1}\cdots e_1. $$
We have $\phi_{r, 0}\neq 0$ if  $\omega_1\neq 0$.

Finally, we assume $\omega_0=\omega_1=0$ and $r=t$. In this case, we have $m=n$ and $p=q$. We claim that $\phi_{r, 0}=0$.

In fact, for any two basis elements $\mathfrak e^r c x^{\kappa_c}$ and
$\mathfrak e^r d x^{\kappa_d}$ in $C(r, 0)$  with $c, d\in \mathcal D_{r, r}^r$, by
using Proposition~\ref{aebe} repeatedly, we have  \begin{equation} \label{101} \mathfrak e^r c  x^{\kappa_c} x^{\kappa_d} d^{-1}
\mathfrak e^r \in \mathfrak e^{r-1}  e_1 \mathscr B_{1, 1}^{p, q}
%(m, n)
e_1.\end{equation}
However, since we are assuming that $\omega_0=\omega_1=0$, it is routine to check  $e_1 \mathscr B_{1, 1}^{p, q}
 e_1=0$. So,
  $\phi_{r, 0}=0$, as required.
\end{proof}

 In \cite{K},  Kleshchev classified the irreducible modules for degenerate cyclotomic Hecke algebra $\mathscr H_{r, n}$ over an arbitrary field via Kleshchev multipartitions of $n$.
As mentioned in \cite[page~130]{AMR},  one could not say that $\phi_\lambda\neq 0$ if and only if $\lambda$ is a Kleshchev multipartition. In our case, since level two walled Brauer algebras are only related to
 representations of cyclotomic degenerate Hecke algebras (more explicitly, level two Hecke algebras) over $\mathbb C$, we can use Vazirani's result \cite[Theorem~3.4]{Va}.
  In fact,
 it is not difficult to prove that there is an epimorphism
from a standard module in  \cite[Theorem~3.4]{Va} to a cell module for degenerate cyclotomic Hecke algebra. So, $\phi_\lambda\neq 0$ if and only if $\lambda$ is {Kleshchev} in the sense of \cite[p273]{Va}. We recall the definition as follows. For our purpose, we only consider bipartitions.

\begin{definition} Fix $u_1, u_2\!\in\! \mathbb C$ with $u_1\!-\!u_2\!\in\!\mathbb Z$.  A bipartition  $\lambda\!= $ $(\lambda^{(1)}, \lambda^{(2)})\!\in\! \Lambda^+_2(n)$ of $n$  is called a {\it Kleshchev bipartition} with respect to $u_1, u_2$ if
$$
\lambda^{(1)}_{u_1-u_2+i}\le \lambda^{(2)}_{i} \text{ for all possible $i$.}$$
If $u_1-u_2\not\in \mathbb Z$, then we say that  all bipartitions of $n$ are Kleshchev bipartitions\vspace*{-4pt}.
\end{definition}
Since we consider a pair of bipartitions $(\lambda^{(1)}, \lambda^{(2)})$, where $\lambda^{(1)}\in \Lambda_2^+(r-f)$ and $\lambda^{(2)}\in \Lambda_2^+(t-f)$ for all $f,\, 0\le f\le \min\{r, t\}$,
we say that $\lambda$ is {\it Kleshchev} if both $\lambda^{(1)}$ and $\lambda^{(2)}$ are Kleshchev with respect to the parameters $u_1=-p,\,u_2=m-q$ and $u_1=q,\,u_2=p-n$ respectively. The following result follows
from Lemmas~\ref{bilequal}, \ref{bilequa2} and our previous arguments immediately\vspace*{-4pt}.

\begin{theorem}\label{main-3}
Let $\DBr=\DBr_{r,t}^{p,q}(m,n)$ be the level two  walled Brauer algebra over $\mathbb C$ with condition $r+t\le \min\{m, n\}$.
\begin{enumerate} \item   Suppose either $r\neq t$ or $r=t$ and one of $\omega_0, \omega_1$ is non-zero.
Then   the set of pairwise non-isomorphic irreducible $\DBr$-modules are indexed by
 $\{(f, \lambda)\!\in\! \Lambda_{2, r, t}\,|\, 0\!\le \!f\!\le\! \min\{r, t\},\,  \lambda \text{ being Kleshchev\,}\}$.
 \item If $r= t$ and $\omega_0=\omega_1=0$, then  the set of pairwise non-isomorphic irreducible $\DBr$-modules are indexed by
 $\{(f, \lambda)\!\in \!\Lambda_{2, r, t}\,|\,  0\!\le\! f\!<\!r,\,  \lambda \text{ being Kleshchev}\,\}$.\end{enumerate}
 \end{theorem}

We close the paper by giving a classification of non-isomorphic indecomposable direct summands of
${\frak{gl}}_{m|n}$-modules $M^{rt}$ (cf.~\eqref{M-st==}) provided that $M=K_\lambda$ is typical. Such
%indecomposable
direct summands are called  {\it indecomposable tilting
modules} of
$\mathfrak {gl}_{m|n}$.

%\vspace*{-4pt}.

\begin{theorem} \label{level-3}
Assume $r\!+\!t\!\le\!\min\{m,n\}$. \begin{enumerate}\item
If $p\!-\!q\!\in\! \mathbb Z$ with either $p\!-\!q\!\le\!-m$ or $p\!-\!q\!\ge\!n$, then  $M^{rt}$ $($cf. \eqref{M-st==}$)$  is a tilting module and  the non-isomorphic indecomposable direct summands of $ M^{rt} $ are indexed by
 $\{(f, \mu)\!\in\! \Lambda_{2, r, t}\,|\, 0\!\le\! f\!\le\! \min\{r, t\},\,  \mu \text{ being Kleshchev}\,\}.$
 \item  If $p\!-\!q\!\notin\!\Z$, then
  the non-isomorphic indecomposable direct summands of $ M^{rt} $ are irreducible and indexed by $\Lambda_{2, r, t}$.
  \end{enumerate}
  \end{theorem}
\begin{proof} Under the assumptions in (1) and (2), the Kac module $K_{\lambda_{pq}}$ is typical and at least one of $\omega_0$ and $\omega_1$ is non-zero. In this case,  $M^{rt}$ is a tilting module (see, e.g.,~\cite[IV]{BS4} for the case $t=0$, from which one sees that it holds in general). By Theorem~\ref{level-2}, there is  a bijection between the set of  non-isomorphic
indecomposable direct summands of $M^{rt}$ and the irreducible modules of $\DBr$. So, (1)--(2) follows from Theorem~\ref{main-3}${\sc\,}$(1).
In particular, if $p\!-\!q\!\notin\! \mathbb Z$, all partitions $\lambda\in \Lambda_2^+(r-f)\times \Lambda_2^+(t-f)$ are Kleshchev. We remark that (2) also
follows from  Corollary~\ref{coro-222} and Graham-Lehrer's result in \cite{GL}, which says that a cellular algebra is semisimple if and only if each cell module  is equal to its simple head.
\end{proof}


\small\begin{thebibliography}{DWH99}
\bibitem{AMR}{\scshape S.~Ariki, A.~Mathas {\normalfont \smfandname}  H. Rui},
{\og {Cyclotomic Nazarov-Wenzl algebras}\fg}, \emph{Nagoya Math.
J.}, Special issue in honor of Prof. G. Lusztig's sixty birthday,
  \textbf{182} (2006), 47--134.


\bibitem{BCHLLJ} {\scshape G.~Benkart, M. Chakrabarti, T. Halverson, R. Leduc, C. Lee {\normalfont \smfandname}   J. Stroomer}, {\og Tensor product representations of
general linear groups and their connections with Brauer algebras\fg}, \emph {J. Algebra} \textbf{166} (1994) 529--567.

\bibitem{B}
{\scshape R.~Brauer}, {\og On algebras which are connected with the
semisimple
  continuous groups\fg}, \emph{Ann. of Math.} \textbf{38} (1937), 857--872.

\bibitem{BS4} {\scshape J. Brundan  {\normalfont \smfandname} C. Stroppel}, {\og Highest weight categories arising from Khovanov's diagram algebra
I, II, III, IV\fg} \emph{Moscow Math. J.} \textbf {11}, (2011), 685--722; \emph{Transform. Groups} \textbf{15}, (2010),  1--45; \emph {Represent.
Theory} \textbf{15}, (2011),  170--243; \emph{J. Eur. Math. Soc.} \textbf{14},  (2012), 373--419.

\bibitem {BS} {\scshape J. Brundan   {\normalfont \smfandname} C. Stroppel}, {\og  Gradings on walled Brauer algebras and Khovanov's arc algebra}, \emph{ Adv. Math.} \textbf {231} (2012), no. 2, 709--773.

\bibitem{CPS} {\scshape E. Cline, B. Parshall {\normalfont \smfandname}  L. Scott},
{\og Finite-dimensional algebras and highest weight categories\fg}, \emph{J. Reine Angew. Math.} \textbf{391} (1988), 85--99.

\bibitem{CD}{\scshape A. ~Cox {\normalfont \smfandname}  M. De Visscher}, {\og  Diagrammatic Kazhdan-Lusztig theory for the (walled) Brauer algebra\fg},
\emph{Journal of Algebra} \textbf{340}, (2011), 151--181.


\bibitem{CVDM}  {\scshape A. ~Cox,  M. De Visscher, S. Doty {\normalfont \smfandname}  P. Martin}, {\og  On the blocks of the walled Brauer algebra\fg}, \emph{J. Algebra}
\textbf { 320} (2008), no. 1, 169--212.


\bibitem{DD} {\scshape R. Dipper {\normalfont \smfandname}   S. Doty} {\og The rational Schur algebra\fg}, \emph{ Represent. Theory } \textbf{12} (2008), 58--82.

\bibitem{enyang} {\scshape J.~Enyang}, {\og  Cellular bases of the two-parameter version of the centraliser algebra for the mixed tensor representations of the quantum general linear group\fg},
 \emph{Surikaisekikenkyusho Kokyuroku} \textbf {1310},  (2003), 134--153.


 \bibitem{G}{\scshape F.~M. Goodman},
{\og  Cellurality of cyclotomic  Birman-Wenzl-Murakami Algebras\fg},
 \emph{J. Algebra} \textbf{321} (2009),
  3299--3320.



\bibitem{GL}
{\scshape J.~J. Graham {\normalfont \smfandname} G.~I. Lehrer}, {\og
Cellular  algebras\fg}, \emph{Invent. Math.} \textbf{123} (1996), 1--34.

\bibitem{Gr} {\scshape  J.A. Green}, {\og Polynomial Representations of $GL_n$\fg}, Second corrected and augmented edition. With an appendix on Schensted
correspondence and Littelmann paths by K. Erdmann, Green and M. Schocker, Lecture Notes in Math., Vol. \textbf{830}, Springer, Berlin, 2007.

\bibitem{Ha}{\scshape T. Halverson}, {\og Characters of the centralizer algebras of mixed tensor representations of $GL(r,C)$ and
the quantum group $U_q(gl(r,C))$\fg}, \emph{ Pacific. J. Math.} \textbf{174} (1996), 359--410.

\bibitem{JK}  {\scshape J.H. Jung {\normalfont \smfandname} S.-J. Kang}, {\og
Mixed Schur-Weyl-Sergeev duality for queer Lie superalgebras\fg}, arXiv:1208.5139.

\bibitem{Kac77} {\scshape V.G.~Kac}, {\og Lie superalgebras\fg}, \emph{Adv.~Math.} \textbf{26} (1977), 8--96.








\bibitem{K} {\scshape A. Kleshchev}, {\og Linear and projective representations of symmetric groups\fg}, { Cambridge Tracts in Mathematics, 163. Cambridge University Press, Cambridge}, 2005.

\bibitem{Koi} {\scshape K. Koike}, {\og On the decomposition of tensor products of the representations of classical groups: By means of
universal characters\fg}, \emph{Adv. Math.} \textbf {74} (1989) 57--86.

\bibitem{N} {\scshape P. Nikitin}, {\og  The centralizer algebra of the diagonal action of the group GLn(C) in
a mixed tensor space\fg}, \emph{J. Math. Sci.} \textbf {141} (2007), 1479--1493.

\bibitem{SM} {\scshape C. Lee Shader    {\normalfont \smfandname}  D. Moon}, {\og Mixed tensor representations and representations for the general linear
Lie superalgebras\fg}, \emph{Comm. Algebra} \textbf{30} (2002), 839--857.


\bibitem{Na}{\scshape M.~Nazarov}, {\og Young's orthogonal form for {B}rauer's
centralizer algebra\fg}, \emph{J. Algebra} \textbf{182} (1996),
664--693.
\bibitem{RSong}{\scshape H. Rui {\normalfont \smfandname} L. Song}, {\og  The
quantized walled Brauer algebras\fg}, preprint, 2011.



\bibitem{SHK} {\scshape Y.~Su, J.W.B.~Hughes {\normalfont \smfandname}
R.C.~King}, {\og Primitive vectors in the Kac-module of the Lie
superalgebra $sl(m|n)$\fg}, \emph{J.~Math.~Phys.} \textbf{41} (2000), 5044--5087.

\bibitem{SZ4}	{\scshape Y.~Su {\normalfont \smfandname}  R.B. Zhang}, {\og Generalised Verma modules for the orthosymplectic Lie superalgebra ${\mathfrak{osp}}_{k|2}$\fg}, \emph{J.~Algebra} \textbf{357} (2012), 94--115.


\bibitem {Tur} {\scshape V.~Turaev}, {\og  Operator invariants of tangles and R-matrices\fg}, \emph{Izv. Akad. Nauk SSSR Ser. Math.} \textbf{53} (1989) 1073--1107 (in Russian).

\bibitem{Va} {\scshape M. Vazirani}, {\og Parameterizing Hecke algebra modules: Bernstein-Zelevinsky multisegments, Kleshchev multipartitions,
and crystal graphs\fg }, \emph{Transformation Groups} \textbf{7} (2002), no.~3, 267--303\vspace*{-7pt}.
\end{thebibliography}
\end{document}